\documentclass[11pt]{amsart}  	

\usepackage{geometry}  

\usepackage[all]{xy}						
\usepackage{lipsum}
\usepackage{bbm}
\usepackage{amssymb}
\usepackage{amscd,latexsym,amsthm,amsfonts,amssymb,amsmath,amsxtra}
\usepackage[mathscr]{eucal}
\usepackage[colorlinks]{hyperref}
\usepackage{mathtools}
\usepackage{cite}
\usepackage{chngcntr}
\numberwithin{equation}{subsection}
\newcounter{keepeqno}
\newenvironment{num}{\setcounter{keepeqno}{\value{equation}}
\begin{list}
{(\theequation)}{\usecounter{equation}}%
\setcounter{equation}{\value{keepeqno}}}{\end{list}}

\hypersetup{linkcolor=black, citecolor=black}

\newcommand{\alg}{\mathrm{alg}}

\newcommand{\BC}{{\mathbb {C}}}

\newcommand{\BG}{{\mathbb {G}}}

\newcommand{\BK}{{\mathbb {K}}}

\newcommand{\BN}{{\mathbb {N}}}

\newcommand{\BR}{{\mathbb {R}}}

\newcommand{\BZ}{{\mathbb {Z}}}

\newcommand{\CE}{{\mathcal {E}}}
\newcommand{\CF}{{\mathcal {F}}}
\newcommand{\CG}{{\mathcal {G}}}
\newcommand{\CH}{{\mathcal {H}}}

\newcommand{\CJ}{{\mathcal {J}}}
\newcommand{\CK}{{\mathcal {K}}}

\newcommand{\CN}{{\mathcal {N}}}
\newcommand{\CO}{{\mathcal {O}}}

\newcommand{\CS}{{\mathcal {S}}}

\newcommand{\CU}{{\mathcal {U}}}

\newcommand{\CW}{{\mathcal {W}}}
\newcommand{\CX}{{\mathcal {X}}}

\newcommand{\CZ}{{\mathcal {Z}}}

\newcommand{\CaWa}{\mathrm{CW}}

\newcommand{\FMG}{{\mathrm{SF}}}

\newcommand{\Fg}{{\mathfrak {g}}}

\newcommand{\Fk}{{\mathfrak {k}}}
\newcommand{\Fl}{{\mathfrak {l}}}

\newcommand{\Fn}{{\mathfrak {n}}}
\newcommand{\Fo}{{\mathfrak {o}}}
\newcommand{\Fp}{{\mathfrak {p}}}

\newcommand{\Fw}{{\mathfrak {w}}}

\newcommand{\RA}{{\mathrm {A}}}
\newcommand{\RB}{{\mathrm {B}}}

\newcommand{\RD}{{\mathrm {D}}}

\newcommand{\RF}{{\mathrm {F}}}
\newcommand{\RG}{{\mathrm {G}}}
\newcommand{\RH}{{\mathrm {H}}}
\newcommand{\RI}{{\mathrm {I}}}
\newcommand{\RJ}{{\mathrm {J}}}

\newcommand{\RL}{{\mathrm {L}}}
\newcommand{\RM}{{\mathrm {M}}}
\newcommand{\RN}{{\mathrm {N}}}
\newcommand{\RO}{{\mathrm {O}}}
\newcommand{\RP}{{\mathrm {P}}}

\newcommand{\RR}{{\mathrm {R}}}
\newcommand{\RS}{{\mathrm {S}}}
\newcommand{\RT}{{\mathrm {T}}}
\newcommand{\RU}{{\mathrm {U}}}

\newcommand{\RW}{{\mathrm {W}}}

\newcommand{\Ad}{{\mathrm{Ad}}}

\newcommand{\disc}{{\mathrm{disc}}}

\newcommand{\FJ}{{\mathrm{FJ}}}

\newcommand{\Gal}{{\mathrm{Gal}}}
\newcommand{\GL}{{\mathrm{GL}}}

\newcommand{\Hom}{{\mathrm{Hom}}}

\renewcommand{\Im}{{\mathrm{Im}}}

\newcommand{\Ind}{{\mathrm{Ind}}}

\newcommand{\Jac}{{\mathrm{Jac}}}

\newcommand{\Mp}{\wt{\mathrm{Sp}}}

\renewcommand{\Re}{{\mathrm{Re}}}

\newcommand{\Res}{{\mathrm{Res}}}

\newcommand{\SO}{{\mathrm{SO}}}

\newcommand{\SU}{{\mathrm{SU}}}
\newcommand{\Sym}{{\mathrm{Sym}}}
\newcommand{\sgn}{{\mathrm{sgn}}}
\newcommand{\Sp}{{\mathrm{Sp}}}

\newcommand{\Span}{{\mathrm{Span}}}
\newcommand{\gen}{\mathrm{genu}}

\newcommand{\udl}{\underline}
\newcommand{\wh}{\widehat}
\newcommand{\wt}{\widetilde}

\newcommand{\lra}{\longrightarrow}

\newcommand{\bs}{\backslash}

\def\diag{{\rm diag}}

\def\std{\rm std}

\def\p{\prime}

\def\Irr{\mathrm{Irr}}

\def\t{\tilde}
\def\ss{\mathrm{ss}}

\def\Vogan{\mathrm{Vogan}}

\def\rel{\mathrm{rel}}

\newtheorem{thm}{Theorem}[subsection]
\newtheorem{theorem}{Theorem}
\newtheorem{defin}[thm]{Definition}

\newtheorem{ex}[thm]{Example}
\newtheorem{pro}[thm]{Proposition}
\newtheorem{lem}[thm]{Lemma}
\newtheorem{cor}[thm]{Corollary}

\newtheorem{conjec}{Conjecture}

\makeatletter

\newcommand{\Rmnum}[1]{\expandafter\@slowromancap\romannumeral #1@}
\makeatother

\begin{document}

\title[Multiplicity formula]{Multiplicity formula for induced representations:  Bessel and Fourier-Jacobi models over Archimedean local fields}

\author{Cheng Chen}
\address{School of Mathematics\\
University of Minnesota\\
Minneapolis, MN 55455, USA}
\email{chen5968@umn.edu}

\date{\today}
\keywords{Bessel models, Fourier-Jacobi models, multiplicity, local Gan-Gross-Prasad conjectures, Meromorphic continuation}
\subjclass[2020]{Primary 22E50 22E45; Secondary 20G20}
\maketitle
\begin{abstract}
This article proves a formula relating the multiplicity of an induced representation and that of the inducing datum for the Bessel and the Fourier-Jacobi models over Archimedean local fields by generalizing the approach of C. M\oe glin and J.-L. Waldspurger in \cite{moeglin2012conjecture}, which was successful for Bessel models of special orthogonal groups over non-Archimedean local fields. As an application, we give a uniform proof of 
the local Gan-Gross-Prasad conjecture for all classical groups over Archimedean local fields for generic local $L$-parameters based on the tempered basic cases.
\end{abstract}

\tableofcontents

\section{Introduction}\label{section: intro}

The Bessel and Fourier-Jacobi models are important concepts in studying automorphic forms and their $L$-functions. Recent interests in these models are related to the local Gan-Gross-Prasad conjecture and the global Gan-Gross-Prasad conjecture.

The local Gan-Gross-Prasad conjecture was formulated by W. Gan, B. Gross, and D. Prasad in \cite{gan2012symplectic} generalizing the local Gross-Prasad conjecture in \cite{gross1992decomposition}\cite{gross1994irreducible} from special orthogonal groups to general classical groups. We first recall the local Gan-Gross-Prasad conjecture for classical groups.  

Let $F$ be a local field of characteristic $0$ and $E$ be a field extension satisfying $[E: F]=1, 2$. There are Bessel cases and Fourier-Jacobi cases for the conjecture.

In Bessel cases,  for given hermitian spaces $W\subset V$ over $E$ such that the complement $W^{\perp}$ of $W$ in $V$ is odd-dimensional and split over $E$. Let
\[
\RG=\begin{cases}
 \SO(V)\times \SO(W), & \text{ when }E=F, \\
 \RU(V)\times \RU(W),& \text{ when }E\neq F.
\end{cases}
\]
Then $\RG$ is the product of two classical algebraic groups $\RG_V$, $\RG_W$ of the same type over $F$.   Let 
\[
\RH= \Delta\RG_W\rtimes\RN,
\]  where $\Delta \RG_W$ denotes the image of diagonal embedding of $\RG_W$ into $\RG$, and $\RN$ is the unipotent radical of the parabolic group of $\RG_V$ stabilizing a full totally isotropic flag  of $W^{\perp}$. Let $\xi$ be a character of $\RH(F)$ obtained from a generic character of $\RN(F)$.

In Fourier-Jacobi cases,  given skew-hermitian spaces $W\subset V$
over $E$. Let
\[
\RG=\begin{cases}
 \Sp(V)\times (\Sp(W)\rtimes \CH(W)) \text{ or }\Mp(V)\times (\Mp(W)\rtimes \CH(W)),& \text{ when }E= F,\\ \RU(V)\times (\RU(W)\rtimes \CH(W)), & \text{ when }E\neq F, 
\end{cases}
\]
where $\wt{\Sp}(V)$ is a nonsplit double cover of $\Sp(V)$ and $\CH(W)$ is the Heisenberg group associated to $W$. In the above cases, $\RG$ is equal to the product of 
a classical algebraic group (or a metaplectic group) $\RG_V$ over $F$ and a Jacobi group $\RG_W^J=\RG_W\rtimes \CH(W)$, where $\RG_W$ is of the same type as $\RG_V$. 

\begin{itemize}
    \item When $W\subsetneq V$, fix a hyperbolic plane $H_1=X_1\oplus Y_1$, where $X_1,Y_1$ are isotropic line over $E$. The Jacobi groups $\RG_W^J$ can be taken as the subgroup of $\RG_{W\oplus^{\perp} H_1}$  fixing all points of $X_1$. Let 
    \[
    \RH=\Delta \RG_W^J\rtimes \RN,
    \] where $\RN$ is the unipotent radical of the subgroup of $\RG_V$ stabilizing a full totally isotropic flag of $(W\oplus H_1)^{\perp}$ and $\xi$ is a character of $\RH(F)$ induced from a generic character $\xi_N$ of $\RN(F)$.
    \item When $W=V$, let \[
    \RH=\Delta\RG_V,
    \] and $\xi$ be the trivial character of $\RH(F)$.
    
\end{itemize}

For  a triple $(\RG,\RH,\xi)$, nonzero elements in
\[
\Hom_{\RH}(\pi,\xi)
\]
are called \textbf{Bessel models} or \textbf{Fourier-Jacobi models},
and the local Gan-Gross-Prasad conjecture speculates the \textbf{multiplicity} 
\[
m(\pi)=\dim\Hom_{\RH}(\pi,\xi)
\]
for the $\RG(F)$-representation
\[
\pi=\begin{cases}
\pi_V\boxtimes \pi_W&\text{ in Bessel cases}\\ 
    \pi_V\boxtimes (\wt{\pi}_W\otimes \omega_{W,\psi_F})& \text{ in Fourier-Jacobi cases}
\end{cases}
\] where $\pi_V,\pi_W,\wt{\pi}_W$ are irreducible admissible representations when $F$ is non-Archimedean, and irreducible Casselman-Wallach representations when $F$ is Archimedean.  Moreover, the representations of metaplectic groups are chosen to be genuine representations.  Here $\omega_{W,\psi_F}$ is the Weil representation of $\wt{\Sp}(W)\rtimes \CH(W)$ or $\RU(W)\rtimes \CH(W)$ associated to the character $\psi_F$.

The multiplicity-one theorem  that
\[
m(\pi)\leqslant 1
\]
was proved in \cite{aizenbud2010multiplicity}\cite{waldspurger2012variante}\cite{sun2012multiplicitynon} when $F$ is non-Archimedean and was proved in  \cite{sun2012multiplicity}\cite{jiang2010uniqueness}\cite{liu2013uniqueness} when $F$ is Archimedean.

The local Gan-Gross-Prasad conjecture is a refinement of the multiplicity-one theorem speculating the multiplicity for representations of all the pure inner forms and giving precise conditions for when the multiplicity is equal to one if the representation has generic $L$-parameters.

The \textbf{Vogan $L$-packet} gathers all representations with the given $L$-parameter of all the pure inner forms. More precisely, for $\RG_V=\SO(V),\RU(V),\Sp(V)$, all pure inner forms  of $\RG_V$ are $\RG_{V_{\alpha}}$ for $\alpha\in H^1(F,\RG_V)$. We denote by $\mathrm{WD}_F$ the Weil-Deligne group of $F$. 
For given $L$-parameter $\varphi_V:\mathrm{WD}_F\to {}^L\RG_{V}$, R. Langlands associated an $L$-packet $\Pi_{\varphi_V}(\RG_{V})$, which is a finite set of representations in $\mathrm{Irr}(\RG_{V})$.
Here $\mathrm{Irr}(\RG_{V})$ is the category of irreducible admissible representations of $\RG_{V}(F)$ when $F$ is non-Archimedean, and the category of irreducible Casselman-Wallach representations when $F$ is Archimedean.
Since $\RG_V={}^L\RG_{V_{\alpha}}$ for all $\alpha\in H^1(F,\RG_V)$, the $L$-parameter $\varphi_V$ can be taken as an $L$-parameter of $\RG_{V_{\alpha}}$.  The Vogan packet 
$\Pi^{\Vogan}_{\varphi_V}$ is defined as
\[
\coprod_{\alpha\in H^1(F,\RG_V)}\Pi_{\varphi_V}(\RG_{V_{\alpha}}).
\]

When $\varphi_V$ is generic, it was conjectured by D. Vogan and known over Archimedean local fields (\cite[Theorem  6.3]{vogan1993local}) that, for a fixed Whittaker datum,  there is a bijection 
\begin{equation}\label{equ: isomorphism component group}
\pi_V\in\Pi^{\Vogan}_{\varphi_V}\longleftrightarrow \chi_{\pi_V}\in\widehat{\CS}_{\varphi_V}.
\end{equation}
Here $\widehat{\CS}_{\varphi_V}$ is the set of character of  component group $\CS_{\varphi_V}$ of $\RS_{\varphi_V}$, where $\RS_{\varphi_V}$ is the centralizer of the image $\Im(\varphi_V)$ in the dual group $\wh{\RG}$.

For the metaplectic group $\Mp(V)$, there is a quadratic space $V'$ over $E=F$ with $\dim_F V'=\dim_F V+1$ and 
 the discriminant equal to $1$. Then the $L$-group ${}^L\wt{\Sp}(V)={}^L\SO(V')=\Sp_{\dim_F V}(\BC)$. 
 The Shimura-Waldspurger correspondence (\cite{adams1998genuine}\cite{gan2012representations}) gives an isomorphism 
 \[
 \theta_V:\mathrm{Irr}^{\gen}(\Mp(V))\to \coprod_{\alpha\in H^1(F,\SO(V'))}\mathrm{Irr}(\SO(V'_{\alpha})),
 \]
where  $\mathrm{Irr}^{\gen}$ denotes the subcategory of $\mathrm{Irr}$ consists of genuine representations.

 For a given $L$-parameter $\wt{\varphi}_V:\mathrm{WD}_F\to {}^L\Mp(V)$, the Vogan $L$-packet associated to $\wt{\varphi}_V$ consisting of irreducible genuine representation of $\Mp(V)$ was defined as 
 \[
 \Pi^{\Vogan}_{\wt{\varphi}_{V}}=\theta_V^{-1}(\Pi^{\Vogan}_{\varphi_{V'}}),
 \]
where $\varphi_{V'}$ is the associated $L$-parameter $\mathrm{WD}_F\to {}^L\SO(V')={}^L\Mp(V)$.

When the $L$-parameter $\wt{\varphi}_V$ is generic, from \eqref{equ: isomorphism component group}, there is an isomorphism
\begin{equation}\label{equ: isomorphism component group meta}
\wt{\pi}_V\in\Pi^{\Vogan}_{\wt{\varphi}_V}\longleftrightarrow \chi_{\wt{\pi}_V}\in\widehat{\CS}_{\wt{\varphi}_V}.
\end{equation}

In the Bessel cases, $\RG=\RG_V\times \RG_W$. The inner twist by $\alpha\in H^1(F,\RH)=H^1(F,\RG_W)$ defines a new triple $(\RG_{\alpha},\RH_{\alpha},\xi_{\alpha})$, where 
\begin{equation}\label{equ: twist Bessel}
\RG_{\alpha}=\RG_{V_{\alpha}}\times \RG_{W_{\alpha}},\quad \RH_{\alpha}=\Delta \RG_{W_{\alpha}}\rtimes \RN,
\end{equation}
and $\xi_{\alpha}:\RH_{\alpha}(F)\to \BC$ is induced from the generic character $\xi_N$ of $\RN(F)$.  With this triple, one can define multiplicity for representations $\pi_{\alpha}$ in $\mathrm{Irr}(\RG_{\alpha})$, which we denote by $m(\pi_{\alpha})$. Gan-Gross-Prasad suggested that one may consider the relevant Vogan $L$-packet
\[
\Pi^{\Vogan}_{\varphi,\rel}=\bigsqcup_{\alpha\in H^1(F,\RH)}\Pi_{\varphi}(\RG_{\alpha}).
\]
\begin{conjec}[\cite{gan2012symplectic}]\label{conj: Bessel intro}
Let $\varphi_V,\varphi_W$ be generic $L$-parameters of $\RG_V,\RG_W$, respectively. Let $\varphi=\varphi_V\times \varphi_W$ and fix a Whittaker datum as in 
\cite[\S 12]{gan2012symplectic}, there are the following properties.
\begin{enumerate}
    \item  There is exactly one representation $\pi=\pi_V\times \pi_W\in \Pi^{\Vogan}_{\varphi,\rel}$ such that $m(\pi)=1$.
    \item  The unique representation $\pi$ in $\Pi_{\varphi,\rel}^{\Vogan}\subset \Pi_{\varphi}^{\Vogan}$ such that $m(\pi)=1$ satisfy
\[
\chi_{\pi_V}=\chi_{\varphi_V,W},\quad \chi_{\pi_W}=\chi_{\varphi_W,V}, 
\]
where $\chi_{\varphi_V,W},\chi_{\varphi_W,V}$ were constructed with symplectic local root number by Gan-Gross-Prasad in \cite{gan2012symplectic} (see also \eqref{equ: definition of character}).
\end{enumerate}
\end{conjec}

When $F$ is non-Archimedean, the conjecture was proved by Waldspurger (\cite{waldspurger2010formule}\cite{waldspurger2012calcul}\cite{waldspurger2012conjecture}\cite{waldspurger2012formule}\cite{waldspurger2012variante}) for tempered $L$-parameters using local trace formula when $\RG_V=\SO(V)$.  M\oe glin and Waldspurger completed the proof for generic $L$-parameters based on the results in the tempered cases. Following a similar approach, R. Beuzart-Plessis completed the proof for the tempered $L$-parameters in \cite{beuzart2014expression}\cite{beuzart2016conjecture} when $\RG_V=\RU(V)$; then based on the tempered results, Gan and A. Ichino proved the conjecture for generic $L$-parameters in \cite{gan2016gross}.

When $F$ is Archimedean and $\RG_V=\RU(V)$, H. He proved the conjecture for discrete series $L$-parameters in \cite{he2017gan}. Beuzart-Plessis proved the first part of the conjecture for tempered $L$-parameters by Beuzart-Plessis in \cite{beuzart2019local} using the local trace formula. H. Xue proved the full conjecture for tempered $L$-parameters using theta correspondence in \cite{xue1bessel} and completed the proof for generic $L$-parameters using Schwartz homology in \cite{xue2020bessel}. When $F$ is Archimedean and $\RG_V=\RU(V)$, the first part of the conjecture for tempered $L$-parameters was proved by Z. Luo in \cite{luo2020local} following \cite{beuzart2019local}.  Luo and the author proved the second part of the conjecture in \cite{chen2022local} by refining Waldspurger's approach. The author proved the generic cases in \cite{chen2021local} following \cite{moeglin2012conjecture} and \cite{xue2020bessel}.

In the Fourier-Jacobi cases, $\RG=\RG_V\times (\RG_W\rtimes \CH(W))$. 
 We define $\wt{\RG}_W$  as
\begin{equation}\label{equ: def of tilde GW}
\wt{\RG}_W=\begin{cases}
    \RU(W) \text{ if } \RG_W =\RU(W),\\
    \wt{\Sp}(W) \text{ if } \RG_W=\Sp(W), 
    \\
    \Sp(W) \text{ if } \RG_W=\wt{\Sp}(W)\\
\end{cases}
\end{equation}
Here $\wt{\Sp}(W)$ is taken to be $\Sp(W)$ when $\Sp(W)$ does not have a non-split double cover.

\begin{conjec}\label{conj: intro FJ}(\cite{gan2012symplectic})
Let $\varphi_V$ be a generic $L$-parameter of $\RG_V$ and $\wt{\varphi}_W$ be a generic $L$-parameter of $\wt{\RG}_W$.
\begin{enumerate}
    \item There is exactly one pair of representations $\pi_V\in \Pi^{\Vogan}_{\varphi_V}$, $\wt{\pi}_W\in \Pi_{\wt{\varphi}_W}^{\Vogan}$  such that $m(\pi_V\boxtimes (\wt{\pi}_W\otimes \omega_{W,\psi_F}))=1$.
    \item The unique pair of representation $\pi_V\in\Pi_{\varphi_V}^{\Vogan}$ and
$\wt{\pi}_W\in \Pi_{\wt{\varphi}_W}^{\Vogan}$ such that $m(\pi_V\boxtimes (\wt{\pi}_W\otimes \omega_{W,\psi_F}))=1$ satisfy
\[
\chi_{\pi_V}=\chi_{\varphi_V,W},\quad \chi_{\wt{\pi}_W}=\chi_{\varphi_W,V}
\]
where $\chi_{\varphi_{V},W},\chi_{\varphi_{W},V}$ were constructed by Gan-Gross-Prasad in \cite{gan2012symplectic} (see also \eqref{equ: definition of character}).    
\end{enumerate}
\end{conjec}

Over non-Archimedean local fields, based on the results for Bessel models, Gan and Ichino proved the unitary cases in \cite{gan2016gross} using theta correspondence. Following a similar approach, H. Atobe completed the proof for symplectic-metaplectic cases in \cite{atobe2018local}.

Over Archimedean local fields, Xue proved the unitary cases using theta correspondence and Schwartz homology. As for symplectic-metaplectic cases, it is our work-in-progress joint with R. Chen and J. Zou attacking the equal-rank tempered cases. As far as we know, Z. Li, Xue, and S. Wang are also considering the same case with a different approach. 
The general symplectic-metaplectic cases from equal-rank tempered cases are still open and will be confirmed as one consequence of this article.

The general approach for Conjecture \ref{conj: Bessel intro} and Conjecture \ref{conj: intro FJ} are based on the corresponding tempered cases and a reduction using a \textbf{multiplicity formula} (Conjecture \ref{conjecture: main}) that relates the multiplicity of an induced representation and that of the inducing datum. However, the known approaches for the multiplicity formula in the previous works vary case by case and involve different techniques, such as the integral method, distributional analysis, Schwartz analysis, and theta correspondence.

Our objective in this article is to give a uniform approach for the multiplicity formula for both Bessel models and Fourier-Jacobi models when $F$ is Archimedean. The reduction is established by generalizing the  M\oe glin and Waldspurger's approach, whose main steps only use the integral method and distributional analysis. This article makes some refinements such that the proof is independent of three facts used in M\oe glin and Waldspurger's original approach, the local Gan-Gross-Prasad conjecture for tempered $L$-parameters, the fact that the nonvanishing of the multiplicity implies the nonvanishing of the tempered intertwining  (\cite[Proposition 5.7]{waldspurger2012formule}) and the existence of Casselman's canonical pairing. The non-Archimedean uniform approach will be established in our forthcoming work.

In most general contexts, we describe the multiplicity formula as the following conjecture.
\begin{conjec}\label{conjecture: main}
For representations of general linear groups
\[
\sigma_V=|\cdot|^{s_{V,1}}\sigma_{V,1}\times \cdots \times |\cdot|^{s_{V,l_V}}\sigma_{V,l_V}\quad \sigma_W=|\cdot|^{s_{W,1}}\sigma_{W,1}\times \cdots \times |\cdot|^{s_{W,l_W}}\sigma_{W,l_W}
\]
where $\sigma_{V,i},\sigma_{W,i}$ are irreducible tempered representations and $s_{V,i},s_{W,i}\in \BC$ satisfying 
\[
\Re(s_{V,1})\geqslant \cdots\geqslant \Re(s_{V,l_V})>0,\Re(s_{W,1})\geqslant \cdots\geqslant \Re(s_{w,l_W})>0.
\]
\begin{enumerate}    
\item {\bf (Bessel cases):} For any two possibly reducible representations
\[
\pi_V=\sigma_{V}\rtimes \pi_{V_0},\quad
\pi_W=\sigma_W\rtimes \pi_{W_0}
\]
where $\pi_{V_0},\pi_{W_0}$ are irreducible tempered representations, the  identity 
\[
m(\pi_V\boxtimes \pi_W)= m(\pi_{V_0}\boxtimes \pi_{W_0})
\]
holds. 
\item {\bf (Fourier-Jacobi cases):}
For any two possibly reducible representations
\[
\pi_V=\sigma_{V}\rtimes \pi_{V_0},\quad
\wt{\pi}_W=\sigma_W\rtimes \wt{\pi}_{W_0}
\]
where $\pi_{V_0},\wt{\pi}_{W_0}$ are irreducible tempered representations, the identity 
\[
m(\pi_V\boxtimes (\wt{\pi}_W\otimes \omega_{W,\psi_F}))= m(\pi_{V_0}\boxtimes (\wt{\pi}_{W_0}\otimes\omega_{W_0,\psi_F}))
\]
holds. 
 \end{enumerate}   
\end{conjec}

Generally speaking, when $\pi_V,\pi_W,\wt{\pi}_W$ are irreducible, the local Gan-Gross-Prasad conjecture implies results in Conjecture \ref{conjecture: main}.
 When considering reducible cases, For Bessel cases over non-Archimedean local fields, Conjecture \ref{conjecture: main} was proved in \cite{moeglin2012conjecture}\cite{gan2016gross} using results in proof for the local Gan-Gross-Prasad conjecture for tempered $L$-parameters. This article proves Conjecture \ref{conjecture: main} for all Archimedean cases uniformly, and the proof is independent of the local Gan-Gross-Prasad conjecture.

\begin{theorem}\label{thm: conj 3}
 Conjecture \ref{conjecture: main} holds over Archimedean local fields.    
\end{theorem}

For Bessel models, the inequality
\[
m(\pi_V\boxtimes \pi_W)\leqslant m(\pi_{V_0}\boxtimes \pi_{W_0})
\]
is called "the first inequality" and the inequality
\[m(\pi_V\boxtimes \pi_W)\geqslant m(\pi_{V_0}\boxtimes \pi_{W_0})\]
is called "the second inequality". 

In the proof for non-Archimedean Bessel cases in \cite{moeglin2012conjecture}\cite{gan2016gross}, there are three main ingredients in the proof of Conjecture \ref{conjecture: main}: Let $W^+=(X^+\oplus Y^+)\oplus W$ 
     such that $\dim_E W^+=\dim_E V+1$. Suppose $\pi_V$, $\pi_W$ are finite-length admissible representations of $\RG_V(F)$, $\RG_W(F)$.
\begin{enumerate}
    \item (Reduction to basic cases) For a suppercuspidal representation $\sigma_{X^+}$ of $\GL(X^+)(E)$, we have
    \[
    m(\pi_V\boxtimes \pi_W)=m((\sigma_{X^+}\rtimes \pi_W)\boxtimes \pi_V)
    \]
    \item (Basic forms of the first inequality)
    For a tempered representation $\sigma_{X^+}$ of $\GL(X^+)(E)$, we have
    \[
    m(\pi_V\boxtimes \pi_W)\geqslant m((|\det|^s\sigma_{X^+}\rtimes \pi_W)\boxtimes \pi_V)
    \]
    when $\Re(s)$ is larger than or equal to a parameter $\mathrm{LI}(\pi_V)$ of $\pi_V$.
    \item (Basic forms of the second inequality) For a representation  $\sigma_{X^+}$ of $\GL(X^+)(E)$ as in Conjecture \ref{conjecture: main}, we have
    \[
    m(\pi_V\boxtimes \pi_W)\leqslant m((\sigma_{X^+}\rtimes \pi_W)\boxtimes \pi_V)
    \]
\end{enumerate}

With these three ingredients, Conjecture \ref{conjecture: main} can be proved by  a mathematical induction from the tempered cases  
(\cite[\S 1.6]{moeglin2012conjecture} \cite[\S 9.3]{gan2016gross}).

Following the idea in \cite{xue2020bessel}\cite{chen2021local}, one can set up the Archimedean counterpart of these three steps in Bessel cases as follows: Let $W^+=(X^+\oplus Y^+)\oplus W$ such that $\dim_E W^+=\dim_E V+1$ and let $r^+=\dim_EX^+$. Suppose $\pi_V,\pi_W$ are Casselman-Wallach representations of $\RG_V(F),\RG_W(F)$.
\begin{enumerate}
    \item (Reduction to basic cases) For a spherical principal series representation $\sigma_{X^+}=|\cdot|^{s_1}\times \cdots \times |\cdot|^{s_{r^+}}$ of $\GL(X^+)(E)$, we have
    \[
    m(\pi_V\boxtimes \pi_W)=m((\sigma_{X^+}\rtimes \pi_W)\boxtimes \pi_V)
    \]
    for $(s_1,\cdots,s_{r^+})\in \BC^{r^+}$ in general positions.
    \item (Basic forms of the first inequality)
    \begin{enumerate}
        \item When $\dim_E X^+=1$,  
        \[
    m(\pi_V\boxtimes \pi_W)\geqslant m((|\cdot|^s\rtimes \pi_W)\boxtimes \pi_V)
    \]
    when $\Re(s)\geqslant \mathrm{LI}(\pi_V)$.
        \item When $E=\BR$ and $\dim_E X=2$, 
        \[
    m(|\cdot|^{s+\frac{m}{2}}\sgn^{m+1}\rtimes\pi_V\boxtimes \pi_W)\geqslant m(|\det|^sD_m\rtimes \pi_W\boxtimes \pi_V)
        \]
    when $\Re(s)+\frac{m}{2}\geqslant \mathrm{LI}(\pi_V)$,    where $D_m$ is the discrete series representation of $\GL_2(\BR)$ with the same central character as the $(2m-1)$-dimensional irreducible representation.
    \end{enumerate}
    \item (Basic forms of the second inequality)
    When $\pi_V,\pi_W$ are irreducible, for a generic representation  $\sigma_{X^+}$ of $\GL(X^+)(E)$, we have
    \[
    m(\pi_V\boxtimes \pi_W)\leqslant m((\sigma_{X^+}\rtimes \pi_W)\boxtimes \pi_V)
    \]
\end{enumerate}

For the reduction to basic cases, \cite{xue2020bessel}\cite{chen2021local} used the technique of Schwartz homology. This approach can also be applied to the Fourier-Jacobi cases (Section \ref{section: proof for (1)}). 

For the basic forms of the first inequality, the result follows from M\oe glin and Waldspurger's approach by using certain refined distribution analysis  based on Borel's lemma in \cite{sun2012multiplicity}, and a vanishing result (in Appendix \ref{section: vanishing app}).

For the basic forms of the second inequality, M\oe glin and Waldspurger's proof used Casselman's canonical pairing, which was proved to exist only for Harish-Chandra modules in Archimedean situations.  To deal with the representations of Casselman-Wallach type, one needs new ideas. This article uses an approach to prove the second inequality 
based on our generalization of the inequalities in   \cite{jiang2010uniqueness} and \cite{liu2013uniqueness} by using methods in \cite{jacquet1990exterior} and \cite{gourevitch2019analytic}. In particular, the $\dim_E X^+=1$ cases was proved in \cite{chen2021local}.

Fourier-Jacobi models have two types of basic cases: equal-rank cases and almost equal-rank cases. For each type, one can formulate  these three steps. To give the proof using parallel geometric structures as the Bessel cases, we will formulate two families of Fourier-Jacobi models (FJ 1) and (FJ 2) associated to a pair of skew-hermitian spaces $(V,W)$ over $E$ such that the basic cases associated to (FJ 1) situations are equal-rank cases, and those associated to (FJ 2) situations are almost equal-rank cases.

\begin{enumerate}
    \item (FJ 1) When $W\subsetneq V$, we set $(\RG,\RH,\xi)$ as before.
    \item (FJ 2) When $V\subset W$, we set  $\RG=\RG_V\times \RG_W^J$. Let $W=V\oplus^{\perp} (X\oplus Y)$, where $X,Y$ are totally isotropic spaces over $E$. We take 
\[
\RH=\Delta \RG_V\rtimes (\RN\rtimes X),
\] 
where $\RN$ is the unipotent part of a parabolic subgroup of $\RG_W$ stabilizing a full totally isotropic flag of $X$.  We take $\xi$ as a unitary character of $\RH(F)$ induced by a generic unitary character of $\RN\rtimes X(F)$.
\end{enumerate}

Except for the equal-rank cases, every (FJ 2) case corresponds to an equivalent (FJ 1) case, but we associate different basic cases to them. More precisely, 
\begin{itemize}
    \item when $V=W$, the associated  (FJ 2) case is precisely the equal-rank case in Conjecture~\ref{conjecture: main}; 
\item when $V\subsetneq W$, the associated (FJ 2) case is isomorphic to the (FJ 1) for 
 $\wt{\RG}=\wt{\RG}_W\times
 \wt{\RG}_V^J$ with associated $\wt{\RH},\wt{\xi}$  via the following isomorphism  (Lemma \ref{lem: isomorphism between FJ1 and FJ2})
\[ \Hom_{\wt{\RH}}(\wt{\pi}_V\boxtimes(\pi_W\otimes \omega_{W,\psi_F}),\wt{\xi})=\Hom_{\RH}(\pi_W\boxtimes(\wt{\pi}_V\otimes \omega_{V,\psi_F}),\xi),
\]   
where $\wt{\pi}_V\in \mathrm{Irr}(\wt{\RG}_V),\pi_W\in \mathrm{Irr}(\RG_W)$, respectively.
\end{itemize}

From now on, we fix the ground field  $F$ to be $\BR$ or $\BC$. To uniformly give the proof for Bessel cases (we abbreviate them as (B))  and Fourier-Jacobi cases (FJ 1)(FJ 2), we introduce the following notions.

\begin{defin}
For given pair $(V,W)$ in (B),  (FJ 1) case or (FJ 2) case, we respectively define the associated basic pair $(V',W')$, which is 
\begin{enumerate}
    \item a  (B) case with $\dim_{E} V'=\dim_{E} W'+1$ and 
    \[
    W'=V,\quad V'=W^+=W\oplus^{\perp} (X^+\oplus Y^+);
    \]
    \item a  (FJ 2) case with $\dim_{E}V'=\dim_{E}W'$ and 
    \[
    V'=V,\quad W'=W^+=W\oplus^{\perp}(X^+\oplus Y^+);
    \]
    \item a (FJ 1) case with $\dim_EV'=\dim_E W'+2$ and
    \[W'=W,\quad V'=V^+=V\oplus^{\perp}(X^+\oplus Y^+).\]
\end{enumerate}

We denote by $(\RG^+,\RH^+,\xi^+)$ the triple associated to an admissible pair $(V',W')$. In particular, $\xi^+$ is always the trivial character of $\RH^+(F)$.
\end{defin}

\begin{defin}
Following \cite{liu2013uniqueness}, for decomposition $W^+=(X^+\oplus Y^+)\oplus W$, Casselman-Wallach representations $\sigma_{X^+}$ of $\GL(X^+)(E)$, $\pi_W$ of the Jacobi group $\RG_W^J(F)$, we define $\sigma_{X^+}\rtimes \pi_W^J$ in Definition \ref{defin: parabolic induction}, using Schwartz induction with the property
\[
\sigma_{X^+}\rtimes (\pi_W\otimes \omega_{W,\psi_F})=(\sigma_{X^+}\rtimes \pi_W)\otimes \omega_{W^+,\psi_F}.
\]
\end{defin}

In the (B), (FJ 1), and (FJ 2) cases, we define a pseudo-parabolic subgroup $\RP^+$ of $\RG^+$ (\ref{def: P^+}) with pseudo-Levi decomposition $\RP^+=\RL^+\rtimes \RN^+$ and $\RL^+\cong \Res_{E/F}\GL(X^+)\times \RG$ such that  the  Schwartz induction from $\RP^+$ to $\RG^+$ of the inflation of the representation $\sigma_{X^+}\boxtimes \pi_V\boxtimes \pi_W$ of $\RL^+$ is equal to
\[
\RI(\sigma_{X^+},\pi_V,\pi_W)=\begin{cases}
(\sigma_{X^+}\rtimes \pi_W)\boxtimes\pi_V& \text{in (B) cases}, \\
\pi_V\boxtimes(\sigma_{X^+}\rtimes \pi_W)& \text{in (FJ 1) cases},\\
(\sigma_{X^+}\rtimes\pi_V)\boxtimes \pi_W &\text{ in  (FJ 2) cases}.
\end{cases}
\]

We can formulate the following theorem. 
    Here we say that $\underline{s}=(s_1,\cdots,s_{\dim_EX^+})\in \BC^{\dim_EX^+}$ is \textbf{in general positions} if $\underline{s}$ does not lie in the set of zeros of countably many fixed polynomial functions on $\BC^{\dim_E X^+}$.

\begin{theorem}\label{thm: main intro}
With the notations introduced as above, the following hold. 
\begin{enumerate}
    \item (Reduction to basic cases) For spherical principal series representations $|\cdot|^{s_1}\times \cdots \times |\cdot|^{s_{\dim_EX}}$ of $\GL(X^+)(E)$
    \[
    m(\pi_V\boxtimes \pi_W)=m(\RI(\sigma_{X^+}, \pi_W,\pi_V))
    \]
    for $(s_1,\cdots,s_{\dim_E X})\in \BC^{\dim_EX}$ in general positions.
    \item (Basic form of the first inequality)
    \begin{enumerate}
        \item When $\dim_E X^+=1$,  
        \[
    m(\pi_V\boxtimes \pi_W)\geqslant m(\RI(|\cdot|^s,\pi_W,\pi_V))
    \]
    when $\Re(s)\geqslant \mathrm{LI}(\pi_V)$.
        \item When $E=\BR$ and $\dim_E X=2$, 
        \[
    m(\RI'(|\cdot|^{s+\frac{m}{2}}\sgn^{m+1},\pi_V, \pi_W))\geqslant m(\RI(|\det|^sD_m, \pi_V, \pi_W))
        \]
    when $\Re(s)+\frac{m}{2}\geqslant \mathrm{LI}(\pi_V)$,    where $\RI'$ is defined in Definition \ref{def: Iprime}.
    \end{enumerate}
    \item (Basic form of the second inequality)
    For generic Casselman-Wallach representation  $\sigma_{X^+}$ of $\GL(X^+)(E)$, we have
    \[
    m(\pi_V\boxtimes \pi_W)\leqslant m(\RI(\sigma_{X^+}, \pi_V, \pi_W))
    \]
\end{enumerate}
\end{theorem}

 The proof for Point (2) needs the idea of M\oe glin and Waldspurger that uses the properties of the $\Hom$-functor based on the analytic and geometric structure of the double cosets  $\RP^+(F)\bs \RG^+(F)/\RH^+(F)$. Over Archimedean local fields, this article computes the geometric structure in Section \ref{section: double coset} based on \cite{ginzburg2011descent} and computes the analytic structure in Section \ref{section: main 2} using results in \cite{chen2020schwartz}\cite{xue2020bessel}
\cite{chen2021local}. The properties of the $\Hom$-functor over Archimedean local fields are studied in Lemma \ref{lem: three properties} with a vanishing result proved in the appendix.

 The proof for Point (1) uses the method of Schwartz analysis in \cite{xue2020bessel}. The proof is parallel to the proof for Point (1) using properties of Schwartz homology (Lemma \ref{lem: two property Schwartz})
 based on similar geometric and analytic structure as in the proof for Point (2).

The Point (3) can be proved in the most general context. We only need to require $\sigma_{X^+}$ to be a generic representation. Special cases of Point (3) have been proved in previous works. 
\cite{jiang2010uniqueness}\cite{liu2013uniqueness} 
proved the second inequality for $\sigma_{X^+}=|\det|^s\sigma$ for a given principal series representation $\sigma$ when $\Re(s)$ is large enough. (\cite[Proposition 3.4]{jiang2010uniqueness} and \cite[Proposition 4.5]{liu2013uniqueness}).

This article refines their integral approach in Section \ref{section: integral}
 based on  the structure of the double cosets  \[\RP^+(F)\bs \RG^+(F)/\RH^+(F).\] Following \cite{jiang2010uniqueness} and \cite{liu2013uniqueness}, for a fixed nonzero Whittaker-functional $\lambda$ of $\sigma_{X^+}$, we construct an element
\[
\RJ_{s,\mu,\lambda}\in\Hom_{\RH^+}(\RI(|\det|^s\sigma_{X^+},\pi_V,\pi_W),1_{\RH^+})
\]
when $\Re(s)$ is large enough (Theorem \ref{thm: main}(1) and Theorem \ref{thm: main 2}(1)). 

Then we take  elements in 
$\Hom_{\RH^+}(\RI(|\det|^s\sigma_{X^+},\pi_V,\pi_W),1_{\RH^+})$   as $(V_{\sigma_{X^+}}\wh{\otimes}V_{\pi_V}\wh{\otimes}V_{\pi_W})^*$-valued $(\RP^+\times \RH^+,(|\det|^{s+s_0}\sigma_{X^+}\boxtimes\pi_V\boxtimes\pi_W)\times 1)$-equivariant distributions for some real number $s_0$ (Lemma \ref{lem: s_0}). For given $\mu\in \Hom_{\RH}(\pi_V\boxtimes \pi_W,\xi)$, when $\Re(s)$ is large enough, we take  $\RI_{s,\mu,\lambda}$ as equivariant distributions, we prove it can be  extended (Theorem \ref{thm: main}(2) and Theorem \ref{thm: main 2}(2)) to a  meromorphic family of equivariant distributions 
\[
\RJ_{s,\mu,\lambda}\in\Hom_{\RH^+}(
\RI(|\det|^s\sigma_{X^+},\pi_V,\pi_W),1_{\RH^+}),\quad s\in \BC.
\]

By taking its principal term at $s=0$, that is, the leading term of the Laurant coefficient at $s=0$, we construct a nonzero element in $\Hom_{\RH^+}(\RI(|\det|^s\sigma_{X^+},\pi_V,\pi_W),1_{H^+})
$. Therefore, when $m(\pi_V\boxtimes \pi_W)\neq 0$, we have $m(\RI(\sigma_{X^+},\pi_V,\pi_W))\neq 0$. Hence, we complete the proof for the second inequality with the multiplicity-one theorem that $m(\pi_V\boxtimes \pi_W)\leqslant 1$.

The key tools for proving the meromorphic continuation follows from \cite{bernstein1972analytic} and \cite{jacquet1990exterior}. More precisely, we apply the method in \cite{jacquet1990exterior} to prove the meromorphic continuation for the construction of a nonzero element $\mu^+$ in 
\[\Hom_{\RP^+\cap \RH^+}(\delta_{\RP^+\cap \RH^+}\otimes (\sigma_{X^+}\boxtimes\pi_V\boxtimes \pi_W),1_{\RP^+\cap \RH^+}).\] 
Then from $\mu^+$, we use the method in \cite{gourevitch2019analytic} to prove the meromorphic continuation for the construction of a nonzero element in 
 \[
 \Hom_{\RH^+}(\RI(\sigma_{X^+},\pi_V,\pi_W),1_{\RH^+}).
 \] 
 The proof for the meromorphic continuation in this construction is based on the results in \cite{bernstein1972analytic}.

Finally, using  Theorem \ref{thm: main intro}, we modify the mathematical induction in \cite{moeglin2012conjecture} and reduce Conjecture \ref{conjecture: main} to the tempered basic cases. 

This multiplicity formula in Conjecture \ref{conjecture: main} gives a uniform proof for the local Gan-Gross-Prasad conjecture from the tempered cases (Section \ref{section: GGP}). We can further reduce to tempered basic (B) case and (FJ 1) case with Theorem \ref{thm: main intro}(1). Besides, we applied this multiplicity formula to study the local descents and arithmetic wavefront set in my joint work with Jiang, D. Liu, and L. Zhang.

It is also worth mentioning that the multiplicity inequalities can be generalized to the Bessel models and Fourier-Jacobi models for orthogonal groups and general linear groups as in \cite{jiang2010uniqueness}\cite{liu2013uniqueness}. We include the proof for the orthogonal group cases in this article. To incorporate the general linear cases, we can take $V, W$ in the article as finite-rank free $\BK$-modules, where $\BK=E$ or $(E,E)$ equipped with an $\epsilon$-hermitian form. The details will be elaborated on in the author's thesis.

This article is organized as follows: 

In Section \ref{section: Bessel Fourier-Jacobi}, we recall the definitions of Bessel models and Fourier-Jacobi models over Archimedean local fields. In Section \ref{section: subgroup}, we recall some representation theory used in this article.  In Section \ref{section: integral}, we introduce the main geometric properties used in the article. In particular, we summarized the structure of double cosets $\RP^+(F)\bs \RG^+(F)/\RH^+(F)$ in the literature.

In Section \ref{section: integral}, we refine the integrals used in \cite{jiang2010uniqueness}\cite{liu2013uniqueness} to construct a family of integrals. Besides the absolute convergence and nonvanishing, we also prove the meromorphic continuation for the refined integral to obtain Theorem \ref{thm: main intro}(3).

In Section \ref{section: 5.3}, we follow the idea of M\oe glin and Waldspurger in \cite{moeglin2012conjecture} and prove  Theorem \ref{thm: main intro}(1)(2).
Then we apply a similar mathematical induction as in \cite{moeglin2012conjecture} to prove Theorem \ref{thm: conj 3} from Theorem \ref{thm: main intro}.

In Section \ref{section: GGP}, we give a uniform proof for the local Gan-Gross-Prasad conjecture from the basic cases for tempered $L$-parameters using Theorem \ref{thm: conj 3} and Theorem \ref{thm: main intro}(1).

\textbf{Acknowledgement. } I would like to thank Dihua Jiang and Lei Zhang for the discussions that encouraged me to think about \cite[Conjecture 3]{chen2021local} when I worked on the local Gross-Prasad conjecture and continuous encouragement for writing the results in a more general setting. This paper's work is partly supported by the Research Assistantship from the NSF grant DMS-19018024 and DMS-220089. I thank Paul Garrett for the discussions on the meromorphic continuation. I thank Rui Chen and Jialiang Zou for helping me understand the representation theory of metaplectic groups and Jacobi groups. I thank Dmitry Gourevitch, Binyong Sun, and Hang Xue for their helpful emails.

\section{Preliminaries}\label{section: settings}
\subsection{The Bessel and Fourier-Jacobi models}\label{section: Bessel Fourier-Jacobi}
In this section, we introduce the Bessel models and Fourier-Jacobi models following \cite{sun2012multiplicity}\cite{jiang2010uniqueness}\cite{liu2013uniqueness}.

Let $F=\BR$ or $\BC$.  We consider an algebraic field extension $E$ of $F$ and $\tau\in \Gal(E/F)$.

\begin{defin}\label{def: epsilon-hermitian}
An $\epsilon$-\textbf{hermitian space} $V$  is a  vector space  over $E$ with a nondegenerate $F$-bilinear map
\[
\langle \cdot , \cdot \rangle_V:V\times V\to E
\]
satisfying 
\[
\langle u,v\rangle_{V}=\epsilon\langle v,u\rangle^{\tau}_{V}\quad \langle au,v\rangle_V=a\langle u,v\rangle_{V},\quad a\in E,\quad u,v\in V,
\]
and $\langle \cdot,\cdot\rangle_{V}$ is called an $\epsilon$-\textbf{hermitian form} on $V$.
\end{defin}
\begin{defin}\label{def: GV} For a given  $\epsilon$-hermitian space $V$ over $E$, we denote by $\RU(V)$ the isometry group over $F$ preserving the $\epsilon$-hermitian form $\langle \cdot,\cdot\rangle_V$. We uniformly define the algebraic group $\RG_V$ over $F$ to be one of the following
\begin{enumerate}
    \item  $\RU(V)$;
    \item (Special cases) $\SU(V)$, the subgroup of $\RU(V)$ with determinant over $F$ equal to $1$ for $\epsilon=1$ and $E=F$, that is, 
    \[
    \SU(V)=\{g\in \RU(V): \mathrm{det}_{F}(g)=1\};
    \]
   \item (Metaplectic cases) $\wt{\RU}(V)$, which contains only one situation 
 $\Mp(V)$ ($\epsilon=-1$, $F=\BR$),
   the unique nonsplit central extension of $\Sp(V)$. 
\end{enumerate}
\end{defin}

 In the Bessel cases, the group $\RG_V$  are given in the following tableau ($\epsilon=+1$): 
\begin{center}
\begin{tabular}{ |c|c|c| c|c|} 
 \hline
 group& type &$F$& $E$&$\tau$\\
 \hline
 $\RU(p,q)$ &$\RU(V)$&$\BR$  &$\BC$  &$\tau_{\mathrm{conj}}$\\ 
 $\RO(p,q)$&$\RU(V)$ &$\BR$&$\BR$&$1_{\BR}$\\ $\SO(p,q)$&$\SU(V)$ &$\BR$&$\BR$&$1_{\BR}$\\ $\RO(2n,\BC)$&$\RU(V)$ &$\BC$&$\BC$&$1_{\BC}$\\$\SO(2n,\BC)$&$\SU(V)$ &$\BC$&$\BC$&$1_{\BC}$\\
 \hline
\end{tabular}
\end{center}

In Fourier-Jacobi cases, $\RG_V$  are given in the following tableau ($\epsilon=-1$):
\begin{center}
\begin{tabular}{ |c|c|c| c|c|} 
 \hline
 group& type &$F$& $E$&$\tau$\\
 \hline
 $\RU(p,q)$ &$\RU(V)$&$\BR$  &$\BC$  &$\tau_{\mathrm{conj}}$\\ $\Sp(2n,\BR)$&$\RU(V)$ &$\BR$&$\BR$ & $1_{\BR}$\\ $\Sp(2n,\BC)$&$\RU(V)$ &$\BC$&$\BC$&$1_{\BC}$\\ 
$\Mp(2n,\BR)$&$\wt{\RU}(V)$ &$\BR$&$\BR$&$1_{\BR}$\\ 
 \hline
\end{tabular}
\end{center}

Groups like $\SU(p,q)$ are not considered because the multiplicity-one theorem does not hold in these situations.

\subsubsection{Admissible pairs of Bessel type} 

Let  $(V,W)$ be a pair of nondegenerate $(+1)$-hermitian spaces over $E$. The pair $(V,W)$ is called \textbf{admissible of Bessel type} if and only if there exists an anisotropic line $D=E.z_0$  and an even-dimensional split nondegenerate quadratic space $Z$ such that
\[
V= W\oplus^{\perp} D\oplus^{\perp}Z.
\]
In particular,  for $r=\frac{\dim_{E}Z}{2}$, there exists an $E$-basis $\{z_{i}\}_{i=\pm 1}^{\pm r}$ of $Z$ such that
\[
\langle z_i,z_j\rangle_V=\delta_{i,-j}, \quad\forall  i,j\in \{\pm 1,\dots,\pm r\}
\]

\subsubsection{Admissible pair of Fourier-Jacobi types}
Let $(V, W)$ be a pair of nondegenerate $(-1)$-hermitian space over $E$. The pair $(V, W)$ is called \textbf{admissible of Fourier-Jacobi type} if there exists a $(-1)$-hermitian space $Z$ over $E$ such that 
\[
V=W\oplus^{\perp} Z\text{ or } W=V\oplus^{\perp} Z.
\] 

In particular, for $r=\frac{\dim_{E} Z}{2}$, there exists an $E$-basis $\{z_{i}\}_{i=\pm 1}^{\pm r}$ of $Z$  such that 
\[
\langle z_i,z_j\rangle_V=\sgn(i)\delta_{i,-j}, \quad\forall  i,j\in \{\pm 1,\dots,\pm r\}.
\]

 We denote $\CH(W)$ the \textbf{Heisenberg group} $\Res_{E/F}W\oplus \BG_a$ over $F$ with the multiplication rules given by
\begin{equation}\label{equ: Heisenberg}
(w,z)\cdot (w',z')=(z'w+zw',\frac{1}{2}\langle w,w'\rangle_W+zz'),\quad w,w'\in \Res_{
E/F}(W),\quad z,z'\in \BG_a.
\end{equation}

We set the \textbf{Jacobi group} $\RG_W^J=\RG_W\rtimes \CH(W)$, where  $\RG_W$ is of the same type as $\RG_V$ in Definition \ref{def: GV} and acts by conjugation on $\CH(W)$ by
\[
g(w,z)g^{-1}=(g.w,z)\quad g\in \RG_W,(w,z)\in \CH(W)
\]

\subsubsection{Bessel triple and Fourier-Jacobi triple}
For an admissible pair  $(V,W)$ (of Bessel or Fourier-Jacobi type), we define the triple $(\RG,\RH,\xi)$, where $\RG$ is a  linear algebraic group over $F$, $\RH$ is a subgroup of $\RG$  and $\xi$ is a unitary character of $\RH(F)$. 

There are three cases that we treat differently: the Bessel cases, which we abbreviate as \textbf{(B)}, and there are examples in \cite{jiang2010uniqueness}; the Fourier-Jacobi and $W\subsetneq V$ cases, which we abbreviate as \textbf{(FJ 1)}, and there are examples in \cite{liu2013uniqueness}; the Fourier-Jacobi and $V\subset W$ cases, which we abbreviate as \textbf{(FJ 2)}.

We define a parabolic group whose unipotent part $\RN$ is critical for defining $\RH$ and $\xi$.
\begin{itemize}
    \item In (B)  cases, we let $\RP_{V,r}=\RL_{V,r}\cdot \RN_{V,r}$ be the parabolic subgroup of $\RG_V$ stabilizing the  totally isotropic flag over
\[
X_1\subset X_2\subset \cdots \subset X_r
\]
where $X_k=\Span_E(z_1,\cdots,z_k)$, the linear span of $z_1,\cdots,z_k$ over $E$;
\item In (FJ 1) cases, we let $\RP_{V,r-1}=\RL_{V,r-1}\cdot \RN_{V,r-1}$ be the parabolic subgroup of $\RG_V$ stabilizing the totally isotropic flag
\[
X_1\subset X_2\subset \cdots \subset X_{r-1},
\]
and $H_1=\BC z_r\oplus \BC z_{-r}$;
\item In (FJ 2) cases, we let $\RP_{W,r}=\RL_{W,r}\cdot \RN_{W,r-1}$ be the parabolic subgroup of $\RG_W$ stabilizing \[X_1\subset X_2\subset \cdots \subset X_{r-1}\]
\end{itemize}

We set $\RN=\RN_{V,r},\RN_{V,r-1}$ in (B)(FJ 1) cases, respectively, and let $\RN^J=\RN_{W,r-1}\rtimes X_{r-1}$.

Then we  define $\RG$ and $\RH_{\mathrm{r}}$  as following
\begin{itemize}
    \item In (B) cases, we set $\RG=\RG_V\times \RG_W$ and we identify $\RN$ as a subgroup of $\RG$ via $
    \RG_V\hookrightarrow 1\times \RG_V$ and $\RH_r=\Delta \RG_W$, the image of the diagonal embedding $\RG_W\hookrightarrow \RG_W\times \RG_V$;
\item In (FJ 1) cases, we set $\RG=\RG_V\times \RG_W^J$  and we identify $\RN$ as a subgroup of $\RG$ via $\RG_V\hookrightarrow 1\times \RG_V$. We identify $\RG_W^J$ as the subgroup of $\RG_{W\oplus H_1}$ leaving $\BC z_k$ invariant and set  $\RH_{r}=\Delta \RG_W^J$, the image of the diagonal embedding $\RG_W^J\hookrightarrow  \RG_V\times \RG_W^J$; 
\item In (FJ 2) cases, we set $\RG=\RG_V\times \RG_W^J$  and we identify $\RN^J$ as a subgroup of $\RG$ via $\RG_W^J\hookrightarrow 1\times \RG_W^J$ and  set $\RH_{r}=\Delta \RG_V$, the image of the diagonal embedding $\RG_V\hookrightarrow \RG_V\times (\RG_W\rtimes 1)\subset  \RG_V\times (\RG_W\rtimes \CH(W))$.
\end{itemize}

In these cases, $\RH_r=\Delta \RG_W$, $\Delta \RG_W^J$, $\Delta \RG_V$ acts on $\RN,\RN,\RN^J$ by adjoint action, respectively, and we set
\[
\RH=
\begin{cases}
\Delta \RG_W\rtimes \RN, \text{ in (B) cases,}\\
\Delta \RG_W^J\rtimes \RN, \text{ in  (FJ\ 1) cases,}\\
\Delta \RG_V\rtimes \RN^J, \text{ in (FJ\ 2) cases.} 
\end{cases}
\]

 Define  morphisms $\lambda_\RN:\RN\to \Res_{E/F}\BG_a$, $\lambda_{\RN^J}:\RN\to \Res_{E/F}\BG_a$ via
\[
\lambda_\RN(n)=\begin{cases}
    \sum_{i=0}^{r-1}\langle z_{-i},nz_{i+1}\rangle,  \text{ in (B) cases,}\\
    \sum_{i=1}^{r-1}\langle z_{-i},nz_{i+1}\rangle  \text{ in (FJ 1) cases,} 
    \end{cases}\quad n\in \RN; 
\]
\[
\lambda_{\RN^J}(n\rtimes (z_W\oplus c))=    \sum_{i=1}^{r-1}\langle z_{-i},nz_{i+1}\rangle  +\langle z_{-r},z_W\rangle, \quad n\in \RN,\quad z_W\oplus c\in \CH_W=\Res_{E/F}W\oplus \BG_a.
\]

The characters $\lambda_\RN,\lambda_{\RN},\lambda_{\RN^J}$ are $\Delta \RG_W, \Delta \RG_W^J,\Delta \RG_V$-conjugation invariant and hence $\lambda_N,\lambda_N,\lambda_{\RN^J}$ admit unique extensions to $\RH$ trivial on $\RH_{r}$, which is  denoted by $\lambda_{\RH}$. Let $\lambda_{\RH,F}: \RH(F)\to \Res_{E/F}\BG_a(F)=E$ be the induced morphism on $F$-rational points. We define an unitary character of $\RH(F)$ by
\[
\xi(h)=\psi_E(\lambda_{\RH,F}(h)),\quad h\in \RH(F),
\]
where $\psi_E$ is a fixed additive (unitary) character 
\begin{equation}\label{equ: psiE}
\psi_E(x)=\psi_F(\frac{1}{|\Gal(E/F)|}\sum_{\tau\in \Gal(E/F)}x^{\tau}),\quad x\in E.
\end{equation}
\begin{defin}
\begin{enumerate}
    \item 
 In (B) cases, the triple $(\RG, \RH,\xi)$ is called the \textbf{Bessel triple} associated to the admissible pair $(V, W)$ of Bessel type;
\item In (FJ 1)(FJ 2) cases, the triple $(\RG, \RH,\xi)$ is called the \textbf{Fourier-Jacobi triple} associated to the admissible pair $(V, W)$ of  Fourier-Jacobi type.
\end{enumerate}
\end{defin}
\begin{defin} (Basic admissible pair and basic triples)
\begin{enumerate}
    \item In (B) cases, we call an admissible pair $(V, W)$ or the associated Bessel triple $(\RG,\RH,\xi)$ is \textbf{codimension-one} if $\dim_{E} V=\dim_{E} W+1$;
    \item In (FJ 1) cases, we call an admissible pair  $(V, W)$ or the associated Fourier-Jacobi triple $(\RG,\RH,\xi)$ is \textbf{equal-rank} if $\dim_{E} V=\dim_{E} W$;
    \item In  (FJ 2) cases, we call an admissible pair $(V, W)$ or the associated Fourier-Jacobi triple $(\RG,\RH,\xi)$ \textbf{almost equal-rank} if $\dim_{E} V=\dim_{E} W+2$.
\end{enumerate}
In all these cases, we say the admissible pair $(V, W)$ and the triple $(\RG,\RH,\xi)$ are \textbf{basic}.
\end{defin}

\subsection{Representations and multiplicities}\label{section: subgroup}
This section specifies the category of representations we work on and defines some functors.

\subsubsection{Fr\'echet representations of moderate growth}  In the article, we work on almost linear Nash groups  and Nash manifolds (\cite{aizenbud2008schwartz}\cite{sun2015almost}). In particular, 
\begin{enumerate}
    \item When $\RG$ is a linear algebraic group over an Archimedean local field $F$,  $\RG(F)$ has a structure as an almost linear Nash group whose topology is compatible with the topology on $\RG(F)$ as a Lie group. All Nash groups treated in this article are in this form;
    \item For a closed algebraic subgroup $\RH$ of $\RG$, the homogeneous space  $\RG(F)/\RH(F)$ is a Nash manifold with Nash $\RG(F)$-action. All Nash manifolds in this article are in this form or one of its connected components.
\end{enumerate}

\begin{defin}
An almost linear Nash group $G$ is called  a \textbf{metaplectic cover} of another almost linear Nash group $\underline{G}$, if there is an exact sequence
\[
0\to \{\pm 1_G\}\to G\to \underline{G}\to 1,
\]
where $\{\pm 1_G\}$ is an order-two subgroup of the center of $G$.
\end{defin}
\begin{ex}
   \begin{enumerate}
        \item Let $V$ be a symplectic space over $\BR$, then $\Mp(V)(\BR)$ is a metaplectic cover over $\Sp(V)(\BR)$;
        \item Let $V$ be a vector space over $\BR$, we denote by $\wt{\GL}(V)$ the algebraic group on $\GL(V)\times \{\pm 1\}$ with  the multiplication rule given by
      \[
        (g_1,\delta_1)\cdot (g_2,\delta_2)=(g_1g_2,\delta_1\delta_2(\det(g_1),\det(g_2))_{\BR}),\quad g_1,g_2\in \GL(V),\quad \delta_1,\delta_2\in \{\pm 1\},
       \]
where $(\cdot,\cdot)_{\BR}$ is the Hilbert symbol. Then $\wt{\GL}(V)(\BR)$ is a metaplectic cover of $\GL(V)(\BR)$.
    \end{enumerate} 
\end{ex}
For an almost linear Nash group $G$, we work in the category of smooth Fr\'echet  $G$-representations of moderate growth and denote it by $\Pi_{\FMG}(G)$. When $G$ is a metaplectic cover of $\udl{G}$, a representation $\pi\in \Pi_{\FMG}(G)$ is called \textbf{genuine} if it has nontrivial restriction of the central character to $\{\pm 1_G\}$.  Non-genuine representations in $\Pi_{\FMG}(G)$ always factor through
 representations in $\Pi_{\FMG}(\underline{G})$. 
 When $G$ is a  metaplectic cover, we denote by $\Pi_{\FMG}^{\gen}(G)$ the category  of genuine Fr\'echet  $G$-representations of moderate growth.

\begin{defin}\label{defin: boxtimes}
\begin{enumerate}
    \item For almost linear Nash groups $G_1$ and $G_2$, and $\pi_1\in\Pi_{\FMG}(G_1)$ and $\pi_2\in \Pi_{\FMG}(G_2)$, we denote by $\pi_1\boxtimes \pi_2$ the  $G_1\times G_2$-representation  on the projective tensor product $V_{\pi_1}\wh{\otimes} V_{\pi_2}$ of the underlying Fr\'echet spaces $V_{\pi_1},V_{\pi_2}$ of $\pi_1$ and $\pi_2$. 
    \item When both $G_1$ and $G_2$ are metaplectic covers, we define a metaplectic cover
    \[G_1\times_{\pm 1}G_2=G_1\times G_2/\{(-1_{G_1})\times 1_{G_2}-1_{G_1}\times (-1_{G_2}) \}\]
      of $\udl{G}_1\times\udl{G}_2$. Let $\pi_1\in \Pi_{\FMG}^{\gen}(G_1)$, $\pi_2\in \Pi_{\FMG}^{\gen}(G_2)$, the $G_1\times G_2$-representation $\pi_1\boxtimes \pi_2$ factors through a representation in $\Pi_{\FMG}^{\gen}(G_1\times_{\pm 1} G_2)$, we denote it by $\pi_1\boxtimes_{\pm 1}\pi_2$.
    \item For an almost linear Nash group $G$ and $\pi_1,\pi_2\in \Pi_{\FMG}(G)$, we denote by $\pi_1\otimes \pi_2$ the $G$-representation obtain from the $G\times G$-representation $\pi_1\boxtimes \pi_2$ via the diagonal embedding $G\hookrightarrow G\times G$. In particular, when $G$ is a metaplectic cover of $\udl{G}$ and $\pi_1,\pi_2\in\Pi_{\FMG}^{\gen}(G)$, the $G$-representation $\pi_1\otimes \pi_2$ descents to a $\underline{G}$-representation, we  denote it by $\pi_1\otimes_{\pm 1} \pi_2$.
\end{enumerate}
\end{defin}

For an algebraic group $\RG$ over $F$, we denote by $\RG(F)$   the almost linear Nash group on the $F$-points of $\RG$. We use the notion $\Pi_{\FMG}(\RG)$ instead of $\Pi_{\FMG}(\RG(F))$ for simplification.

\subsubsection{Classification theorem in  $\Pi_{\FMG}^{\mathrm{genu}}(\wt{\GL}(V))$}

The following lemma classifies genuine representations of $\wt{\GL}(V)$.
\begin{lem}\label{lem: classification of representations}
   Let $F=\BR$. There is a categorical equivalence
 \[
        \begin{aligned}
            \Pi_{\FMG}(\GL(V))&\to \Pi_{\FMG}^{\gen}(\wt{\GL}(V))\\
            \sigma&\mapsto\tilde{\sigma} =\sigma\otimes (\chi_{\psi_F}\circ \wt{\det})
        \end{aligned}
\]
    where $\chi_{\psi}$ is a fixed character of $\wt{\GL}_1(\BR)$ given as \cite[\S 1]{adams1998genuine} and $\wt{\det}:\wt{\GL}(V)\to \wt{\GL}_1$ is given by
    \[
    \wt{\det}(g,\delta)=(\det(g),\delta).
    \]
Its inverse is\[
        \begin{aligned}
            \Pi_{\FMG}^{\gen}(\wt{\GL}(V))&\to \Pi_{\FMG}({\GL}(V))\\
            \tilde{\sigma}&\mapsto\sigma=\wt{\sigma}\otimes_{\pm 1} (\chi_{\psi_F}^{-1}\circ\wt{\det}).
        \end{aligned}\]
\end{lem}

\subsubsection{Weil representations}
For a $(-1)$-hermitain space $V$ over $E$, we  recall the explicit formulas in the mixed model of the Weil representations $\omega_{V,\psi_F}$ of  
\[
\begin{cases}
\wt{\Sp}(V)\rtimes \CH(V)& \text{ when } E=F,\\
\RU(V)\rtimes \CH(V) &\text{ when }E\neq F.
\end{cases}\]

We fix a decomposition 
\[
V=V_0\oplus ^{\perp}(X\oplus Y),
\]
where $X, Y$ are totally isotropic spaces over $E$ and $Y=X^*$.

The Weil representation  $\omega_{V,\psi_F}$ can be realized as the space  on 
\[
\omega_{V_0,\psi_F}\wh{\otimes}\CS(Y(E))
\]  
with  the action $\omega_{V,\psi_F}$ of $\RG_V\rtimes\CH(V)$ in (\ref{equ: Weil Heisenberg})(\ref{equ: mixed Weil}). In particular, when $V_0=0$, $\omega_{V_0,\psi_F}$ is the one dimension representation $\psi_F$, and this realization gives the Schrodinger's model of $\omega_{V,\psi_F}$. The Schrodinger model can be taken as a definition for the Weil representation $\omega_{V,\psi_F}$.

The action $\rho_{V,\psi_F}=\omega_{V,\psi_F}|_{\CH(V)}$ on $\omega_{V_0,\psi_F}\wh{\otimes}\CS(Y(E))$ is  the following
\begin{equation}\label{equ: Weil Heisenberg}
\rho_{V,\psi_F}(x+y+v_0,z)\varphi_{V_0}\otimes f_Y(y')=\psi_E(z+\langle x,y'\rangle_V+\frac{1}{2}\langle x,y\rangle_V)\rho_{V_0,\psi_F}(v_0)\varphi_{V_0}\otimes  f_Y(y+y')\end{equation}
for  $x\in X(E)$, $y,y'\in Y(E)$, $z\in F$
where $\psi_E$ is defined in (\ref{equ: psiE}), $\varphi_{V_0}\in \omega_{V_0,\psi_F}$ and $f_X\in \CS(X(E))$.

Denote by $\RP_{V,X}$ be the parabolic subgroup of $\RG_V$ stabilizing $X$. The parabolic subgroup $\RP_{V,X}$ has Levi decomposition $\RP_{V,X}=\RL_{V,X}\rtimes \RN_{V,X}$ with 
\[
\RL_{V,X}\cong\begin{cases}
    \Res_{E/F}\GL(X)\times \RG_{V_0}&\text{when }E=\BC,\\
    \wt{\GL}(X)\times_{\pm 1} \RG_{V_0} &\text{when }E=\BR,
\end{cases} 
\] 
and $\RN_V=\Hom_E(V_0,X)\rtimes\mathrm{Herm}(X,Y)$, where $\mathrm{Herm}(X,Y)$ is the additive group of hermitian forms on $X^*\otimes Y$.

 We choose basis $\{z_i\}_{i=1}^r$ of $X$ and $\{z_{-i}\}_{i=1}^r$ as in Section \ref{section: Bessel Fourier-Jacobi}. We denote by $w_{X,Y}$ the element in $\RG_V$ satisfying
\[
w_{X,Y}z_i=z_{-i},\quad w_{X,Y}z_{-i}=-z_{i},\quad i=1,\cdots,r.
\]

The action $\pi_{V,\psi_F}=\omega_{V,\psi_F}|_{\RG_V}$ is generated by the following actions
\begin{equation}\label{equ: mixed Weil}
\begin{aligned}
\pi_{V,\psi_F}(c)\varphi_{V_0}\otimes f_Y(y')&=\psi_F(\frac{1}{2}\langle cy',y'\rangle)\varphi_{V_0}\otimes f_Y(y'), \quad c\in \mathrm{Herm}(X,Y)\subset \RN_V(F);\\
\pi_{V,\psi_F}(b)\varphi_{V_0}\otimes f_Y(y')&=\rho_{V_0,\psi_F}(b^*y')\varphi_{V_0}\otimes f_Y(y'),\quad b\in \mathrm{Hom}_E(V_0,X)\subset \RN_V(F);\\
\pi_{V,\psi_F}(a)\varphi_{V_0}\otimes f_Y(y')&=
|\det(a)|^{\frac{\dim_E Y}{2}}\varphi_{V_0}\otimes f_Y(a^*y'),\quad a\in \GL(X)\subset \RL_{V}(F);\\
\pi_{V,\psi_F}(g_{V_0})\varphi_{V_0}\otimes f_Y(y')&=\pi_{V_0,\psi_F}(g_{V_0})\varphi_{V_0}\otimes f_Y(y'),\quad g_{V_0}\in\RG_{V_0}(F)\subset \RL_{V}(F);\\
\pi_{V,\psi_F}(w_{X,Y})\varphi_{V_0}\otimes f_Y(y')&=\gamma \cdot\varphi_{V_0}\otimes\int_{X(E)}\psi_F(\langle y',x\rangle_V) f_{Y}(-w_{X,Y}x)dx.\\
\end{aligned}
\end{equation}

Here  $\gamma$ is a constant satisfying 
\[
\gamma^2=\omega_{E/F}(-1)^{\dim X},
\]
where $\omega_{E/F}$ is the quadratic Hecke character.

\subsubsection{Coinvariants} Then we compute some coinvariants. See \cite[Lemma 3.55]{gan2023theta} for the non-Archimedean counterpart.

\begin{defin}\label{def: Hausdorff coinvariant}(Coinvariants)
Let $G$ be an almost linear Nash group with a character $\chi_G:G\to \BC^{\times}$. Suppose $\pi\in\Pi_{\FMG}(G)$. 

\begin{enumerate}
    \item We denote by 
\[
\pi_{(G,\chi_G)}=\pi/\sum_{\substack{g\in G\\v\in V_{\pi}}}(\pi(g)-\chi_G(g))v.
\]
When  $\chi_G$ is the trivial character, we simply it as $\pi_G$ and we also simply $\pi_{\RG(F)}$ and $\pi_{\RG}$.
\item We denote by $\pi_{(G,\chi_G)}^{\mathrm{Haus}}$ the maximal Hausdorff quotient of $\pi_{G,\chi_G}$, which is equal to 
\[
\pi_{(G,\chi_G)}=\pi/\overline{\sum_{\substack{g\in G\\v\in V_{\pi}}}(\pi(g)-\chi_G(g))v}.
\] 
We simply it as $\pi_{\RG}^{\mathrm{Haus}}$ when $\chi_G$ is the trivial character.
\end{enumerate}
\end{defin}

Let $V$ be a $(-1)$-hermitian space with decomposition $V=V_0\oplus^{\perp}(X\oplus Y)$. 
Let $r=\dim_E X$. 
Recall $\RN_{V,r}$ as the unipotent radical of parabolic subgroup of $\RG_V$ stabilizing  the filtration
\[
X_1\subset \cdots\subset X_r=X,
\]
where $X_i$ are defined as in Section \ref{section: Bessel Fourier-Jacobi}  with $\dim_E X_i=i$.
\begin{lem}\label{lem: coinvariant}
Let $\xi_X$ be a character of $X(E)$ such that  $\xi_X(x)=\psi_E(\langle x,y_0\rangle)$, then
\[
(\omega_{V,\psi_F})_{(X(E),\xi_X)}=\omega_{V_0,\psi_F}.
\]
Moreover,
\begin{enumerate}
    \item When $y_0=0$, $\CH(V_0)(F)$ acts on it  and $\RN_{V,r-1}(F)$ acts trivially on it.
\item When $y_0=z_{-k}$, $\CH(V_0)(F)$ as  $\rho_{V_0,\psi_F}$ and $\RN_{V,r-1}(F)$ acts trivially on it.
\end{enumerate}
\end{lem}
\begin{proof}
We denote by $\Re(\langle ,\rangle_V)=\frac{1}{|\Gal(E/F)|}\sum_{\tau\in \Gal(E/F)}\langle,\rangle_V^{\tau}$, then 
\[
\psi_E(\langle ,\rangle_V)=\psi_F(\Re\langle ,\rangle_V).
\]

Let $V=X_V\oplus Y_V$ for totally isotropic $X_V,Y_V$ such that $X\subset X_V$ and $Y\subset Y_V$. From (\ref{equ: Weil Heisenberg}), we have 
\[
((\rho_{V,\psi_F}(x)f)(y)=\psi_F(\Re\langle x,y \rangle_V)f(y),\quad y\in Y_V(F)
\]

Hence,
\[
((\rho_{V,\psi_F}(x)f-\xi(x)f)(y)=\psi_F(\Re\langle x,y-y_0 \rangle_V)f(y),\quad y\in Y_V(F)
\]

Denote by $\wh{f}$ the Fourier-transform of $f$ with respect to $\psi_F, \Re\langle, \rangle_V$, then the above equation can be written as
\begin{equation}\label{equ: trans on wh}
(\wh{\rho}_{V,\psi_F}(x)\wh{T_{y_0}f})(\lambda)=\wh{T_{y_0}f}(x+\lambda)=(T_x\wh{T_{y_0}f})(\lambda),
\end{equation}
where $T_{y_0},T_x$ are the translation operator.

Notice that $\wh{T_{y_0}\CS}(Y(E))=\wh{\CS}(Y(E))=\CS(X(E))$, and, from \cite[Theorem 1.4]{chen2020schwartz}, 
\[
\dim \CS(X(E))_{X(E)}=\dim\CS(X(E))/\sum_{x\in X(E)}(T_x-T_0).\CS(X(E))=1.
\]
Hence, from (\ref{equ: trans on wh}), $\dim\CS(Y(E))_{(X(E),\xi_X)}=\dim \CS(X(E))_{X(E)}=1$ and  the evaluation map at $y=y_0$ gives a nonzero $X(E)$-invariant map from $\CS(Y(E))$ to $\BC$, so it is a generator of $\CS(Y(E))_{(X(E),\xi_X)}$. Therefore,
\[
(\omega_{V,\psi_F})_{X(E)}=\CS(Y_V(E
))_{X(E)}=(\CS(Y(E)))_{X(E)}\wh{\otimes }\omega_{V_0,\psi_F}=\omega_{V_0,\psi_F},
\]
Recall the following two equations in (\ref{equ: mixed Weil}), 
\[\pi_{V,\psi_F}(c)\varphi_{V_0}\otimes f_Y(y')=\psi_F(\frac{1}{2}\langle cy',y'\rangle)\varphi_{V_0}\otimes f_Y(y'), \quad c\in \mathrm{Herm}(X,Y)\subset \RN_V(F)
\]
\[
\pi_{V,\psi_F}(b)\varphi_{V_0}\otimes f_Y(y')=\rho_{V_0,\psi_F}(b^*y')\varphi_{V_0}\otimes f_Y(y'),\quad b\in \mathrm{Hom}_E(V_0,X)\subset \RN_V(F)
\]
By applying the evaluation map $y'=y_0$, we obtain the  action of $\CH(V_0)(F)$ and  $\RN_{V,r-1}(F)$ when $y_0=0$ or $z_{-k}$ as stated in the lemma.
\end{proof}

We denote by $\xi$ the character in the triple associated to the admissible pair $(V_0,V)$ of (FJ 2) type and define $\RN_{V,r-1}$ as in Section \ref{section: Bessel Fourier-Jacobi}.

\begin{lem}\label{lem: coinvariant tensor weil}
\[
(\omega_{V,\psi_F}\otimes \omega_{V,\psi_F^{-1}})_{(\CH(V),1_{\CH(V)})}=\BC.1
\]
and $\RG_V$ acts trivially on it.
\end{lem}
\begin{proof}
Let $V=X_V\oplus Y_V$, where $X_V,Y_V$ are totally isotropic spaces. Using the Schrodinger's model, we can
 realize $\omega_{V,\psi_F}$  on $\CS(Y_1(E))$ and  $\omega_{V,\psi_F^{-1}}$ on $\CS(Y_2(E))$, where $Y_1=Y_2=Y_V$. Then 
$\omega_{W,\psi_F}\wh{\otimes} \omega_{W,\psi_F^{-1}}$
can be realized on $\CS(Y_1(E)\oplus Y_2(E))$ with action
\[
x.f(y_1,y_2)=\psi_F(\frac{1}{2}\langle x,y_1-y_2\rangle_V)f(y_1,y_2)\quad x\in X(E),
\]
\[
y.f(y_1,y_2)=f(y_1+y,y_2+y)\quad y\in Y(E),
\]
\[
z.f(y_1,y_2)=f(y_1,y_2), z\in  E,
\]
where $y_1\in Y_1(E), y_2\in Y_2(E)$.

Take $f$ as a function on variable $y^-=y_1-y_2\in Y^-(E)$ and $y^+=y_1+y_2\in Y^+(E)$, where $Y^+=Y^-=Y$, then $x\in X(E)$ acts on $\CS(Y^-(E)\oplus Y^+(E))$ as multiplication by $\psi_F(\frac{1}{2}\langle x,y^-\rangle)$ and $y\in Y(E)$ acts on $\CS(Y^-(E)\oplus Y^+(E))$ as translation on $y^+$ by $2y$.

Then 
\[
(\omega_{W,\psi_F}\wh{\otimes} \omega_{W,\psi_F^{-1}})_{\CH(V)}=\CS(Y^-(E)\oplus Y^+(E))_{X(E)\oplus Y(E)}=\CS(Y^-(E))_{X(E)}\wh{\otimes}\CS(Y^+(E))_{Y(E)}.
\]

As in the proof for Lemma \ref{lem: coinvariant}, we have
\[
\dim\CS(Y^-(E))_{X(E)}=\dim\CS(Y^+(E))_{Y(E)}=1.
\]

A nonzero invariant map can be constructed on $\CS(Y_1(E))\wh{\otimes}\CS(Y_2(E))$ as 
\[
f_1\otimes f_2\mapsto \int_{Y(E)}f_1(y)f_2(y)dy.
\]

Using  (\ref{equ: mixed Weil}), one can verify that $\RG_V$ acts trivially on it.
\end{proof}

\subsubsection{Classification of representations of Jacobi groups}
 Let $G=\RG_W^J(F)
 $ be the Jacobi group for nondegenerate $(-1)$-hermitian form $W$ over $E$. (Section \ref{section: Bessel Fourier-Jacobi})
 \begin{defin}\label{defin: nondegenerate}
\begin{enumerate}
    \item We denote by $\Pi_{\FMG}^{\mathrm{irr},\mathrm{nd}}(G)$ all the representations in $\Pi^{\mathrm{irr}}_{\FMG}(G)$ with the central character nontrivial  on \[F=Z(\CH(W)(F))\subset Z(\RG_W^J(F)).\] We call these representations  \textbf{nondegenerate} irreducible Casselman-Wallach representations.
    \item For fixed additive character $\psi_F$ of $F$, we denote by $\Pi_{\FMG}^{\mathrm{irr},\psi_F}(G)$ representations in  $\Pi_{\FMG}^{\mathrm{irr},\mathrm{nd}}(G)$ with the central character equal to $\psi_F$  on $Z(\CH(W)(F))$. 
\end{enumerate}    
 \end{defin}

There is a classification theorem on non-degenerate Casselman-Wallach representations of the Jacobi groups $\RG_V^J=\RG_V\rtimes \CH(V)$.
\begin{lem}\label{lem: classification of representations 2}
 There  is a categorical equivalence
   \[
    \begin{aligned}
\CF_{\RG_V}:\Pi_{\FMG}^{\mathrm{irr}}(\RG_V)&\to \Pi_{\FMG}^{\mathrm{irr},\psi_F}(\wt{\RG}_V\rtimes \CH(V))\\
    \pi_V&\mapsto \t{\pi}_V^J=\pi_V\otimes \omega_{V,\psi_F}
    \end{aligned}    
    \]    
when $\RG_V$ is not metaplectic and a categorical equivalence
    \[
 \begin{aligned} \CF_{\RG_V}:\Pi_{\FMG}^{\mathrm{irr}}(\RG_V)&\to \Pi_{\FMG}^{\mathrm{irr},\psi_F}(\wt{\RG}_V\rtimes \CH(V))\\
    \pi_V&\mapsto \wt{\pi}_V^J=\pi_V\otimes_{\pm 1} \omega_{V,\psi_F}
    \end{aligned}
    \]
when $\RG_V$ is metaplectic.
Here $\omega_{V,\psi_F}$ is the Weil representation of $\wt{\Sp}(V)\rtimes \CH(V)$, $\wt{\GL}(V)\rtimes \CH(V)$, or $\RU(V)\rtimes \CH(V) $ associated to the character $\psi_F$.

     The inverse map $\CG_{\RG_V}$ is given by
    \[
\CG(\wt{\pi}_V^J)=\Hom_{\CH(V)}(\omega_{V,\psi_F},\pi_V^J)
    \]
\end{lem}
\begin{proof}
By Stone–von Neumann theorem, $\Pi_{\FMG}^{\mathrm{irr},\psi_F}(\wt{\RG}_V\rtimes \CH(V))$ consists of $(\CH(V),\omega_{V,\psi_F})$-isotypic representations. For the verification of analytic properties, we refer to \cite{sun2012representations} and \cite[\S 2]{xuefourier}, which elaborate on all cases.
\end{proof}

\subsubsection{(Twisted) Jacquet modules of $\RG_V$, $\RG_V^J$}
We recall the definition of Jacquet and twisted Jacquet modules for reductive groups $\RG_V$ in \cite{gomez2017generalized} and define those for nondegenerate representations of $\RG_W^J$.

Let $(V,W)$ be an admissible pair of (B) or (FJ 1) type. Let $r=\frac{\dim_E V-\dim_E W}{2}$. Recall that  $\RP_{V,r}$ is the parabolic subgroup of $\RG_V$ stabilizing the flag $X_1\subset X_2\subset \cdots \subset X_{r}$ with Levi decomposition $\RP_{V,r}=\RL_{V,r}\rtimes \RN_{V,r}$. Then we define Jacquet modules as $\RG_W(F)$-representations as following.
\begin{defin}
\begin{enumerate}
    \item  For $\pi_V\in \Pi_{\FMG}(\RG_V)$, we set the Jacquet module
    \[
\Jac_{\RP_{V,r}}^0(\pi_V)=(\pi_V)_{(\RN_{V,r},1_{\RN_{V,r}})}^{\mathrm{Haus}}=\pi_V/\overline{\sum_{\substack{n\in \RN_{V,r}(F)\\v\in V_{\pi_V}}}(\pi_V(n)-1)v}.
\]
\item For $\pi_V\in \Pi_{\FMG}(\RG_V)$, we set the twisted Jacquet module 
\[
\Jac_{\RP_{V,r}}^{\CO_{r}}(\pi_V)=(\pi_V)^{\mathrm{Haus}}_{(\RN_{V,r},\xi^{-1}|_{\RN_{V,r}})}=\pi_V/\overline{\sum_{\substack{n\in \RN_{V,r}(F)\\v\in V_{\pi_V}}}(\pi_V(n)-\xi^{-1}(n))v}.
\]
\end{enumerate}
\end{defin}

Let $(V, W)$ be an admissible pair of (FJ 2) type. Let $r=\lfloor\frac{\dim_E W-\dim_E V}{2}\rfloor$. Denote by  $\RP_{W,r}^J$  the sub subgroup of $\RG_W^J$ stabilizing the flag $X_1\subset X_2\subset \cdots \subset X_r$. Then $\RP_{W,r}^J=\RP_{W,r}\rtimes \CH(X_r^{\perp})$, and it has  decomposition $\RP_{W,r}^J=\RL_{W,r}^J\rtimes \RN_{W,r}^J$, where $\RL_{W,r}^J=\RL_{W,r}\rtimes \CH(V)$, $\RN_{W,r}^J=\RN_{W,r}\rtimes \Res_{E/F}X_r$ (see Definition \ref{defin: parabolic induction}). Then we define the twisted Jacquet model as $\RG_V(F)$-representation as following.
\begin{defin}
\begin{enumerate}
    \item  For $\pi_W^J\in \Pi_{\FMG}(\RG_W^J)$, we set the Jacquet module
    \[
\Jac^0_{\RP_{W,r}^J}(\pi_W^J)=(\pi_W^J)_{(\RN_{W,r}^J,1_{\RN_{W,r}^J})}^{\mathrm{Haus}}=\pi_W^J/\overline{\sum_{\substack{n\in \RN_{W,r}^J(F)\\v\in V_{\pi_W^J}}}(\pi_W^J(n)-1)v}
\]
\item For $\pi_W^J\in \Pi_{\FMG}(\RG_W^J)$, we set the twisted Jacquet module  as 
\[
\Jac_{\RP_{W,r}^J}^{\CO_r}(\pi_W^J)=(\pi_W^J)^{\mathrm{Haus}}_{(\RN_{W,r}^J,\xi^{-1}|_{\RN_{W,r}^J})}=\pi_W^J/\overline{\sum_{\substack{n\in \RN_{W,r}^J(F)\\v\in V_{\pi_W^J}}}(\pi_W^J(n)-\xi^{-1}(n))v}
\]
\end{enumerate}
\end{defin}

\begin{lem}\label{lem: Jac of Jaco}
Let $\pi_W\in \Pi_{\FMG}(\RG_W)$, we have
  \begin{enumerate}
      \item 
      \[\Jac^0_{\wt{\RP}_{W,r}^J}(\pi_W\otimes \omega_{W,\psi_F})=\Jac^0_{\RP_{W,r}}(\pi_{W})\otimes \omega_{V,\psi_F}|_{\RG_W} \]
      \item 
      \[\Jac^{\CO_r}_{\wt{\RP}_{W,r}^J}(\pi_W\otimes \omega_{W,\psi_F})=(\omega_{V,\psi_F}\otimes\Jac_{\wt{\RP}_{W,r-1}}(\wt{\pi}_W)|_{\RG_V^J})_{\CH(V)}\]
  \end{enumerate}  
\end{lem}
\begin{proof}
For Point (1), since $X_r(E)$ acts trivially on $\pi_W$,
\[
(\pi_W\otimes \omega_{W,\psi_F})_{\RN_{W,r}^J\rtimes X_r}=(\pi_W|_{\RN_{W,r}}\otimes (\omega_{W,\psi_F})_{X_r})_{\RN_{W,r}}.
\]

From Lemma \ref{lem: coinvariant}(1), $(\omega_{W,\psi_F})_{X_r}=\omega_{V,\psi_F}$ and $\RN_{W,r}(F)$ acts trivially on it, then
\[
(\pi_W|_{\RN_{W,r}^J}\otimes (\omega_{W,\psi_F})_{X_r})_{\RN_{W,r}}=\Jac^0_{\RP_{W,r}}(\pi_{W})\otimes \omega_{V,\psi_F}|_{\RG_W}.
\]

Point (2) can be computed similarly from Lemma \ref{lem: coinvariant}(2).
\end{proof}
\subsubsection{Relation between (FJ 1) and (FJ 2)}
With the classification and the computation of Jacquet modules, we show that for $V\subset W$, the (FJ 2) cases for $(V,W)$ with $\RG=\RG_V\times \RG_W^J$ is isomorphic to the (FJ 1) cases (or equal-rank (FJ 2) cases) for $(W,V)$ with $\wt{\RG}=\wt{\RG}_W\times \wt{\RG}_V^J$ using the following lemma. 
\begin{lem}\label{lem: isomorphism between FJ1 and FJ2}
For $\pi_V\in \Pi_{\FMG}(\RG_V)$ and $\wt{\pi}_W\in \Pi_{\FMG}(\wt{\RG}_W)$, we have
\[ \Hom_{\RH}(\pi_V\boxtimes(\wt{\pi}_W\otimes \omega_{W,\psi_F}),\xi)=\Hom_{\wt{\RH}}(\wt{\pi}_W\boxtimes(\pi_V\otimes \omega_{V,\psi_F}),\wt{\xi})
\]      
\end{lem}
\begin{proof}
When $V=W$, the result is straightforward.  When $V\subsetneq W$, \[
\RH=\Delta{\RG_V}\rtimes (\RN\rtimes X),\quad \wt{\RH}=\Delta \wt{\RG}_V^J\rtimes \wt{\RN}.
\]

By definition, we have
\[
\Hom_{\RH}(\pi_V\boxtimes(\wt{\pi}_W\otimes \omega_{W,\psi_F}),\xi)=\Hom_{\Delta{\RG_V}}(\pi_V\boxtimes\Jac^{\CO_r}_{\wt{\RP}_{W,r}^J}(\pi_W\otimes \omega_{W,\psi_F}),1_{\Delta\RG_V}), \text{ and}
\]
\[\Hom_{\wt{\RH}}(\wt{\pi}_W\boxtimes(\pi_V\otimes \omega_{V,\psi_F}),\wt{\xi})=\Hom_{\Delta{\RG_V^J}}(\Jac_{\wt{\RP}_{W,r-1}}(\wt{\pi}_W)\boxtimes (\pi_V\otimes \omega_{V,\psi_F}),1_{\Delta{\RG_V^J}}).
\]

On the one hand, we can compute $\Jac^{\CO_r}_{\RP_{W,r}^J}(\pi_W\otimes \omega_{W,\psi_F})$ using Lemma \ref{lem: Jac of Jaco}, and obtain
\[
\Hom_{\RH}(\pi_V\boxtimes(\wt{\pi}_W\otimes \omega_{W,\psi_F}),\xi)=\Hom_{\RG_V}((\omega_{V,\psi_F}\otimes\Jac_{\wt{\RP}_{W,r-1}}(\wt{\pi}_W)|_{\RG_V^J})_{\CH(V)}\otimes \pi_V,1_{\RG_V})
\]

On the other hand, from Lemma \ref{lem: classification of representations 2} and Lemma \ref{lem: coinvariant tensor weil}, we have
\[
\begin{aligned}
&\Hom_{\Delta{\RG_V^J}}(\Jac_{\wt{\RP}_{W,r-1}}(\wt{\pi}_W)\boxtimes (\pi_V\otimes \omega_{V,\psi_F}),1_{\Delta{\RG_V^J}})\\
=&\Hom_{\RG_V}(\Hom_{\CH(V)}(\omega_{V,\psi^{-1}_F},\Jac_{\wt{\RP}_{W,r-1}}(\wt{\pi}_W)|_{\RG_V^J})\otimes \pi_V,1_{\RG_V})\\
=&\Hom_{\RG_V}((\omega_{V,\psi_F}\otimes\Jac_{\wt{\RP}_{W,r-1}}(\wt{\pi}_W)|_{\RG_V^J})_{\CH(V)}\otimes \pi_V,1_{\RG_V}).
\end{aligned}
\]
This completes the proof.
\end{proof}

\subsubsection{Casselman-Wallach representations}
In the Archimedean local theory of automorphisms, we study a smaller class of representations, Casselman-Wallach representations.  In general, the induction or restriction of a Casselman-Wallach representation may not be Casselman-Wallach, so we must return to the category $\Pi_{\FMG}$ when we do distributional analysis on orbits.

Let $G$ be an almost linear  Nash group. Based on the Lie group structure of $G$, we  denote  by $\Fg_{\BC}$ the complexification of its Lie algebra, by $\CU(\Fg_{\BC})$ the universal enveloping algebra of $\Fg_{\BC}$, by $\CZ(\Fg_{\BC})$ the center of the $\CU(\Fg_{\BC})$.

\begin{defin}\label{def: CasselmanWallach}
Following the notion in \cite[\S 3.3]{liu2013uniqueness}, we denote by $\RD^{\xi}(G)$ the space of Schwartz densities on $G$.
For every representation $\pi\in \Pi_{\FMG}(G)$, $\RD^{\xi}(G)$ acts on $\pi$, and $\pi$ is called \textbf{Casselman-Wallach} if  
\begin{enumerate}
    \item every $\RD^{\xi}(G)$-submodule is closed in $V$;
    \item $V$ is of finite length as an abstract $\RD^{\xi}(G)$-module.
\end{enumerate}
In particular, when $G$ is real reductive, a representation $\pi\in \Pi_{\FMG}(G)$ is Casselman-Wallach if and only if the associated $(\Fg_{\BC},K)$-module $\pi_K$  is of finite length (\cite[\S 3]{du1991representations}). Here  $K$ is the maximal compact subgroup of $G$.

We denote by $\Pi_{\CaWa}(G)$ the category of Casselman-Wallach representations of $G$ and denote by $\Pi_{\CaWa}^{\mathrm{irr}}(G)$ the category of irreducible Casselman-Wallach representations of $G$.
\end{defin}
\subsubsection{Multiplicity}
Recall that for a Bessel  triple $(\RG,\RH,\xi)$ associated to $(V,W)$
and $\pi\in \Pi_{\CaWa}^{\mathrm{irr}}(\RG)$, we set
\[m(\pi)=\dim \Hom_{\RH}(\pi,\xi);\]
 For a Fourier-Jacobi  triple $(\RG,\RH,\xi)$ associated to $(V,W)$
and $\pi_V\in \Pi_{\CaWa}^{\mathrm{irr}}(\RG_V)$, $\pi_W\in \Pi_{\CaWa}^{\mathrm{irr},\mathrm{nd}}(\RG_W^J)$, we set
\[m(\pi)=\dim \Hom_{\RH}(\pi,\xi),\quad \pi=\pi_V\boxtimes \pi_W.\]

There is a multiplicity-one theorem on $m(\pi)$.
\begin{thm}\label{thm: multiplicity basic}
We have 
\[
m(\pi)\leqslant 1.  
\]
\end{thm}
\begin{proof}
 All (B) and  (FJ 1) cases   were proved in \cite{jiang2010uniqueness}\cite{liu2013uniqueness}. All non-metaplectic equal-rank (FJ 2) cases were proved in 
 \cite{sun2012multiplicity}, Lemma \ref{lem: isomorphism between FJ1 and FJ2} implies the result for metaplectic equal-rank (FJ 2) cases.

 The multiplicity of (FJ 2) cases for $V\subsetneq W$ is equal to that of certain (FJ 1) cases via Lemma \ref{lem: isomorphism between FJ1 and FJ2}. Therefore, the multiplicity-one theorem also holds for (FJ 2) cases. 
\end{proof}

\subsubsection{Schwartz induction} For almost linear Nash groups $H\subset G$, 
there is a functor called Schwartz induction from $\Pi_{\FMG}(H)$ to $\Pi_{\FMG}(G)$.

\begin{defin}\label{def: Schwartz induction}
    The \textbf{Schwartz induction} is the functor 
    \[
    \Ind^{\CS,G}_P:\Pi_{\FMG}(H)\to \Pi_{\FMG}(G)
    \]
    that maps every $\pi_H$ to the image of 
    \[
\RT_{H,\pi}:\CS(G,V_{\pi_H})\to C^{\infty}(G,V_{\pi_H}),
    \]
    \[
f\mapsto (g\mapsto\int_H\pi_H(h)f(h^{-1}g)dh).
    \]
In particular, when $H\bs G$ is compact, the Schwartz induction is equal to the smooth induction (\cite[\S 2.1]{du1991representations}). 

\end{defin}

In Section \ref{section: GGP}, when we do distributional analysis on Nash manifold. We use the following geometric alternative definition.
\begin{pro}[Proposition 6.7 of \cite{chen2020schwartz}]\label{pro: geometric definition}
\label{pro: alt Schwartz induction}
For $\pi_H\in \Pi_{\FMG}(H)$, we denote by $H\bs (G\times \pi_H)$ the vector bundle over $H\bs G$ obtained by $G\times \pi_H$ quotient by left $H$-action ($h.(g,v)=(h\cdot g,\pi_H(h).v)$ for $h\in H$, $g\in G$ and $v\in \pi_H$) and this vector bundle is tempered.
Then
\[
 \Ind^{\CS,G}_H(\pi_H)=\Gamma^{\mathcal{S}}(H\backslash G,\pi_H),
\]
where $\Gamma^{\mathcal{S}}(H\backslash G,\pi_H)$ stands for the space of Schwartz sections over the tempered vector bundle $H\backslash (G\times \pi_H)$.
\end{pro}

\subsubsection{Schwartz induction in Bessel and Fourier-Jacobi models} Then we return to the algebraic group setting and use the notations $\epsilon,E,F,\tau$  as in Section \ref{section: Bessel Fourier-Jacobi}.

We first define a parabolic subgroup $\RP_{W^+,X^+}$ of $\RG_{W^+}$ and a pseudo-parabolic subgroup $\RP_{W^+,X^+}^J$ of $\RG_{W}^J$ based on a decomposition $W^+=(X^+\oplus Y^+)\oplus^{\perp} W$
, where $W,W^+$ are $\epsilon$-hermitian spaces over $E$ and $X^+,Y^+$ are totally isotropic subspaces.

\begin{defin}[Parabolic and Pseudo-Parabolic Subgroups]\label{defin: parabolic induction}
 \begin{enumerate}
\item In all situations, we denote by $\RP_{W^+,X^+}$ the parabolic subgroup of $\RG_{W^+}$ stabilizing the flag $X^+\subset W^+$. The Levi decomposition gives $\RP_{W^+,X^+}=\RL_{W^+,X^+}\cdot \RN_{W^+,X^+}$, where 
\[
\RL_{W^+,X^+}=
\begin{cases}
\Res_{E/F}\GL(X^+)\times \RG_W, \text{ if $\RG_W$ is not metaplectic},\\
\wt{\GL}(X^+)\times_{\pm 1} \RG_W, \text{ if $\RG_W$ is metaplectic}.
\end{cases}
\] 
\item 
When $\epsilon=-1$, let \[W^{++}=W^+\oplus^{\perp} (X_1^{++}\oplus Y_1^{++}),\] where $X_1^{++},Y_1^{++}$ are isotropic lines over $E$ such that $X_1^{++}\oplus Y_1^{++}$ is a non-degenerate $\epsilon$-hermitian space. We consider the embedding of $\RG_{W^+}^J=\RG_{W^+}\rtimes \CH(W^+)$ into $\RG_{W^{++}}$ as the subgroup that action invariantly on $X_1^{++}$. 

We denote by $\RP_{W^+,X^+}^J$ the subgroup of $\RG_{W^+}^J$ stabilizing the flag $X^+\subset W^{++}$, then 
\[
\RP_{W^+,X^+}^{J}=\RP_{W^+,X^+}\rtimes \CH((X^+)^{\perp}).
\]
There is a pseudo-Levi decomposition
\[
\RP_{W^+,X^+}^{J}=\RL_{W^+,X^+}^{J}\rtimes \RN_{W^+,X^+}^J,
\]
where
\[
\RL_{W^+,X^+}^J=\RL_{W^+,X^+}\rtimes \CH(W)=\begin{cases}
    \Res_{E/F}\GL(X^+)\times \RG_W^J \text{ if $\RG_W$ is not metaplectic,}\\
    \wt{\GL}(X^+)\times_{\pm 1} \RG_W^J \text{ if $\RG_W$ is metaplectic.}
\end{cases}
\]
and
\[
\RN_{W^+,X^+}^J=\RN_{W^+,X^+}\rtimes \Res_{E/F}X^+.
\]
\end{enumerate}
\end{defin}

Next, we define a normalized Schwartz induction based on the parabolic and pseudo-parabolic structure. For general linear groups, in \cite{bernstein1977induced}, Bernstein and Zelevinsky used the notion $\sigma_1\times \sigma_2$ to denote the normalized parabolic induction from the inflation of the representation $\sigma_1\boxtimes \sigma_2$ of the Levi component. We use the symbol $\sigma\rtimes \pi$ to denote the normalized Schwartz induction from the inflation of the representation $\sigma\boxtimes \pi$ of the pseudo-Levi component for  $\RG_W,\RG_W^J$. In particular, for general linear groups $\sigma\times\pi=\sigma\rtimes\pi$. 
\begin{defin}
\begin{enumerate}
 \item  Let $\epsilon=\pm 1$.
\begin{enumerate}
    \item When $\RG_W$ is not metaplectic, let $\sigma_{X^+}\in \Pi_{\FMG}(\Res_{E/F}\GL(X^+))$ and let $\pi_{W}\in \Pi_{\FMG}(\RG_W)$. We denote by $\sigma_{X^+}\boxtimes \pi_W$ the extension of $\RL_{W^+,X^+}$-representation $\sigma_{X^+}\boxtimes \pi_W$ to $\RP_{W^+,X^+}$.  We define
\[\sigma_{X^+}\rtimes \pi_W\]
to be the normalized Schwartz induction
\[
\Ind_{\RP_{W^+,X^+}}^{\CS,\RG_{W^+}}(\delta_{\RP_{W^+,X^+}}^{1/2}\sigma_{X^+}\boxtimes \pi_W).
\]
\item When $\RG_W$ is metaplectic, let $\sigma_{X^+}\in \Pi_{\FMG}^{\gen}(\Res_{E/F}\wt{\GL}(X^+))$ and let $\pi_W\in \Pi_{\FMG}^{\gen}(\RG_W)$. We denote by $\sigma_{X^+}\boxtimes_{\pm 1} \pi_W$ the inflation of the $\RL_{W^+,X^+}$-representation $\sigma_{X^+}\boxtimes \pi_W$ to $\RP_{W^+,X^+}$. We define  
\[\sigma_{X^+}\rtimes \pi_W\]
to be the normalized Schwartz induction
\[
\Ind_{\RP_{W^+,X^+}}^{\CS,\RG_{W^+}}(\delta_{\RP_{W^+,X^+}}^{1/2}\sigma_{X^+}\boxtimes_{\pm 1} \pi_W).
\]
\end{enumerate}
 \item Let $\epsilon=-1$. 

\begin{enumerate}
    \item  When $\RG_W$ is not metaplectic, let $\sigma_{X^+}\in \Pi_{\FMG}(\Res_{E/F}\GL(X^+)$ and $\pi_W\in\Pi_{\FMG}(\RG_W^J(F))$. We denote by $\sigma_{X^+}\boxtimes \pi_W$ the extension of $\RL_{W^+}^{J}$-representation $\sigma_{X^+}\boxtimes \pi_W$ to $\RP_{W^+}^J$. We define 
 \[
 \sigma_{X^+}\rtimes \pi_W
 \] to be the normalized Schwartz induction
 \[
 \Ind_{\RP_{W^+,X^+}^J}^{\CS,\RG_{W^+}^J}(\delta_{\RP_{W^+,X^+}^J}^{1/2}\sigma_{X^+}\boxtimes \pi_W).
 \]
 \item When $\RG_W$ is metaplectic, let $\sigma_{X^+}\in \Pi_{\FMG}(\Res_{E/F}\GL(X^+))$ and $\pi_W\in\Pi_{\FMG}^{\gen}(\RG_W^J)$.  We denote by $\wt{\sigma}_{X^+}$ the genuine representation of $\Pi_{\FMG}(\Res_{E/F}\wt{\GL}(X^+)$ obtained from Lemma \ref{lem: classification of representations}(1).  Denote by $\wt{\sigma}_{X^+}\boxtimes \pi_W$ the extension of $\RL_{W^+}^{J}$-representation $\wt{\sigma}_{X^+}\boxtimes \pi_W$ to $\RP_{W^+,X^+}^J$. We define 
 \[
 \sigma_{X^+}\rtimes \pi_W=\wt{\sigma}_{X^+}\rtimes \pi_W
 \] to be the normalized Schwartz induction
 \[
 \Ind_{\RP_{W^+,X^+}^J}^{\CS,\RG_{W^+}^J}(\delta_{\RP_{W^+,X^+}^J}^{1/2}\wt{\sigma}_{X^+}\boxtimes_{\pm 1} \pi_W).
 \]
\end{enumerate}
\end{enumerate}
\end{defin}

In \cite[p10]{liu2013uniqueness}, Liu and Sun used the structure of the mixed model and obtain the following lemma about the compatibility of Schwartz induction and the tensor with Weil representation. 
\begin{lem}\label{lem: compatibility with tensor Weil}
For $\sigma_{X^+}\in \Pi_{\FMG}(\Res_{E/F}(\GL(X^+)))$ and $\pi_W\in \Pi_{\FMG}(\RG_W)$
\[
(\sigma_{X^+}\rtimes\pi_W) \otimes   \omega_{W^+,\psi_F}=\sigma_{X^+}\rtimes (\pi_W\otimes \omega_{W,\psi_F}).
\]
\end{lem}
\subsubsection{Relation with branching's problem}
The multiplicity studied in this article is closely related to the branching problem. Let $\RG$ be an almost linear Nash group.
\begin{defin}
\begin{enumerate}
    \item  For $\sigma\in \Pi_{\FMG}(G)$, we denote by $\sigma^*$ the strong dual of $\sigma$, that is, the continous dual of $\sigma$ equip with the strong dual topology and the dual action of $G$.
    \item When $G$ is a reductive Lie group, for $\sigma\in \Pi_{\CaWa}(G)$, we denote by $\sigma^{\vee}$ the contragredient of $\sigma$, that is,  the Casselman-Wallach completion of the contragredient $(\sigma^{\alg})^{\vee}$ of the $(\Fg_{\BC},K)$-module $\sigma^{\alg}$ of $\sigma$. 

\end{enumerate}
\end{defin}
\begin{lem}\label{lem: reciprocity}
Let $G$ is a reductive Lie group, let $\pi_1\in\Pi_{\FMG}(G)$,  $\pi_2\in \Pi_{\CaWa}(G)$, then  we have
\[
\Hom_{\Delta G}(\pi_1\boxtimes\pi_2,\BC)=\Hom_G(\pi_1,\pi_2^{\vee}).
\]
\end{lem}
\begin{proof}
Since $\pi_2$ is Casselman-Wallach, from \cite{bernstein2014smooth}, the underlying space $V_{\pi_2}$ is nuclear. Then from \cite[Proposition 50.4, 50.7]{treves2016topological},
\[
\Hom_{G}(\pi_1\wh{\otimes}\pi_2,\BC)= \Hom_G(\pi_1,\pi_2^{*}).
\]

By definition, we have
\[
\Hom_{G}(\pi_1\wh{\otimes}\pi_2,\BC)=\Hom_{\Delta G}(\pi_1\boxtimes\pi_2,\BC).
\]

Since continuous $G$-equivariant map sends $G$-smooth vectors to $G$-smooth vectors, we have
\[
\Hom_G(\pi_1,\pi_2^{*})=\Hom_G(\pi_1,(\pi_2^{*})^{\infty}),
\]
where $(\pi_2^{*})^{\infty}$ consists of smooth vectors in $\pi_2^*$. Denote by $(\pi_2^{*})^{\infty,\alg}$ the $(\Fg_{\BC},K)$-module consisting of $K$-finite vectors in $(\pi_2^{*})^{\infty}$. Since $\pi_2^{\alg}$ is dense in $\pi_2$, then
\[
(\pi_2^{*})^{\infty,\alg}=(\pi_2^{\alg})^{\vee}.
\]

Therefore, from the uniqueness of Casselman-Wallach's completion 
\[(\pi_2^*)^{\infty}=\pi_2^{\vee}.\]
Then we complete the proof.
\end{proof}
Therefore, in (B) cases, we have
\[
m(\pi_V\boxtimes \pi_W)=\dim\Hom_{\Delta\RG_W}(\pi_V\boxtimes \pi_W,1_{\Delta\RG_W})=\dim\Hom_{\RG_W}(\pi_V|_{\RG_W},\pi_W^{\vee}),
\]
which is the multiplicity studied in the branching problem.
\subsection{Double cosets}\label{section: double coset}
With notations in Section \ref{section: Bessel Fourier-Jacobi}, for each admissible pair $(V, W)$, we define a basic admissible pair and totally isotropic spaces $X^+, Y^+$ for further construction.
\begin{defin}\label{defin: basic admissible pair associated to}
\begin{enumerate}
    \item When $(V, W)$ is of Bessel type with decomposition $V=W\oplus^{\perp} Z\oplus^{\perp} D$, we can construct a codimension-one admissible pair $(W^+, V)$ of Bessel type by taking $W^+=V\oplus^{\perp} D^+$ for an anisotropic line $D^+$ over $E$. 
    We further assume $D$ and $D^+$ have different signatures. Then there is a 
    generator $z_0$ of $D$ and a 
    generator $z_{0}^+$ of $D^+$ such that
    \[
    \langle z_0,z_0\rangle=-\langle z_{0}^{+},z_0^{+} \rangle\neq 0
    \] 
we take  $X^+=X\oplus E(z_0+z_0^+)$ and $Y^+=Y\oplus E(z_0-z_0^+)$.

     Then $X^+$ and $Y^+$ are totally isotropic spaces  over $E$ and 
     \[W^+=W\oplus^{\perp} (X^+\oplus Y^+).\]
    \item  When  $(V, W)$ is of Fourier-Jacobi type and $W\subsetneq V$, we can construct an equal-rank admissible pair $(V, W^+)$ of Fourier-Jacobi type and $V\subset W^+$.

    We take $W^+=V$,  $X^+=X$, $Y^+=Y$, then \[W^+=W\oplus^{\perp} (X^+\oplus Y^+).\]
    \item When $(V,W)$ is of Fourier-Jacobi type and $V\subset W$, we construct an almost equal-rank admissible pair $(V^+,W)$ of Fourier-Jacobi type and $W\subsetneq V$.

    Let $H_1$ be a hyperbolic plane over $E$ with decomposition $H=X_1^+\oplus Y_1^+$ for totally isotropic $X_1^+,Y_1^+$. Set $V^+=W\oplus^{\perp}H_1$, $X^+=X\oplus X_1^+$, $Y^+=Y\oplus Y_1^+$, then
    \[V^+=V\oplus^{\perp}(X^+\oplus Y^+).\]
\end{enumerate}
We call $(W^+,V)$, $(V,W^+)$, resp. $(V^+,W)$ \textbf{the basic admissible pair} of  $(V,W)$.
\end{defin}

Let $(\RG,\RH,\xi)$ be the triple associated to $(V,W)$ and let $(\RG^+,\RH^+,\xi^+)$ the triple associated to its basic admissible pair. In particular, $\RH^+=\Delta \RG_V,\Delta \RG_W^J,\Delta \RG_V$, respectively in (B), (FJ 1), (FJ 2) cases,  and $\xi^+$ is the trivial character $1_{\RH^+}$ of $\RH^+(F)$. 

\begin{enumerate}
    \item Following Definition \ref{defin: parabolic induction}(1), we define $\RP_V$, $\RL_V,\RN_V$ for the decomposition
    \[
  V=(W\oplus^{\perp} D)\oplus^{\perp}(X\oplus Y)  
  \]
   in (B) cases, or
  \[
   V= W\oplus^{\perp}(X\oplus Y) 
    \]
     in (FJ 1) cases.
    \item Following Definition \ref{defin: parabolic induction}(2), we define $\RP_{W}^J,\RL_W^J,\RN_W^J$ for the decomposition 
    \[
    W= V\oplus^{\perp}(X\oplus Y)
    \]
    in the (FJ 2) cases.
\end{enumerate}

We let $\RP^+$ be the closed subgroup of $\RG^+$ defined by
\begin{num}
    \item \label{def: P^+}
\begin{enumerate}
    \item  $\RP^+=\RP_{W^+,X^+}\times \RG_V$ in (B) cases;
    \item $\RP^+=\gamma^{-1}(\RG_V\times \RP_{W^+,X^+}^J)\gamma=\RG_V\times \gamma^{-1}\RP_{W^+,X^+}^J\gamma$ in (FJ 1) cases, where $\gamma=1\times (1\rtimes z_{-1}^+)\in \RG_V\times (\RG_{W^+}\rtimes \CH(W^+))=\RG^+$;
    \item $\RP^+=\gamma^{-1}(\RP_{V^+,X^+}\times \RG_W^J)\gamma=\gamma^{-1}\RP_{V^+,X^+}\gamma\times \RG_W^J$ in (FJ 2) cases, where $\gamma\in \RG_V^+$ is the element satisfying 
    \[
    \gamma z_W=z_W, \quad\gamma z_{1}^+=z_{1}^-,\quad\gamma z_1^{-1}=-z_1^+.
    \]
\end{enumerate}
\end{num}

Based on the Levi decomposition of $\RP_{W^+,X^+},\RP_{V^+,X^+}$ and the pseudo-Levi decomposition of $\RP^{J}_{W^+}$, we set the pseudo-Levi subgroup $\RL^+\subset \RP^+$ as
\begin{equation}\label{equ: L^+}
\begin{cases}
  (\Res_{E/F}\GL(X^+)\times \RG_W)\times \RG_V& \text{in (B) cases,}\\
\gamma(\RG_V\times (\Res_{E/F}\GL(X^+)\times \RG_W^J))\gamma^{-1}& \text{in (FJ 1) cases,}\\
\gamma((\Res_{E/F}\GL(X^+)\times \RG_V)\times \RG_W^J)\gamma^{-1}&\text{in (FJ 2) cases}.
\end{cases}
\end{equation}
and, in each situation, we denote by $p_{\GL}$ the morphism from $\RL^+$ to $\Res_{E/F}\GL(X^+)$.

Then we study the double cosets $\RP^+(F)\bs \RG^+(F)/\RH^+(F)$ to analyze the space
\begin{equation}
\Hom_{\RH^+}(\Ind_{\RP^+}^{\CS,\RG^+}(\sigma_{X^+}\boxtimes \pi_V\boxtimes \pi_W),1_{\RH^+}).    
\end{equation}

In (B) (FJ 1) cases, this has been studied in the literature  (\cite{ginzburg2011descent}).
\subsubsection{The open double coset}
To study the open double closet, we introduce some basic notions.
\begin{defin}\label{defin: mirabolic}
\begin{enumerate}
\item (Mirabolic subgroups)
    Let $X'\subset X''$ be  vector spaces over $E$ such that $\dim_EX'+1=\dim_E X''$, we denote by $\RR_{X',X''}$ the mirabolic subgroup of $\GL(X^+)$ stabilizing $X'$ and invariant on $X''/X'$, that is,
\[
\RR_{X',X''}=\{g\in \GL(X''): gX'\subset X',\ g \text{ acts trivially on } X''/X'\}.
\]
The isomorphism class of $\RR_{X',X''}$ is only determined by on $\dim_E X'$ and we denote  this isomorphism class by $\RR_{\dim_E X',1}$.

\item In (B) cases,  let $\RN_V'$ be the subgroup of $\RN_V$ that fixes all points on $X\oplus D$, then
\begin{num}
\item \label{equ: invariant XD}  
\[
(1\times \RG_W)\rtimes \RN_V'\subset (\Res_{E/F}\GL(X)\times \RG_{W\oplus D})\rtimes \RN_V
\]
is the subgroup of $\RP_V$ that fixes all points on $X\oplus D$.
\end{num}
\end{enumerate}
\end{defin}

\begin{lem}\label{lem: decomposition of the stabilizer group}
If $\RG_V$ is not metaplectic, $\RP^+(F)\RH^+(F)$ is an open double coset and   $\RP^+\cap \RH^+=\Delta \RS_{\mathrm{open}}$, where 
\[
\RS_{\mathrm{open}}=\begin{cases}
    \RP_{W^+,X^+}\cap \RG_V& \text{in (B) cases}\\
    \gamma^{-1}\RP_{W^+,X^+}^J\gamma\cap \RG_{V}& \text{ in (FJ 1) cases}\\
    \gamma^{-1}\RP_{V^+,X^+}\gamma\cap \RG_W^J& \text{ in (FJ 2) cases}
\end{cases}
\]
Moreover, 
there is  a decomposition
\[
 \RS_{\mathrm{open}}=\begin{cases}
    (\Res_{E/F}\RR_{X',X^+}\times \RG_{V_0})\rtimes \RN_V' & \text{ in (B) cases,}\\
    (\Res_{E/F}\RR_{X',X^+}\times \RG_W)\rtimes \RN_V &\text{ in (FJ 1) cases,}\\
    (\Res_{E/F}\RR_{X',X^+}\times \RG_V^J)\rtimes \RN_W &\text{ in (FJ 2) cases,}
 \end{cases}
\]
where 
\[
X'=\begin{cases}
    X&\text{ in (B)(FJ 2) cases,}\\
    \Span\{z_2^+,\cdots,z_r^+\},&\text{in (FJ 1) cases.}
\end{cases}
\]\end{lem}  
\begin{proof}
It suffices to show that the right $\RH^+(F)$-action on $\RP^+(F)\bs \RG^+(F)$ has an open orbit $[1]\RH^+(F)$ and  the stabilizer group $\RS_{\mathrm{open}}$ at $[1]$ has the decomposition in the lemma.
\begin{num}
   \item\label{equ: realization} 
\begin{enumerate}
    \item In (B) cases, we identify the $\RH^+(F)$-action on $\RP^+(F)\bs \RG^+(F)$ as the $\RG_V(F)$ action on $\RP_{W^+,X^+}(F)\bs \RG_{W^+}(F)$, which consists of 
totally isotropic subspaces $X_{W^+}=X^+g\subset W^+(g\in \RG_{W^+}(F))$  with $\dim_EX_{W^+}=r+1$ 
 and the open orbit consists of those not contained in $V$.
 \item In (FJ 1) cases, we identify $\RH^+(F)$-action on $\RP^+(F)\bs \RG^+(F)$ as the $\RG_V(F)$-action on $\RP_{W^+,X^+}^J(F)\bs \RG_{W^+}^J(F)$, which consists of the totally isotropic subspaces $X^+\gamma g\ (g\in \RG_{W^+}^J(F))$  of $W^{++}=W^+\oplus^{\perp} Z^{++}$  where $Z^{++}=X_1^{++}\oplus Y_1^{++}$ is a hyperbolic plane  over $E$.  The open orbit consists of those not contained in $W$. 
 \item In (FJ 2) cases, we identify $\RH^+(F)$-action on $\RP^+(F)\bs \RG^+(F)$ as the $\RG_W^J(F)$-action on $\RP_{V^+,X^+}(F)\bs \RG_{V^+}(F)$, which consists of the totally isotropic subspaces $X^+\gamma g\ (g\in \RG_{V^+}^J(F))$  of $V^{+}$ and the open orbit consists of those not contained in $W\oplus X_1^+$.
\end{enumerate}

\end{num}

In (B) cases
\[
\begin{aligned}
\RS_{X^+}=&\{g\in \RG_V: X^+g=X^+\}\\
=&\{g\in \RG_V: Xg=X,\ (z_0+iz_0^+)g\subset X^+\}\ (\text{Since }V\cap X^+=X)  \\
=&\{g\in \RG_V: Xg=X,\ z_0(g-1)\subset X\}\ (\text{Since } z_0^+(g-1)=0 \text{ for } g\in \RG_V)\\
=&\{g\in \RG_V: Xg=X,\ X^+(g-1)\subset X\}.
\end{aligned}
\] 

In particular, $\RP_{W^+,X^+}\cap \RG_V\subset \RP_V$. From (\ref{equ: invariant XD}), we have $\RN_V'\rtimes \RG_W$ is a normal subgroup of $\RP_{W^+,X^+}\cap \RG_V$ and $\RP_{W^+,X^+}\cap \RG_V/\RN_V'\rtimes \RG_W\cong \Res_{E/F}\RR_{X,X^+}$. Therefore,
\[
\RP_{W^+,X^+}\cap \RG_V=(\Res_{E/F}\RR_{X,X^+}\times \RG_W)\rtimes \RN_V'
\]

In (FJ 1) cases, by definition, $X^{+}\gamma$ has a basis
\[
z_1^++z_1^{++},z_2^+ \cdots, z_r^+
\] 
so it is not contained in $W$. The stabilizer subgroup $\RS_{X^{+}\gamma}$ in $\RG_{V}$ consists of elements stabilizing 
\[
X'=W\cap X^{+}\gamma=\Span\{z_2^+, \cdots, z_r^+\}
\]
and invariant on $X/X'$. Then we have a decomposition 
\[
\RS_{X^{+}\gamma}=(\Res_{E/F}\RR_{X',X^{+'}}\times \RG_W)\rtimes \RN_V
\]
By conjugation with $\gamma$, we have 
\[
\gamma^{-1} \RP_{W^+,X^+}^J\gamma\cap  \RG_V=(\RR_{X',X^+}\times \RG_W)\rtimes \RN_V
\]

In (FJ 2) cases, by definition, $X^+\gamma$ has a basis 
\[
z_1,,z_2,\cdots,z_r,z_1^-
\]

The stabilizer subgroup $\RS_{X^+\gamma}$ of the $\RG_W^J$ consists of elements stabilizing 
\[
X'=(W\oplus X_1^+)\cap X^+\gamma=\Span\{z_1,z_2,\cdots,z_r\}
\]
and invariant on $X^+/X'$. Then we have a decomposition
\[
\RS_{X^+\gamma}=(\Res_{E/F}\RR_{X',X^+}\times \RG_V^J)\rtimes \RN_{V}
\]

\end{proof}

\begin{cor}\label{cor: mirabolic structure}
For general $\RG_V$, we have $\RH\subset \RP^+\cap \RH^+$ and \[
\RH(F)\bs \RP^+\cap \RH^+(F)=\RN_{0,X^+}(E)\bs \RR_{X',X^+}(E),
\]
where $X',X^+$ are chosen in Lemma \ref{lem: decomposition of the stabilizer group} when $\RG_V$ is not metaplectic; and are chosen to be the $X',X^+$ associate to $\udl{\RG}_V$ in Lemma \ref{lem: decomposition of the stabilizer group} when $\RG_V$ is metaplectic, where 
\[
\udl{\RG}_V=\begin{cases}
 \GL(V)&\text{ when }\RG_V=\wt{\GL}(V),\\
 \Sp(V)&\text{ when }\RG_V=\wt{\Sp}(V).\\
\end{cases}
\]
\end{cor}
\begin{proof}
When $\RG_V$ is not metaplectic, by definition 
\[
\RH=\begin{cases}
    \Delta( (\Res_{E/F}\RN_{0,X^+}\times \RG_V)\rtimes \RN_V') & \text{ in (B) cases,}\\
    \Delta((\Res_{E/F}\RN_{0,X^+}\times\RG_W)\rtimes \RN_V) & \text{in (FJ 1) cases,}\\
    \Delta((\Res_{E/F}\RN_{0,X^+}\times \RG_{V}^J)\rtimes\RN_W) & \text{ in (FJ 2) cases.}
\end{cases}
\]
from Lemma \ref{lem: decomposition of the stabilizer group}, we have 
\[
\RH(F)\bs \RP^+\cap \RH^+(F)=\RN_{0,X^+}(E)\bs \RR_{X',X^+}(E),
\]
When $\RG_V$ is metaplectic,  we set $\udl{\RG}^+,\udl{\RP}^+,\udl{\RH}^+$ for $\udl{\RG}_W$, then $\RG^+(F),\RP^+(F),\RH^+(F)$ are central extensions of $\udl{\RG}^+,\udl{\RP}^+,\udl{\RH}^+$ with the same kernel \[(\pm 1_{\RG_V})\times (\pm 1_{\RG_W})\]
Hence,
\[
\RP^+\cap \RH^+=\begin{cases}
\Delta((\Res_{E/F}\wt{\RR}_{X',X^+}\times_{\pm 1}\RG_W)\rtimes \RN_V) & \text{in (FJ 1) cases,}\\
   \Delta((\Res_{E/F}\wt{\RR}_{X',X^+}\times_{\pm 1}\RG_V^J)\rtimes \RN_W) & \text{ in (FJ 2) cases.}
\end{cases} 
\]
\[
\RH=\begin{cases}
\Delta((\Res_{E/F}\wt{\RN}_{0,X^+}\times_{\pm 1}\RG_W)\rtimes \RN_V) & \text{in (FJ 1) cases,}\\
   \Delta((\Res_{E/F}\wt{\RN}_{0,X^+}\times_{\pm 1}\RG_V^J)\rtimes \RN_W) & \text{ in (FJ 2) cases.}
\end{cases} 
\]
Therefore
\[
\RH(F)\bs \RP^+\cap \RH^+(F)=\RN_{0,X^+}(E)\bs \RR_{X',X^+}(E),
\]
then we conclude the proof.
\end{proof}

\subsubsection{Other double cosets}
We study other double cosets for $\RP^+(F)\bs \RG^+(F)/\RH^+(F)$  using the structure of $\RP^+(F)\bs \RG^+(F)$ in the proof for Lemma \ref{lem: decomposition of the stabilizer group}.

We need the following lemma for the proof of our multiplicity inequality.
\begin{lem}\label{lem: geometric for refine 2}
The complement $\RG^+(F)-\RH^+(F)\RP^+(F)$ is the zero set of a polynomial $f^+$ on $\RG^+(F)$ that is left $\RH^+(F)$-invariant and right $(\RP^+,\psi_{\RP^+})$-equivariant for the algebraic character $|\det\circ p_{\GL}|^2$ of $\RP^+(F)$, where $p_{\GL}$ is the the projection from $\RP^+$ to the $\GL$-part of the pseudo-Levi component $\RL^+$ of $\RP^+$ (\ref{equ: L^+}).    
\end{lem}
\begin{proof}
In (B) cases,  we choose
    an $E$-basis $v_1,\cdots, v_n$ of $V$ and a $E$-basis $v_1^{X^+},\cdots,v_{r+1}^{X^+}$ of $X^+$. 
We set 
\[
A_g=\left[g_V^{-1}.v_1,\cdots,g_V^{-1}.v_{n},g_{W^+}.v_1^X,\cdots,g_{W^+}.v_{r+1}^X\right].
\]
The element $g=g_{W^+}\times g_V\in \RG^+(F)$ is not contained in $\RP^+(F)\RH^+(F)$ if and only if $\mathrm{rk}(A_g)=n$, equivalently,
\[
\det(A_gA_g^t)=0
\]
so we can set $f^+(g)=\det(A_gA_g^t)$, which is left $\RH^+(F)$-invariant and right $(\RP^+,|\det(p_{\GL}(p^+)|^2)$-equivariant.

In (FJ 1) cases, we choose an $E$-basis 
 $v_1,\cdots, v_n$ of $W$ and a $E$-basis $v_1^{X^+},\cdots,v_{r}^{X^+}$ of $X^+$,  we set 
\[
A_g=\left[g_V^{-1}.v_1,\cdots,g_V^{-1}.v_{n},\gamma g_{W^+}.v_1^{X^+},\cdots,\gamma g_{W^+}.v_{r}^{X^+}\right].
\]
The element $g=g_{V}\times g_{W^+}\in \RG^+(F)$ is not contained in $\RP^+(F)\RH^+(F)$ if and only if $\mathrm{rk}(A_g)=n$, equivalently,
\[
\det(A_gA_g^t)=0
\]
So we can set $f^+(g)=\det(A_gA_g^t)$, which is left $\RH^+(F)$-invariant and right $(\RP^+,|\det(p_{\GL}(p^+)|^2)$-equivariant.

In (FJ 2) cases, we choose an $E$-basis 
 $v_1,\cdots, v_{n+1}$ of $W\oplus X_1^+$ and a $E$-basis $v_1^{X^+},\cdots,v_{r}^{X^+}$ of $X^+$,  we set 
\[
A_g=\left[g_W^{-1}.v_1,\cdots,g_W^{-1}.v_{n+1},\gamma g_{V^+}.v_1^{X^+},\cdots,\gamma g_{V^+}.v_{r+1}^{X^+}\right].
\]
The element $g=g_{V^+}\times g_{W}\in \RG^+(F)$ is not contained in $\RP^+(F)\RH^+(F)$ if and only if $\mathrm{rk}(A_g)=n+1$, equivalently,
\[
\det(A_gA_g^t)=0.
\]
Hence, we can set $f^+(g)=\det(A_gA_g^t)$, which is left $\RH^+(F)$-invariant and right $(\RP^+,|\det(p_{\GL}(p^+)|^2)$-equivariant.
\end{proof}

We also need the following three lemmas about the structure of the complement of the open double for the analysis of closed orbits in  Section \ref{section: GGP}. 
These structures is not hard to compute using the realization in (\ref{equ: realization}), and has been computed case by case in \cite{xue2020bessel}\cite{chen2021local}\cite{chen2022FJ}. For uniformity of notions, we take $\CX=\RP^+(F)\bs\RG^+(F)$ with $\RH^+(F)$ acting on the right, $\CU$ be the open orbit $\RP^+(F)\bs\RP^+(F)\RH^+(F)$  and $\CZ=\CX-\CU$.

\begin{lem}\label{lem: complement}
We can choose appropriate $\gamma'\in \RG^+(F)$ such that
\begin{enumerate}
    \item In (B) cases,    \begin{enumerate}
\item when $\RG_V=\RU(V)$, $\CZ$ has one  $\RH^+(F)$-orbit $[\gamma']$;
\item  when $\RG_V=\SO(V)$,  $\CZ$ has two  $\RH^+$(F)-orbits $[\gamma']$ and $[\gamma'g'']$ for an element $g''\in \Delta \RO(V)-\Delta \SO(V)$, and both are closed orbits.
\end{enumerate}
    \item In (FJ 1) cases,  
    $\CZ$ has one $\RH^+(F)$-orbit $[\gamma']$;
    \item In (FJ 2) cases, 
    \begin{enumerate}
        \item When $\dim X^+>1$, $\CZ$ has one  $\RH^+(F)$-orbit $[\gamma']$;
    \item When $\dim X^+=1$, $\CZ$ has  two orbits $[\gamma^{-1}]$ and $[\gamma']$, where $[\gamma^{-1}]$ is a single point and $[\gamma']$ is not.
\end{enumerate}
\end{enumerate}
\end{lem}

\begin{lem}\label{lem: stabilizer group}
Following the notations in Lemma \ref{lem: complement}, the stabilizer group
\[
\RS_{\gamma'}=\begin{cases}
  \Delta(\Res_{E/F}\GL(X_c)\times \RG_{V_0})\rtimes \RN_V  & \text{ in (B) (FJ 1) cases}, \\
\Delta(\Res_{E/F}\GL(X_c)\times \RG_{W_0}^J)\rtimes \RN_{W}^J  & \text{in (FJ 2) cases},\\
\end{cases}
\]
where the notations are defined as following
\begin{itemize}
    \item We denote by $X_c'$  the totally isotropic subspace  of $V,W^+=V,W\oplus X_1^+$ in (\ref{equ: realization}) corresponding the representative $\gamma'$, respectively. We denote  by 
    \[
    X_c=\begin{cases}
    X_c'& \text{in (B) (FJ 1) cases},\\
        W\cap X_c'& \text{ in (FJ 2) cases}.
    \end{cases}    
    \]
    \item There are decompositions \[
    V=X_c\oplus Y_c\oplus V_0 \text{ in (B)(FJ 1) cases},
    \]
    \[
    W=X_c\oplus Y_c\oplus W_0 \text{ in (FJ 2) cases},
    \]
    respectively, satisfying $X_c\oplus Y_c$  is a nondegenerate $\epsilon$-hermitian space. 
\item $\RN_{V,X_c}$, $\RN_{W,X_c}^J$ are the pseudo-unipotent part of the pseudo-parabolic subgroup $\RP_{V,X_c},\RP_{W,X_c}^J$ of $\RG_V,\RG_{W}^J$ in Definition \ref{defin: parabolic induction}.
\end{itemize} 
\end{lem}

\begin{lem}\label{lem: conormal bundle}
As a representation of the stabilizer group $\RS_{\gamma'}$, the fiber of the conormal bundle $\CN_{\CZ|\CX}^{\vee}$ at $\gamma'$ is
\[
\mathrm{Fib}_{\gamma'}(\CN_{\CZ|\CX}^{\vee})=\begin{cases}
    \std_{X_c}\oplus \overline{\std_{X_c}} & \text{  when } \RG_V=\RU(V),\\
    \std_{X_c} &\text{when $\RG_V=\SO(V),\Sp(V),\Mp(V)$},\\
\end{cases}
\]    
where $\std_{X_c}$ is the standard representation of $\GL(X_c)(E)$.
\end{lem}

\section{The integrals method
}\label{section: integral}
This section introduces and applies the integrals methods to prove Theorem \ref{thm: main intro}(3).

For a Bessel or Fourier-Jacobi triple $(\RG,\RH,\xi)$ associated to an admissible pair $(V, W)$, we consider the triple $(\RG^+, \RH^+,\xi^+)$ associated to its basic admissible pair (Definition \ref{defin: basic admissible pair associated to}).


In \cite{jiang2010uniqueness} and \cite{liu2013uniqueness}, the authors constructed a family of integrals
\[
\RI_{s,\mu,\lambda}(\varphi_s,v)=\int_{\RH(F)\bs \RH^+(F)}\mu(v,\Lambda(\varphi_s(h)))d(h,h)\quad \varphi_s\in |\det|^s\sigma\rtimes \pi_W,\ v\in \pi_V
\]
in (B) (FJ 1) cases, and proved the integral is absolutely convergent and nonvanishing when $\Re(s)$ is large enough, where $\Lambda:\sigma\boxtimes \pi_W\to \pi_W$ constructed from a fixed nonzero Whittaker functional $\lambda:\sigma\to \BC$. With this method, they proved the inequalities
\[
m(\pi_V\boxtimes \pi_W)\leqslant m((|\det|^s\sigma_{X^+}\rtimes\pi_W)\boxtimes \pi_V) \text{ in (B) cases}
\]
\[
m(\pi_V\boxtimes \pi_W)\leqslant m(\pi_V\boxtimes(|\det|^s\sigma_{X^+}\rtimes\pi_W)) \text{ in (FJ 1) cases}
\]
when $\Re(s)$ is large enough and concluded the multiplicity-one theorems from the basic cases.  

In this section, we define the integral for all cases ((B)(FJ 1)(FJ 2)) uniformly and  prove the integral is absolutely convergent and nonvanishing $\Re(s)$ is large enough as in \cite{jiang2010uniqueness}\cite{liu2013uniqueness}. Besides, we prove the meromorphic continuation of the integrals to the complex plane. To be more precise, we treat $\int_{\RH(F)\bs \RH^+(F)}$ as the composition of two  integral operators respectively on $\RH^+\cap \RP^+(F) \bs \RH^+(F)$ and $\RH(F)\bs \RH^+\cap \RP^+(F)$ and prove the  properties for each integral operator.


\subsection{The integral over $\RH(F)\bs \RP^+\cap\RH^+(F)$}
From Corollary \ref{cor: mirabolic structure}, we have
\[
\RH(F)\bs  \RP^+\cap\RH^+(F)=\RN_{0,X^+}(E)\bs \RR_{X',X^+}(E)
\]
so we use the classical methods in the works of Jacquet, Piatetski-Shapiro and Shalika (\cite{jacquet1979automorphic}\cite{jacquet1990exterior}\cite{jacquet1981euler}\cite{jacquet2009archimedean}) to refine the integral and prove the necessary results.

Let $\RP^+=\RL^+\rtimes \RN^+$ be the pseudo-parabolic subgroup  of $\RG^+$ defined in (\ref{def: P^+}) and its pseudo-Levi subgroup $\RL^+$ is given in (\ref{equ: L^+}).
For a generic representation $\sigma_{X^+}\in\Pi_{\CaWa}(\Res_{E/F}\GL(X^+))$ and representations $\pi_V\in \Pi^{\mathrm{irr}}_{\CaWa}(\RG_V)$, and (i) $\pi_W\in \Pi^{\mathrm{irr}}_{\CaWa}(\RG_W)$ in (B) cases (ii) $\pi_W\in \Pi^{\mathrm{irr},\mathrm{nd}}_{\CaWa}(\RG_W)$ in (FJ 1)(FJ 2) cases, we define a $\RP^+(F)$-representation  $l(\sigma_{X^+},\pi_V,\pi_W)$ inflated from the $\RL^+(F)$-representation 
 \[
\begin{cases}
(\sigma_{X^+}\boxtimes \pi_W)\boxtimes\pi_V& \text{in (B) cases}, \\
\pi_V\boxtimes(\sigma_{X^+}\boxtimes \pi_W)& \text{in (FJ 1) cases},\\
(\sigma_{X^+}\boxtimes\pi_V)\boxtimes \pi_W &\text{ in  (FJ 2) cases}.
\end{cases}
\]
and denote by $\RI(\sigma_{X^+},\pi_V,\pi_W)$ the normalized Schwartz induction
\[
\Ind_{\RP^+}^{\CS,\RG^+}(\delta_{\RP^+}^{1/2}\otimes l(\sigma_{X^+},\pi_V,\pi_W)).
\]

We also fix a nonzero Whittaker functional $\lambda\in \Hom_{\Res_{E/F}\RN_{0,X^+}}(\sigma_{X^+},\xi_{0,X^+})$ and a nonzero Bessel function $\mu\in \Hom_{\RH}(\pi_V\boxtimes \pi_W,\xi)$. When $\Re(s)$ is large enough, we construct a family of elements 
\[
\RJ_{s,\mu,\lambda}\in \Hom_{\RP^+\cap \RH^+}(l(|\det|^s\sigma_{X^+},\pi_V,\pi_W),\BC). 
\]

Following the notation in Lemma \ref{lem: decomposition of the stabilizer group}, 
we define
\[
\begin{aligned}
&\RJ_{s,\mu,\lambda}(l(v_{\sigma_{X^+}},v_{\pi_V},v_{\pi_W}))\\=&\int_{\RH(F)\bs \RP^+\cap \RH^+(F)}\lambda(\sigma_{X^+}(p_{\GL}(g))v_{\sigma_{X^+}})\mu(\pi_{V,W}(g)v_{\pi_V}\boxtimes v_{\pi_{W}})|\det(g_{X})|^sd(g,g)\\
=&\int_{\RN_{0,X^+}(E)\bs \RR_{X',X^+}(E)}\lambda(\sigma_{X^+}(g_{X^+})v_{\sigma_{X^+}})\mu(\pi_{V,W}(g_{X^+})v_{\pi_V}\boxtimes v_{\pi_{W}})|\det(g_{X^+})|^sdg_{X^+}
\end{aligned}
\]
for every  $v_{\sigma_{X^+}}\in \sigma_{X^+}$,
$v_{\pi_W}\in \pi_W$ and $v_{\pi_V}\in \pi_V$. In the integral, $g\in \RS_{\text{open}}(F)$, $(g,g)\in \RP^+\cap \RH^+(F)$ and we denote by $\pi_{V,W}(g)$ the action of $l(1,\pi_V,\pi_W)((g,g))$ on  $l(1,\pi_V,\pi_W)=\pi_V\boxtimes \pi_W$. We denote by $p_{\GL}$ the projection from $\RS_{\text{open}}(F)$ to $\RR_{X',X^+}(E)$ 
in the decomposition in Lemma \ref{lem: decomposition of the stabilizer group}.



We will prove the following theorem following \cite{jacquet1981euler}\cite{jacquet1990exterior}\cite{jacquet2009archimedean}.
\begin{thm}\label{thm: main}
\begin{enumerate}
    \item The integral   $\RJ_{s,\mu,\lambda}$  is absolutely convergent when $\Re(s)$ is large enough;
    \item In  the continuous dual $(l(\sigma_{X^+}\boxtimes \pi_V\boxtimes \pi_W))^*$, $\RJ_{s,\mu,\lambda}$ has a meromorphic continuation to the complex plane $\BC$.
\end{enumerate} 
\end{thm}


\subsubsection{Absolute convergence} \label{section: absolute convergence}

We prove the absolute convergence of $\RJ_{s,\mu,\lambda}$ when $\Re(s)$ is large enough following \cite{jacquet2009archimedean}. 
\begin{defin}
\begin{enumerate}
\item For $v_{\sigma_{X^+}}\in V_{\sigma_{X^+}}$, we denote by $\RW_{v_{\sigma_{X^+}}}$  the \textbf{Whittaker functional} defined by 
\[
\RW_{v_{\sigma_{X^+}}}(g_{X^+})=\lambda(\sigma_{X^+}(g_{X^+})v_{\sigma_{X^+}}),\quad g_{X^+}\in \GL(X^+)(E).
\]
\item For $v_{\pi_V}\in \pi_V$ and $v_{\pi_W}\in \pi_W$, we denote by $\RB_{v_{\sigma_{V}},v_{\sigma_W}}$  the \textbf{Bessel functional} defined by 
\[
\RB_{v_{\pi_{V}},v_{\pi_W}}(g_{V},g_W)=\mu(\pi_{V}(g_{V})v_{\pi_{V}},\pi_W(g_W)v_{\pi_W}),\quad g_V\in \RG_V(F), \quad g_W\in  \RG_W(F).
\]
We denote by $\FJ_{v_{\sigma_{V}},v_{\sigma_W}}$ the \textbf{Fourier-Jacobi functional}
\[\FJ_{v_{\pi_{V}},v_{\pi_W}}(g_{V},g_W)=\mu(\pi_{V}(g_{V})v_{\pi_{V}},\pi_W(g_W)v_{\pi_W}),\quad g_V\in \RG_V(F), \quad g_W\in  \RG_W^J(F)\]
\item Following the notion in Section \ref{section: double coset}, for $g\in \RR_{X',X^+}(E)\subset \RP^+\cap \RH^+(F)$, we define the  functional
\[
\RW'_{v_{\pi_{V}},v_{\pi_W}}(g)=\begin{cases}
  \RB_{v_{\pi_V},v_{\pi_W}}(\diag(g,1_{W\oplus D}),1_W) \text{ in (B) cases,} \\
   \RB_{v_{\pi_V},v_{\pi_W}}(1_W,\diag(g,1_{W})) 
 \text{ in (FJ 1) cases,}\\
 \FJ_{v_{\pi_V},v_{\pi_W}}(\diag(g,1_{V}),1_V)  \text{ in (FJ 2) cases.}
\end{cases}
\]
Then
\[
\RW'_{v_{\pi_V},v_{\pi_W}}(gg')=\RW'_{\pi_{V,W}(g')(v_{\pi_V},v_{\pi_W})}(g).
\]
\end{enumerate}
\end{defin}

From the continuity of $\lambda$ and $\mu$, we can find constants $N_1,N_2,C_1,C_2>0$ and semi-norms $\nu_{X^+}$ on $V_{\sigma_{X^+}}$, $\nu_{V,W}$ on $V_{\pi_V}\wh{\otimes}V_{\pi_W}$ such that
\begin{equation}\label{equ: first estimate}
|\RW_{v_{\sigma_{X^+}}}(g_{X^+})|\leqslant C_1\|g_{X^+}\|^{N_1}\nu_{X^+}(v_{\sigma_{X^+}}).    
\end{equation}
Then, for $g\in \RG_V(F)$ 
\[
\begin{aligned}
|\RB_{v_{\pi_V},v_{\pi_W}}(g_{V},g_W)|&\leqslant C_2\|(g_{V},g_W)\|^{N_2}\nu_{V,W}(v_{\pi_V}\wh{\otimes }v_{\pi_W}), \quad g_W\in \RG_W(F), \text{ and}\\
|\FJ_{v_{\pi_V},v_{\pi_W}}(g_{V},g_W)|&\leqslant C_2\|(g_{V},g_W)\|^{N_2}\nu_{V,W}(v_{\pi_V}\wh{\otimes }v_{\pi_W}),\quad g_W\in \RG_W(F).
\end{aligned}
\]
In particular,
\begin{equation}\label{equ: first estimate B}
|\RW'_{v_{\pi_V},v_{\pi_W}}(g)|\leqslant C_2\|g\|^{N_2}\nu_{V,W}(v_{\pi_V}\wh{\otimes }v_{\pi_W}),\quad g\in \RR_{X',X^+}(E).
\end{equation}

We denote by $K_{X'}$ the maximal compact subgroup of $\GL(X')(E)$, $\RN_{X'}=\RN_{X^+}\cap \GL(X')$, and $A_{X'}$ be the maximal $\BR$-split torus of $\GL(X')(E)$,  
then from Iwasawa decomposition, we have 
\[
\GL(X')(E)=\RN_{X'}(E)\cdot \RA_{X'}(\BR)\cdot K_{X'},
\]
then 
\begin{equation}\label{equ: decomposition of mirabolic}
\RR_{X',X^+}(E)= \RN_{X^+}(E)\cdot (\RA_{X'}(\BR)\cdot K_{X'}\times 1).
\end{equation}

Following \cite[Section 3.2]{jacquet2009archimedean}, we define 
\[
\xi_{h}(g_X)=\prod_{i=1}^{\dim X'-1}(1+(a_ia_{i+1}^{-1})^2),
\]
for  $g_X\in \RR_{X',X^+}(E)$ with the decomposition
\[
g_X=n_{0,X^+}\cdot \diag(a_1,\cdots ,a_{\dim X'},1)\cdot k_{X'}
\]
where $n_{0,X^+}\in \RN_{0,X^+}(E)$, $\diag(a_1,\cdots,a_{\dim X'})\in \RA_{X'}(\BR)$ and $k_{X'}\in K_{X'}$.

Then 
\begin{equation}\label{equ: absolute convergent}
\begin{aligned}
    &\RJ_{s,\mu,\lambda}(v_{\sigma_{X^+}}\boxtimes v_{\pi_W},v_{\pi_V})\\=&\int_{\RN_{0,X^+}(E)\bs \RR_{X',X^+}(E)}\RW_{v_{\sigma_{X^+}}}(g_{X})\RW_{v_{\pi_V},v_{\pi_W}}'(g_{X})|\det(g_{X})|^sdg_{X}\\
=&\int_{K_{X'}}\int_{\RA_{X'}(\BR)}\RW_{v_{\sigma_{X^+}}}(\diag(ak,1))\RW_{v_{\pi_V},v_{\pi_W}}'(\diag(ak,1))|\det(ak)|^sdkda
    \end{aligned}
\end{equation}

From \cite[Proposition 3.1]{jacquet2009archimedean}, there exists $C_3,M,N>0$, such that
\[
|\RW_{v_{\sigma_{X^+}}}(\diag(ak,1))|\leqslant C_3\cdot\xi_{h}(a)^{-N}\|a\|^M\nu_{X^+}(v_{\sigma_{X^+}}),\quad a\in \RA_{X'}(\BR),{k\in K_{X'}}
\]
and from (\ref{equ: first estimate B})
\[
|\RW'_{v_{\pi_V},v_{\pi_W}}(\diag(ak,1))|\leqslant C_2\|a\|^{N_2}\nu_{V,W}(v_{\pi_V}\wh{\otimes }v_{\pi_W})\quad a\in \RA_{X'}(\BR),{k\in K_{X'}}
\]
Then the absolute convergence of (\ref{equ: absolute convergent}) follows from \cite[Lemma 3.5]{jacquet2009archimedean}. So we proved Theorem \ref{thm: main}(1).

\subsubsection{Meromorphic continuation}\label{section: meromorphic continuation}
This section proves the meromorphic continuation of the integral following  \cite{jacquet1990exterior}.
\begin{defin}
For a real Lie group $G$, a function $\xi$ on $G$ is called a \textbf{finite function} if and only if the translations of $\xi$  span a finite-dimensional space. 
\end{defin}
It is easy to verify the following lemma.
\begin{lem}\label{lem: finite functions}
\begin{enumerate}
    \item Every $(E^{\times})^n$-finite function is a linear combinations of 
\[ 
\prod_{i=1}^m\chi_i(u_i)|u_i|^{s_i}(\log|u_i|)^{n_i},
\]
where $\chi_i$  are unitary characters of $E^{\times}$, $s_i\in \BR$ and $n_i\in \BN$.
\item A finite product of $(E^{\times})^m$-finite functions is  $(E^{\times})^m$-finite.
\end{enumerate}
\end{lem}

\begin{defin}\label{defin: fo p}
\begin{enumerate}
    \item For every finite function $\xi$, we denote by $\Fo(\xi)$ the tuple $\{(\chi_i,s_i,n_i)\}_{1\leqslant i\leqslant r}$ such that
\[
\xi=\prod_{i=1}^m\chi_i(u_i)|u_i|^{s_i}(\log|u_i|)^{n_i}.
\]
\item With the notations above, we denote by $p_{\xi}$ the polynomial
\[
p_{\xi}(s)=\prod_{i=1}^m(s+s_i)^{n_i}.
\]
\end{enumerate}    
\end{defin}

For a Schwartz function $\Phi\in \CS((E^{\times})^m)$ and a tuple $\{(\chi_i,s_i,n_i)\}_{1\leqslant i\leqslant r}$, we define the following Mellin transform
    \begin{equation}\label{equ: continuity}
\CJ^{\{(\chi_i,s_i,n_i)\}_{1\leqslant i\leqslant m}}(s,\Phi)=\int_{(\BR^{\times})^{m}}\Phi(a_1,\cdots,a_m)\prod_{i=1}^m\chi_i(a_i)|a_i|^{s+s_i}(\log|a_i|)^{n_i}d(a_1,\cdots,a_m).
    \end{equation}
Here 
\begin{equation}
    d(a_1,\cdots,a_m)=da_1^{\times}\cdots da_m^{\times}=\frac{da_1}{|a_1|}\cdots \frac{da_m}{|a_m|}.
\end{equation}
\begin{pro}\label{equ: basic meromorphic continuation}
For a Schwartz function $\Phi\in \CS((E^{\times})^m)$ and a tuple $\{(\chi_i,s_i,n_i)\}_{1\leqslant i\leqslant r}$, the integral 
\[
\prod_{i=1}^m(s+s_i)^{n_i}\CJ^{\{(\chi_i,s_i,n_i)\}_{1\leqslant i\leqslant r}}(s,\Phi)
\] 
has an analytic continuation to an entire function.
\end{pro}
\begin{proof} When $\Re(s)$ large enough, $\prod_{i=1}^m\chi_i(a_i)|a_i|^{s+s_i}(\log|a_i|)^{n_i}$ is a function of modernate growth on $(\BR^{\times})^r$. Hence, the integral (\ref{equ: continuity}) defines a continuous functional on $\CS((\BR^{\times})^r)$, so we may assume \[\Phi(a_1,\cdots,a_r)=\prod_{i=1}^r\Phi_i(a_i),\]
then
\[
\CJ^{\{(\chi_i,s_i,n_i)\}_{1\leqslant i\leqslant r}}(s,\Phi)=\prod_{1\leqslant i\leqslant r}\CJ^{\{(\chi_i,s_i,n_i)\}}(s,\Phi_i).
\]

From \cite[Section 4.2]{igusa1978lectures}, for every $1\leqslant i\leqslant r$,  $(s+s_i)^{n_i}\CJ^{\{(\chi_i,s_i,n_i)\}}(s,\Phi_i)$ can be extended to an entire function on $s\in\BC$.
\end{proof}

From \cite[Proposition 4.1]{jacquet1990exterior}, we have an expansion of the Whittaker functions on $\RA_{r}(\BR)\times 1\subset \RR_{X',X^+}(E)$.
\begin{equation}\label{equ: expansion whittaker}
\RW_{v_{\pi_{X^+}}}(\diag(a_1,\cdots,a_r,1))=\sum_{\xi_{X^+}\in \Sigma_{X^+}}\phi_{\xi_{X^+},v_{\pi_{X^+}}}(a_1,\cdots,a_r)\xi_{X^+}(a_1,\cdots,a_r),
\end{equation}
where $\Sigma_{X^+}$ is a finite set of finite functions of $(\BR^{\times})^r$ that is only dependent on $\sigma_{X^+}$ and $\phi_{\xi_{X^+}}$ are Schwartz functions in $\CS((\BR^{\times})^r)$.

We can generalize the proof for \cite{jacquet1990exterior} (see also \cite[Proposition 3.3]{soudry1993rankin}) to obtain  an expansion of $\RW'_{v_{\pi_{V}},v_{\pi_W}}$  on $(\RA_{X'}(\BR)\times 1)\times 1\subset \RR_{X',X^+}(E)$.
\begin{lem}\label{lem: expansion pseudo-Whittaker}
We have the following expansion
\begin{equation}
\RW'_{v_{\pi_{V},v_{\pi_W}}}(\diag(a_1,\cdots,a_r,1))=\sum_{\xi_{V,W}\in \Sigma_{V,W}}\phi_{\xi_{V,W},v_{\pi_{V}},v_{\pi_W}}(a_1,\cdots,a_r)\xi_{V,W}(a_1,\cdots,a_r),
\end{equation}
where $\Sigma_{V,W}$ is a finite set of finite functions of $(\BR^{\times})^r$ that is 
only dependent on $\pi_V,\pi_W$, and $\phi_{\xi_{X^+}}$ are Schwartz functions in $\CS((\BR^{\times})^r)$.
\end{lem}

The proof in \cite{jacquet1990exterior} uses properties of Jacquet models of the reductive groups. From Lemma \ref{lem: Jac of Jaco}(1), the Jacquet modules of Jacobi groups can be computed from that of the reductive groups. To set up the proof uniformly as in \cite{jacquet1990exterior},  we work on the reductive group $\RG_V\times \RG_W$ and Jacquet modules on it.
\begin{defin}
\begin{enumerate}
    \item For a representation $\pi$ of a  reductive group $\RG$, we denote by $\pi^{\alg}$ the $(\Fg_{\BC},K)$-module associated to $\pi$. Here $\Fg_{\BC}=\Fg(F)\otimes_{\BR}\BC$ is the complexified Lie algebra of $\RG(F)$.
    \item For a parabolic subgroup $\RP=\RL\rtimes \RN$ of $\RG$, we define the Jacquet module 
    \[
\Jac_\RP^{\alg}(\pi^{\alg})=\pi^{\alg}/\pi(\Fn_{\BC})\pi^{\alg}.
    \]
\end{enumerate}    
\end{defin}

\begin{defin}
Let $(\RG,\RH,\xi)$ be a Bessel or Fourier-Jacobi triple associated to $(V,W)$ and $\pi=\pi_V\boxtimes \pi_W\in \Pi_{\CaWa}(\RG)$.
\begin{enumerate}
    \item In (B) cases, we set $\pi_W'=\pi_W$; In (FJ 1)(FJ 2) cases,  we set $\pi_W'$ as the irreducible representation of $\wt{\RG}_W$ in Lemma \ref{lem: classification of representations}  such that 
    \[
\pi_W=\pi_W'\otimes\omega_{W,\psi_F}.
    \]
    \item For $1\leqslant j\leqslant r$, we define 
    \[
    \RP_j=\begin{cases}
       \RP_{V,j}\times \wt{\RP}_{W} & \text{ in (B) and (FJ 1)  cases},\\
      \RP_V\times \wt{\RP}_{W,j} &\text{ in  (FJ 2) cases}.
    \end{cases}
    \]
    where $\RP_{V,j},\wt{\RP}_{W,j}$ is defined as in Section \ref{section: Bessel Fourier-Jacobi}, then the Jacquet module 
    \[
    \Jac_{\RP_j}^{\alg}((\pi_V\boxtimes \pi_W')^{\alg})=\begin{cases}
      \Jac_{\RP_{V,j}}^{\alg}(\pi_V^{\alg})\boxtimes \pi_W^{\prime,\alg} & \text{ in (B) and (FJ 1)  cases},\\
      \pi_V^{\alg}\boxtimes \Jac^{\alg}_{\wt{\RP}_{W,j}}(\pi_W^{\prime,\alg}) & \text{ in  (FJ 2) cases}.
    \end{cases}
    \]
\end{enumerate}
\end{defin}

\begin{proof}[Proof for Lemma \ref{lem: expansion pseudo-Whittaker}]

Consider the Levi decompostion $\RP_j=\RL_j\rtimes \RN_j$, where
\[
\Fl_j(F)=(E^{\times})^j\times \Fg_j(F).
\]
for some reductive $\Fg_j$. We consider the generalized eigenspace decomposition of $\Jac_{\RP_j}^{\alg}((\pi_V\boxtimes \pi_W')^{\alg})$ with respect to $(E^{\times})^j$
\[
\Jac_{\RP_j}^{\alg}((\pi_V\boxtimes \pi_W')^{\alg})=\bigoplus_{\mu\in \Sigma_j}\Jac_{\RP_j}^{\alg}((\pi_V\boxtimes \pi_W')^{\alg})_{\mu}.
\]

From \cite[Lemma 4.3.1]{wallach1988real},  $\Jac_{\RP_j}^{\alg}((\pi_V\boxtimes \pi_W)^{\alg})$ is a  finite-length $\CU((\Fl_j)_{\BC})$-module, then there exists an integer $k$ such that 
\[
(\pi_{V,W}(a)-\mu(a))^k\Jac((\pi_V\boxtimes \pi_W')^{\alg})_{\mu}=0,\quad a\in (E^{\times})^j, \mu\in \Sigma_j.
\]

Define $\CK$ the space of functions $\varphi$ on $(\BR^{\times})^n$ in to form of 
\[
\varphi(a_1,\cdots,a_n,1)=\begin{cases}
    \RW_{v_{\pi_V}^{\alg},v_{\pi_W}^{\alg}}(\diag(a_1,\cdots,a_n,1))& \text{ in (B) cases},\\
   \RW_{v_{\pi_V}^{\alg},v_{\pi_W}^{\alg}\otimes \varphi}(\diag(a_1,\cdots,a_n,1)) & \text{ in (FJ 1)(FJ 2) cases},\\
\end{cases}\]
where $v_{\pi_V}^{\alg}\in \pi_V^{\alg}$, $v_{\pi_{W}}^{\alg}\in \pi_W^{\alg}$, $\varphi\in \omega_{W,\psi_F}$.

It follows verbatim from \cite[\S 4.1]{jacquet1990exterior} when $E=\BC$  and from \cite[\S 3.3]{soudry1993rankin} when $E=\BR$ that $\CK$ has the following properties.
\begin{enumerate}
\item Every function $\varphi\in\CK$ is smooth on $(E^{\times})^r$ and is bounded by 
\[
C\cdot\prod_{j=1}^r(1+|a_j|^2+|a_j^{-1}|^2)^N;
\]
for some $C>0$ and a positive integer $N$. 
    \item The space $\CK$ is closed under 
    \begin{enumerate}
        \item Multiplication by $a_j$, $\bar{a}_j$;
        \item Euler operators \[D_j=\begin{cases}
  a_j\frac{\partial f}{\partial a_j} &\text{ when }E=\BR,\\
a_j\frac{\partial f}{\partial a_j}+ \bar{a}_j\frac{\partial f}{\partial \bar{a}_j}  &\text{ when }E=\BC;\end{cases}\]
\item The difference between the holomorphic and antiholomorphic derivatives
\[
R_j=\frac{\partial f}{\partial a_j}- \frac{\partial f}{\partial \bar{a}_j}.
\]
    \end{enumerate}
    \item Each $\varphi\in \CK$ has a decomposition 
    \[
    \varphi=\sum_{\mu\in \Sigma_j}\varphi_{\mu}
    \]
    such that 
    \begin{enumerate}
        \item When $E=\BR$, there exists $\theta\in \CK$ such that
    \[
    (D_j-\mu)^k\varphi_x=a_j\theta.
    \]
    \item When $E=\BC$, there exists $\theta_1,\theta_2\in\CK$ such that
    \[
    (D_j-\mu)^k\varphi_x=a_j\theta_1+\bar{a}_j\theta_2.
    \]
    \end{enumerate}
\end{enumerate}

Following the general situations discussed in \cite[\S 4.2]{jacquet1990exterior}, there is a finite set $\Sigma_{V,W}$ of finite functions that is only dependent on $\Sigma_j$'s (thus only dependent on $\pi_V^{\alg},\pi_W^{\alg}$) such that  $\RW'_{v_{\pi_{V}}^{\alg},v_{\pi_W}^{\alg}}$ has the expansion 
\[
\sum_{\xi_{V,W}\in\Sigma_{V,W}}\phi_{\xi_{V,W}.{v_{\pi_V}^{\alg},v_{\pi_W}^{\alg}}}(a_1,\cdots,a_n)\xi_{V,W}(a_1,\cdots,a_n).
\]

Following  the computation of \cite[\S 4.3]{jacquet1990exterior}, the method in \cite{casselman1989canonical} can be applied to extend the results to $\RW'_{v_{\pi_{V}},v_{\pi_W}}$ for $v_{\pi_V}\in \pi_V$ and $v_{\pi_W}\in \pi_W$.

\end{proof}

\begin{proof}[Proof for the meromorphic continuation]
We define
\[
\wt{\RJ}_{s,\mu,\lambda}(v_{\sigma_{X^+}},v_{\pi_V},v_{\pi_W})=\int_{\RA_{X'}(\BR)}\RW_{v_{\sigma_{X^+}}}(\diag(a,1))\RW'_{v_{\pi_V},v_{\pi_{W}}}(\diag(a,1))|\det(a)|^sda.
\]

Then we expand the Whittaker functional and the Bessel functional using (\ref{equ: expansion whittaker}) and Lemma \ref{lem: expansion pseudo-Whittaker} and obtain that
\begin{equation}\label{equ:expansion}
 \begin{aligned}
&\wt{\RJ}_{s,\mu,\lambda}(v_{\sigma_{X^+}},v_{\pi_W},v_{\pi_V})\\
=&\int_{\RA_{X'}(\BR)}\RW_{v_{\pi_{X^+}}}(\diag(a_1,\cdots,a_r,1))\RW'_{v_{\pi_V},v_{\pi_W}}(\diag(a_1,\cdots,a_{r},1))|a_1\cdots a_r|^sd(a_1,\cdots,a_r)   \\
=&\sum_{\substack{\xi_{X^+}\in \Sigma_{X^+}\\\xi_{V,W}\in \Sigma_{V,W}}}\int_{\RA_{X'}(\BR)}\phi_{\xi_{X^+},v_{\pi_{X^+}}}\phi_{\xi_{V,W},v_{\pi_V},v_{\pi_W}}\xi_{X^+}\xi_{V,W}(a_1,\cdots,a_r)|a_1\cdots a_r|^sd(a_1,\cdots,a_r).
\end{aligned}   
\end{equation}

From Proposition \ref{equ: basic meromorphic continuation}, the function
\[
\begin{aligned}
p_{{\xi_{X^+}\xi_{V,W}}(s)}(s)\int_{\RA_{X'}(\BR)}\phi_{\xi_{X^+},v_{\pi_{X^+}}}\phi_{\xi_{V,W},v_{\pi_V},v_{\pi_W}}\xi_{X^+}\xi_{V,W}(a_1,\cdots,a_r)|a_1\cdots a_r|^sd(a_1,\cdots,a_r)    
\end{aligned}
\]
 has an analytic continuation and we denote it by $\Psi_{\xi_{X^+},v_{\pi_{X^+}},\xi_{V,W},(v_{\pi_V},v_{\pi_W})}(s)$.

From the Fubini's theorem and (\ref{equ: absolute convergent}), when $\Re(s)$ is large enough, we have
\[
\begin{aligned}
&\RJ_{s,\mu,\lambda}(l(v_{\sigma_{X^+}},v_{\pi_V}, v_{\pi_W}))\\
=&\int_{K_r}\int_{\RA_{X'}(\BR)}\RW_{v_{\sigma_{X^+}}}(\diag(ak,1))\RW'_{v_{\pi_V},v_{\pi_W}}(\diag(ak,1))|\det(ak)|^sdkda.
\end{aligned}
\]

Then, when $\Re(s)$ is large enough,
\[
\begin{aligned}
&\RJ_{s,\mu,\lambda}(l(v_{\sigma_{X^+}},v_{\pi_V},v_{\pi_W}))\\
=&\int_{K_r}\wt{\RJ}_{s,\mu,\lambda}(\sigma_{X^+}(\diag(k,1))v_{\sigma_{X^+}},\pi_{V,W}(\diag(k,1))(v_{\pi_V},v_{\pi_W}))dk\\
=&\sum_{\substack{\xi_{X^+}\in \Sigma_{X^+}\\\xi_{V,W}\in \Sigma_{V,W}}}p_{\xi_{X^+},\xi_{V,W}}(s)^{-1}\int_{K_r}{\Psi_{\xi_{X^+},\sigma_{X^+}(\diag(k,1))v_{\pi_{X^+}},\xi_{V,W},\pi_{V,W}(\diag(k,1))(v_{\pi_V},v_{\pi_W})}(s)}dk.
\end{aligned}
\]

Since $K_r$ is compact, and, for $k\in K_r$, the functions
\[
{\Psi_{\xi_{X^+},\sigma_{X^+}(\diag(k,1))v_{\pi_{X^+}},\xi_{V,W},\pi_{V,W}(\diag(k,1))(v_{\pi_V},v_{\pi_W})}}
\]
are entire, the integral
\[\int_{K_r}{\Psi_{\xi_{X^+},\sigma_{X^+}(\diag(k,1))v_{\pi_{X^+}},\xi_{V,W},\pi_{V,W}(\diag(k,1))(v_{\pi_V},v_{\pi_W})}(s)}dk\]
gives an entire function of $s$. Then
$\RJ_{s,\mu,\lambda}(l(v_{\sigma_{X^+}},v_{\pi_V},v_{\pi_W}))$  has an meromorphic continuation to the whole complex plane. So we complete the proof for Theorem \ref{thm: main}(2).
\end{proof}
\subsection{The integral over $ \RP^+\cap\RH^+(F)\bs \RH^+(F)$}\label{section: second integral}

In (B) cases, from \cite[Proposition 4.17]{gourevitch2019analytic}, one can construct  a nonzero functional in 
\[
\Hom_{\RH^+}(\Ind^{\CS,G^+}_{P^+}(l(|\det|^s\sigma_{X^+},\pi_V,\pi_W)),1_{\RH^+})
\]
from a nonzero functional in  
\[
\Hom_{\RP^+\cap \RH^+}(\delta_{\RP^+\cap \RH^+}\delta_{\RH^+}^{-1}\otimes l(|\det|^s\sigma_{X^+},\pi_V,\pi_W),1_{\RP^+\cap \RH^+})
\] following \cite[Section 5.5]{chen2021local}, where 
 $\delta_{\RP^+\cap \RH^+},\delta_{\RH^+}$ are the modular character of $\RP^+\cap \RH^+(F)$, $\RH^+(F)$, respectively. In this section, we follow the work of D. Gourevitch, S. Sahi, and E. Sayag to generalize their results so that this approach also applies to (FJ 1)(FJ 2) cases.

In all situations, we have $\delta_{\RH^+}=1$, and from Lemma \ref{lem: decomposition of the stabilizer group}, we can compute $\delta_{\RP^+\cap \RH^+}$ to obtain the following lemma. 
\begin{lem}\label{lem: s_0}
There exists $s_0\in \BR$ such that
\[
\delta_{\RP^+\cap \RH^+}\delta_{\RH^+}^{-1}\otimes l(|\det|^s\sigma_{X^+},\pi_V,\pi_W)=l(|\det|^{s+s_0}\sigma_{X^+},\pi_V,\pi_W).
\]
\end{lem}

From Lemma \ref{lem: geometric for refine 2}, the complement $\RG^+(F)-\RH^+(F)\RP^+(F)$ is the zero set of a polynomial $f^+$ on $\RG^+(F)$ and $f^+$ is left $\RH^+$-invariant and right $(\RP^+,\psi_{\RP^+})$-equivariant for an algebraic character $\psi_{\RP^+}$ of $\RP^+$, where $\psi_{\RP^+}=|\det\circ p_{\GL}|^2$.

Given $\mu^+\in \Hom_{\RP^+\cap \RH^+}(l(|\det|^s\sigma_{X^+},\pi_V,\pi_W),1_{\RP^+\cap \RH^+})$, we construct
\[
\RJ_{s^+,\mu^+}^+\in \Hom_{H^+}(\RI(|\det|^{s_0+s^+}\sigma_{X^+},\pi_V,\pi_W),1_{\RH^+})
\]
by the integral
\[
\RJ_{s^+,\mu^+}^+(\varphi_{s^+})=\int_{\RH^+\cap \RP^+(F)\bs \RH^+(F)}\mu^+(\varphi_{s^+} |f^+|^{s^+}(h^+))dh^+,\quad \varphi_{s^+}\in \RI(|\det|^{s_0+2s^+}\sigma_{X^+},\pi_V,\pi_W).
\]

Recall that, by Definition \ref{def: Schwartz induction}, the map
\[\RT_{\RP^+,s}:\CS(\RG^+,l(\sigma_{X^+}, \pi_V, \pi_W))\to \Ind_{\RP^+}^{\CS,\RG^+}(l(|\det|^{s_0+s^+}\sigma_{X^+},\pi_V,\pi_W))\]
\[
\varphi\mapsto \int_{\RP^+}l(|\det|^{s_0+s^+}\sigma_{X^+}, \pi_V,\pi_W)((p^{+})^{-1})\varphi(p^+g^+)dp^+
\]
is surjective. Then 
\begin{equation}\label{equ: subspace of Schwartz distribution}
\begin{aligned}
\Hom_{\RH^+}(\RI(|\det|^{s+s_0+s^+}\sigma_{X^+},\pi_V,\pi_W),1_{H^+})&\to\CS^*(\RG^+,l(\sigma_{X^+}\boxtimes \pi_W\boxtimes \pi_V)^*)^{(\RP^+,l(|\det|^s\sigma_{X^+},\pi_V,\pi_W))}\\ 
J^+&\mapsto J^+\circ \RT_{P^+,s}
\end{aligned}
\end{equation}
is an injective map, where $\CS^*(\RG^+,l(\sigma_{X^+},\pi_V,\pi_W))$ is the continuous dual of the space of Schwartz maps from $\RG^+(F)$ to $l(\sigma_{X^+}, \pi_V,\pi_W)$, and the superscript denotes the  $(\RP^+,l(\sigma_{X^+},\pi_V,\pi_W))$-equivariant part.

\begin{thm}\label{thm: main 2}
\begin{enumerate}
    \item The integral   $\RJ_{s^+,\mu^+}^+$  is absolutely convergent when $\Re(s^+)$ is large enough.
    \item 
      $\RJ_{s^+,\mu^+}$ has a meromorphic continuation to the complex plane $\BC$ in 
      \[\CS^*(\RG^+,l(\sigma_{X^+},\pi_V,\pi_W))^{(\RP^+,l(\sigma_{X^+},\pi_V,\pi_W))}.
      \] 
\end{enumerate} 
\end{thm}
\subsubsection{Absolute convergence}
We prove the absolute convergence following \cite[Corollary 2.22]{gourevitch2019analytic}.

Recall  that $\RP^+\RH^+$ is Zariski open in $\RG^+$.  We denote by 
\[
\CX=\RP^+(F)\bs \RG^+(F),\ \CU=\RP^+(F)\bs \RP^+\RH^+(F),\ \CZ=\CX-\CU
\]  

We denote by $\CS_{\CZ,n}(\CX)$ the space of Schwartz functions on $\CX$ such that all their $k$-th derivatives vanish on $\CZ$ for $k\leqslant n$. From \cite[Theorem 5.4.3]{aizenbud2008schwartz}, we have the following equation 
\[
\CS(\CU)=\varprojlim_n\CS_{\CZ,n}(\CX)\]

We take $\int_{\CU}du$ as a continuous functional on $\CS(\CU)$, then
\[
\int_{\CU}du\in (\varprojlim_n\CS_{\CZ,n}(\CX))^*= \varinjlim_n \CS_{\CZ,n}(\CX)^*
\]
\begin{lem}\label{lem: integral}
There exists $n_{\CX,\CU}>0$ such that $\int_{\CU} g(x) dx$ is absolutely convergent when $g\in \CS_{\CZ,n_{\CX,\CU}}(\CX)$. 
\end{lem}

\begin{proof}[Proof for Thereom \ref{thm: main 2}(1)]

From Proposition \ref{pro: alt Schwartz induction}
\[
\Ind^{\CS,\RG^+}_{\RP^+}(l(|\det|^s\sigma_{X^+},\pi_V,\pi_W))=\Gamma^{\CS}(\CX,\CE_{l(|\det|^s\sigma_{X^+},\pi_V,\pi_W)})
\]
so $\mu^+(\Ind^{\CS,\RG^+}_{\RP^+}(l(|\det|^s\sigma_{X^+},\pi_V,\pi_W)))\subset \CS(\CX)$.


For $s^+\geqslant n_{\CX,\CU}+1$, one can verify by chain rule that
 \[
 (f^+)^{s^+}\mu^+(\varphi_{s^+})\in \CS_{\CZ,n_{\CX,\CU}}(\CX)
 \]
 
 Therefore, from Lemma \ref{lem: integral},  
 \[
 \int_{\CU}|\mu^+(|f^+|^{s^+}\varphi_{s^+})|(u)du= \int_{\CU}|f^+|^{s^+}|\mu^+(\varphi_{s^+})|(u)du
 \]
 is absolutely convergent for $s^+\geqslant n_{\CX,\CU}+1$.
\end{proof}
\subsubsection{Meromorphic continuation}
This section proves that $\RJ_{s^+,\mu^+}^+$ has a meromorphic continuation to the complex plane following the proof for \cite[Proposition 4.9]{gourevitch2019analytic}.
\begin{lem}\label{lem: surjection of minimal parabolic induction}
There exists a subgroup $\RP_0$ and a finite-dimensional representation $\rho_0$ of $\RP_0$ such there is a surjection
    \[
    \Ind_{\RP_0^+}^{\CS,\RP^+}(\rho_0^+)\twoheadrightarrow l(\sigma_{X^+},\pi_V,\pi_W)
    \] 
\end{lem}
\begin{proof}
We set $\RP_{0,X^+}=\RL_{0,X^+}\rtimes\RN_{0,X^+}$ to be a minimal parabolic subgroup of $\GL(X^+)$ over $E$ and $\RP_{0,V}=\RL_{0,V}\rtimes \RN_{0,V}$, $\RP_{0,W}=\RL_{0,W}\rtimes \RN_{0,W}$ to be a minimal parabolic subgroup of $\RG_V$, $\RG_W$ over $F$.
since $\RG_V$, $\RG_W$,  $\Res_{E/F}\GL(X^+)$ are reductive, there are surjections
\[
\Ind_{\RP_{0,X^+}(E)}^{\GL(X^+)(E)}(\rho_{0,X^+})\twoheadrightarrow
\sigma_{X^+},\quad \Ind_{\RP_{0,V}(F)}^{\RG_V(F)}(\rho_{0,V})\twoheadrightarrow
\sigma_{V},
\]
where $\sigma_{X^+}\in \Pi_{\CaWa}(\Res_{E/F}\RL_{0,X^+})$, $\rho_{0,V}\in \Pi_{\CaWa}(\RL_{0,V})$ are finite-dimensional representations.

In (B) cases, there is also a surjection
\[
\Ind_{\RP_{0,W}(F)}^{\RG_W(F)}(\rho_{0,W})\twoheadrightarrow
\sigma_{W}
\]
for finite-dimensional $\rho_{0,W}\in \Pi_{\CW}(\RL_{0,W})$. Then the lemma holds for
\[
\RP_0^+=(\Res_{E/F}\RP_{0,X^+}\times \RP_{0,W})\times\RP_{0,V}\subset \RP^+
\]
\[
\rho_0^+=(\rho_{0,X^+}\boxtimes\rho_{0,W})\boxtimes \rho_{0,V}
\]
In (FJ 1)(FJ 2) cases, from Lemma \ref{lem: classification of representations}, $\pi_W=\pi_W'\otimes \omega_{W,\psi}$ for irreducible Casselman-Wallach representation $\pi_W'$ of $\wt{\RG}_W$.
Since $\wt{\RG}_W$ is reductive, there is a surjection
\[
\Ind_{\RP_{0,W}}^{\RG_W}(\rho_{0,W}')\twoheadrightarrow \pi_W'
\]
for minimal parabolic subgroup $\RP_{0,W}$ and finite-dimensional representation $\rho_{0,W}'$.

There is a decomposition $W=X_0\oplus Y_0$ such that $X_0,Y_0$ are totally isotropic,
then Schrodinger's model of Weil representation gives
\[
\omega_{W,\psi}=\Ind_{\CH(X_0)}^{\CS,\CH(W)}(\psi)
\]
Then  
\[
\Ind_{\RP_{0,W}}^{\RG_W}(\rho_{0,W}')\wh{\otimes} \omega_{W,\psi}=\Ind_{\RP_{0,W}\rtimes \CH(X_0)}^{\RG_W^J}(\rho_{0,W}'\otimes \psi)
\]
Recall that $\RP_{0,W}^J=\RP_{0,W}\rtimes \CH(X_0^{\perp})=\RP_{0,W}\rtimes \CH(X_0)$, then 
\[
\Ind_{\RP_{0,W}^J}^{\RG_W}(\rho_{0,W})\twoheadrightarrow \pi_W
\]
is a surjection for $\rho_{0,W}=\rho_{0,W}'\otimes \psi$. Then the lemma holds for
\[
\RP_0^+=\gamma^{-1} \RP_{0,V}\times (\RP_{0,X^+}\times\RP_{0,W}^J)\gamma\subset \RP^+
\]
\[
\rho_{0}^+=\rho_{0,V}\boxtimes (\rho_{0,X^+}\boxtimes \rho_{0,W})
\]
in (FJ 1) cases, and 
\[
\RP_0^+=\gamma^{-1}(\RP_{0,X^+}\times\RP_{0,V})\times \RP_{0,W}^J\gamma\subset \RP^+
\]
\[
\rho_{0}^+=(\rho_{0,X^+}\boxtimes \rho_{0,V})\boxtimes \rho_{0,W}
\]
in (FJ 2) cases.
\end{proof}
To prove Theorem \ref{thm: main 2}, we need the following geometric structure on $\RP_0^+(F)\bs \RG^+(F)/\RH^+(F)$.
\begin{lem}\label{lem: finite double coset}
  $\RP_0^+(F)\bs \RG^+(F)/\RH^+(F)$ has finitely many double cosets.  
\end{lem}
\begin{proof}
  In (B) cases, one can verify as in \cite[\S 3]{chen2021local} that $\RG^+/\RH^+$ is absolutely spherical, which implies the finiteness of the double coset.  In (FJ 2) cases, as $\RP_{0,W}^J$ contains a maximal unipotent subgroup of $\RG_{V^+}$, from Bruhat decomposition, $\RP_{0,V^+}$ has finitely many orbits on $\RG_{V^+}/\RP_{0,W}^J$, so we have the finiteness.

 In (FJ 1) cases, we take a basis $\{z_i^+\}_{i=\pm 1,\cdots,\pm \frac{\dim_E W^+}{2}}$ of $W^+$ such that $\RP_{0,W^+}$ stabilizes the flag
 \[
 0\subset X_1^+\subset \cdots\subset X_{\frac{\dim_E W^+}{2}}^+=X_0
 \]
 the Bruhat decomposition gives a decomposition 
\[
\RG_{W^+}=\RP_{0,W^+}w_i\RP_{0,W^+},\quad i=1,\cdots,l_1
\]
    
For $p'\in \RP_{0,V}(F)=\RP_{0,W^+}(F)$, $g\rtimes v\in \RG_{W^+}\rtimes W^+ =\RG_{W^+}\rtimes \CH(W^+)/\BG_a$, and $p\rtimes x\in \RP_{0,W^+}\rtimes X_0=\RP_{0,W^+}\rtimes \CH(X_0)/\BG_a$, we have
\[
p'(g,v)(p,x)=(p'gp,v^px)
\]
where $v^p$ denotes the action of $p$ on $v$ by standard representation of $\RG_{0,W^+}$. Then the double cosets $\RP_{0,V}(F)\bs\RG_{W^+}^J(F)/\RP_{0,W^+}^J(F)$ are $\CO_{i,S}$ where $i=1,\cdots,l_1$ and $S\subset \{1,\cdots,\frac{\dim_E W^+}{2}\}$,
\[
\CO_{i,S}=\{(g,v): g\in \RP_{0,W^+}(F)w_i\RP_{0,W^+}(F), \langle v,z_j\rangle =0, \text{ for } i\in S,  \langle v,z_j\rangle \neq 0, \text{ for } i\notin S\}
\]
Therefore, in (FJ 1) cases, there are also finitely many double cosets.
\end{proof}
\begin{proof}[Proof for Theorem \ref{thm: main 2}(2)]
By  definition,
\[
\RJ_{s^+,\mu^+}^+\circ \RT_{\RP^+,s^+}\in \CS^*(\RG^+(F),l(\sigma_{X^+}, \pi_V, \pi_W))
\]

From Lemma \ref{lem: surjection of minimal parabolic induction}, there exists a subgroup $\RP_0^+$ of $\RG^+$ and a finite-dimensional representation $\rho_0^+$ of $\RP_0^+(F)$ such there is a surjection
    \[
\RF_{\sigma_{X^+},\pi_V,\pi_W}:\Ind_{\RP_0^+}^{\CS,\RP^+}(\rho_0^+)\twoheadrightarrow l(\sigma_{X^+},\pi_V,\pi_W)
    \] 
and by Definition \ref{def: Schwartz induction}, we have a surjection 
\[
\RT_{\RP_0^+,\rho_0^+}:\CS(\RP^+(F),\rho_0^+)\twoheadrightarrow \Ind_{\RP_0^+}^{\CS,\RP^+}(\rho_0^+)
\]

 We construct an injective map with the surjection 
 \[
\RF_{\sigma_{X^+,\pi_V,\pi_W}}\circ\RT_{P_0^+,\rho_0^+}:\CS(\RP^+(F),\rho_0^+)\to l(\sigma_{X^+},\pi_V,\pi_W)
 \] following \cite[Proposition 4.9]{gourevitch2019analytic},
\[
\Phi:\CS^*(\RG^+(F),l(\sigma_{X^+},\pi_V,\pi_W))\hookrightarrow \CS^*(\RG^+\times \RP^+(F),\rho_0^+)
\]

We denote by $\tau$ to be the kernel of $\RF_{\sigma_{X^+,\pi_V,\pi_W}}\circ\RT_{P_0^+,\rho_0^+}$ and by tensoring  with the nuclear Fr\'echet space $\CS(\RG^+(F))$, we obtain an exact sequence (\cite[Lemma A.3]{casselman2000bruhat})
\[
0\to \CS(\RG^+(F))\wh{\otimes}\tau \to 
\CS(\RG^+(F))\wh{\otimes}\CS(\RP^+(F),\rho_0^+) \to \CS(\RG^+(F))\wh{\otimes}l(\sigma_{X^+},\pi_V,\pi_W) \to 0
\]

Notice that \[
\CS(\RG^+(F))\wh{\otimes}\CS(\RP^+,\rho_0^+)=\CS(\RG^+\times \RP^+(F),\rho_0^+), 
\] 

By dualizing the exact sequence, we obtain an exact sequence
\[
0\to \CS^*(\RG^+(F),l(\sigma_{X^+},\pi_V,\pi_W))\overset{\Phi}{\to}\CS^*(\RG^+\times \RP^+(F),\rho_0^+)\to \CS^*(\RG^+(F),\tau)\to 0
\]

The map $\Phi$ is defined as in the exact sequence. Then we can prove the meromorphic continuation of $\RJ_{s^+,\mu^+}$ by showing that
\begin{enumerate}
    \item The family $\Phi(\RJ_{s^+,\mu^+})$ has a meromorphic continuation;
    \item The meromorphic family lies in $\Im(\Phi)$.
\end{enumerate}
We first \cite[Lemma 3.3]{gourevitch2019analytic} to prove $\Phi(\RJ_{s^+,\mu^+}^+\circ \RT_{\RP^+,s^+})$ is holonomic.

Hence, in each situation $\RP_0^+(F)\bs \RG^+(F)/\RH^+(F)$ has finitely many double cosets 
\[
\RP_0^+(F)x_i\RH(F),\quad i=0,1,\cdots,k
\] and we may assume $\RP_0^+(F)x_0\RH^+(F)$ is open. Then from \cite[Lemma 3.3]{gourevitch2019analytic}(1), $\Phi(\RJ_{s^+,\mu^+}^+\circ \RT_{\RP^+,s^+})|_{\RP_0^+(F)x_0\RH^+(F)}$ is holonomic and then one can prove by induction from \cite[Lemma 3.3]{gourevitch2019analytic}(2) that $\Phi(\RJ_{s^+,\mu^+}^+\circ \RT_{\RP^+,s^+})$ is holonomic. Therefore, Point (1) follows from \cite[Theorem 3.4]{gourevitch2019analytic}.

Point (2) follows the proof in \cite[Corollary 4.9]{gourevitch2019analytic} verbatim using  the fact that $\Im(\Phi)$ can be characterized by $(\CS(\RG^{+})\wh{\otimes}\tau)^{\perp}$ in $\CS^*(\RG^+\times \RP^+(F),\rho_0^+)$.
\end{proof}
\subsection{
Proof for Theorem \ref{thm: main intro}(3)}\label{section: pf for first multi formula}
In this section, we prove Theorem \ref{thm: main intro}(3) based on Theorem \ref{thm: main} and Theorem \ref{thm: main 2}.

\begin{proof}[Proof for Theorem \ref{thm: main intro}(3)]
From Theorem \ref{thm: multiplicity basic}, $m(\pi_V\boxtimes \pi_W)\leqslant 1$, so it suffices to show that there is a nonzero element in $\Hom_{\RH^+}(\RI(\sigma_{X^+},\pi_V,\pi_W),1_{\RH^+})$  when $\Hom_{\RH}(\pi_V\boxtimes \pi_W,\xi_{\RH})\neq 0$. 

For given nonzero $\mu\in \Hom_{\RH}(\pi_V\boxtimes \pi_W,\xi_{\RH})$, we construct meromorphic family $\RJ_{s,\mu,\lambda}$ for $\Re(s)$ large enough. From Theorem \ref{thm: main}(2), it has meromorphic continuation  \[
\RJ_{s,\mu,\lambda}\in \Hom_{\RP^+\cap \RH^+}(l(|\det|^s\sigma_{X^+},\pi_V,\pi_W),1_{\RP^+\cap \RH^+}).
\] Then we take $\mu^+$ to be the principal term
\[
\lim_{s\to 0}s^{-n}\RJ_{s,\mu,\lambda}
\]
where $n$ is the order of zero at $s=0$ for $\RJ_{s,\mu,\lambda}$. From \cite[\S 2]{gourevitch2019analytic},  
\[
 \mu^+\in \Hom_{\RP^+\cap \RH^+}(l(\sigma_{X^+},\pi_V,\pi_W),1_{\RP^+\cap \RH^+})
\]

We use $\mu^+$ to construct a nonzero meromorphic family $\RJ_{s^+,\mu^+}^+$ as in  Section \ref{section: second integral}. From  Theorem \ref{thm: main 2}(2), we has its meromorphic continuation 
\[
\RJ_{s^+,\mu^+}^+\in\Hom_{H^+}(\RI(|\det|^{s_0+s^+}\sigma_{X^+},\pi_V,\pi_W),1_{\RH^+})
\]
to the complex plane; again, by taking the principal term at $s^+=-s_0$, we get a nonzero element in
\[
\Hom_{\RH^+}(\RI(\sigma_{X^+},\pi_V,\pi_W),1_{\RH^+}) 
\]

Therefore, $\Hom_{\RH^+}(\RI(\sigma_{X^+},\pi_V,\pi_W),1_{\RH^+})\neq 0$. This completes the proof.
\end{proof}

\section{Proof for Theorem \ref{thm: conj 3}}\label{section: 5.3}
In this section, we prove Theorem \ref{thm: conj 3} that Conjecture \ref{conjecture: main} holds over Archimedean local fields following the ideas in \cite{moeglin2012conjecture}. 

The proof is based on Theorem \ref{thm: main intro}.  Theorem \ref{thm: main intro}(1)  reduces the theorem to the basic cases. The counterpart of the reduction over non-Archimedean fields are in \cite[\S 19]{gan2012symplectic}\cite[\S 1.5]{moeglin2012conjecture}.  Then, in the basic cases, we prove Theorem \ref{thm: main intro}(2)(3), which are the counterparts of the basic forms of "the first inequality" in \cite[\S 1.4]{moeglin2012conjecture} and "the second inequality" in \cite[\S 1.8]{moeglin2012conjecture}, respectively. 

More precisely, in Section \ref{section: main 2}, we complete the proof for Theorem \ref{thm: main intro}(2) using refined distribution analysis as in \cite[\S 1.4]{moeglin2012conjecture}. Then we use an analog of this proof in Schwartz analysis to prove Theorem \ref{thm: main intro}(1) in Section \ref{section: proof for (1)}. 

In Section \ref{section: first inequality}, we apply Theorem \ref{thm: main intro}(1) and a similar mathematical induction as in \cite[\S 1.6]{moeglin2012conjecture} to prove the first equality  from its basic forms (Theorem \ref{thm: main intro}(2)). 
Similarly, the second inequality also follows  from its basic forms (Theorem \ref{thm: main intro}(3)) and mathematical induction. Since Theorem \ref{thm: main intro}(3) works in the more general context, there are at most two steps in the mathematical induction, detailed in Section \ref{section: second inequality}.

\subsection{Proof for Theorem \ref{thm: main intro}(2)}\label{section: main 2}

In this section, we prove Theorem \ref{thm: main intro}(2) using a refined distributional analysis. We first define $\mathrm{LI}(\pi_V)$ and $\RI'(\chi,\pi_V,\pi_W)$ in Theorem \ref{thm: main intro}(2).
\begin{defin}
For a  representation
\[
\pi_V=|\cdot|^{s_1}\sigma_{V,1}\times \cdots \times |\cdot|^{s_r}\sigma_{V,r}\rtimes \pi_{V_0},
\]
as in Theorem \ref{thm: the first inequality}, we define $\mathrm{LI}(\pi_V)$ to be the supermum  of
\begin{enumerate}
    \item $\Re(s_i)$, when $\sigma_{V,i}$ is a unitary character of $\GL_1(E)$;
    \item $\Re(s_i)+\frac{m_i}{2}$, when $\sigma_{V,i}$ is the discrete series $D_{m_i}$ of $\GL_2(\BR)$.
\end{enumerate}
\end{defin}

\begin{defin}\label{def: Iprime}
When $\dim X^+=2$, let $\chi$ be a character of $\BR^{\times}$, then we define 
\[\RI'(\chi,\pi_V,\pi_W)=\begin{cases}
    (\chi\rtimes \pi_W)\boxtimes \pi_V &\text{ in (B) cases},\\
    \pi_W\boxtimes(\chi\rtimes \pi_V) &\text{ in (FJ 1) cases},\\
    (\chi\rtimes \pi_V)\boxtimes \pi_W&\text{ in (FJ 2) cases}.
\end{cases}
\]
\end{defin}

Since $\RG^+(F)$ is unimodular, by definition,
\[\RI(|\det|^s\sigma_{X^+},\pi_V, \pi_W)=\Ind_{\RP^+}^{\CS,\RG^+}(\delta_{\RP^+}^{1/2}\otimes l(|\det|^s\sigma_{X^+},\pi_V,\pi_W)).\]

From Proposition \ref{pro: alt Schwartz induction}, 
\[\Ind_{\RP^+}^{\CS,\RG^+}(\delta_{\RP^+}^{1/2}\otimes l(|\det|^s\sigma_{X^+},\pi_V,\pi_W))=\Gamma^{\CS}(\RP^+(F)\bs \RG^+(F),\CE_s),
\]
where $\CE_s$ the bundle
\[
\RP^+(F)\bs (\RG^+(F)\times \delta_{\RP^+}^{1/2}\otimes l(|\det|^s\sigma_{X^+},\pi_V,\pi_W))
\]
with the $\RP^+(F)$-action  given in Proposition \ref{pro: geometric definition}.

Let $\CX=\RP^+(F)\bs \RG^+(F)$ and $\CU$ be the open $\RH^+$-orbit $\RP^+(F)\bs\RP^+(F)\RH^+(F)$  (Lemma \ref{lem: decomposition of the stabilizer group}) and $\CZ$ be the complement of $\CU$ in $\CX$. 

The extension by zero gives an embedding of $\Gamma^{S}(\mathcal{U},\mathcal{E}_s)
$ into $\Gamma^{S}(\mathcal{X},\mathcal{E}_s)$, and we define
\[
\Gamma^{\CS}_{\CZ}(\CX,\CE_s)=\Gamma^{S}(\mathcal{X},\mathcal{E}_s)/\Gamma^{S}(\mathcal{U},\mathcal{E}_s).
\]

To prove Theorem \ref{thm: main intro}(2), it suffices to show that 
\begin{lem}\label{lem: inequalities}
Under the  conditions on $\Re(s)$ in Theorem \ref{thm: main intro}(2), we have
\begin{enumerate}
    \item 
$\dim\Hom_{\RH^+}(\Gamma^{S}(\mathcal{U},\mathcal{E}_s),1_{\RH^+})\geqslant \dim\Hom_{\RH^+}(\Gamma^{S}(\mathcal{X},\mathcal{E}_s),1_{\RH^+})$;
\item
\begin{enumerate}
    \item $m(\pi_V\boxtimes \pi_W)\geqslant\dim\Hom_{\RH^+}(\Gamma^{S}(\mathcal{U},\mathcal{E}_s),1_{\RH^+})$ when $\sigma_{X^+}=|\cdot|^s\chi$;
    \item $m(\RI'(|\cdot|^{s+\frac{m}{2}}\sgn^{m+1},\pi_V,\pi_W))\geqslant\dim\Hom_{\RH^+}(\Gamma^{S}(\mathcal{U},\mathcal{E}_s),1_{\RH^+})$ when $\sigma_{X^+}=|\det|^sD_m$.
\end{enumerate}
\end{enumerate}    
\end{lem}
The proof for Lemma \ref{lem: inequalities} is based on the following properties of the $\Hom$-functor in the category  $\Pi_{\FMG}$.

\begin{lem}\label{lem: three properties} 
\begin{enumerate}
    \item (Left exactness)
    For an exact sequence $0\to \pi_1 \to \pi_2 \to \pi_3 \to 0$ of $\RH^+(F)$-representations, there is an exact sequence
    \begin{equation*}\label{equ: left exact sequence}
    0\to \Hom_{\RH^+}(\pi_3,1_{\RH^+})\to \Hom_{\RH^+}(\pi_2,1_{\RH^+})\to \Hom_{\RH^+}(\pi_1,1_{\RH^+});
    \end{equation*}
    \item (Link between direct limit and inverse limit)
    Given a projective system $\{\pi_{\alpha}\}_{I}$, we have
    \begin{equation*}\label{equ: direct limit and inverse limit}
    \Hom_{\RH^+}(\varprojlim_{\alpha} \pi_{\alpha},1_{\RH^+})=\varinjlim_{\alpha} \Hom_{\RH^+}(\pi_{\alpha},1_{\RH^+});
    \end{equation*}
    \item (Vanishing results for reductive groups)
    For tempered representations $\sigma_i$ ($1\leqslant i\leqslant l$), tempered representation $\pi_{V_0'}$ of $\RG_{V_0'}$, and $\Re(s_1)\geqslant \cdots \geqslant \Re(s_l)>0$, we let
    \[
\pi_V=|\det|^{s_1}\sigma_1\times\cdots\times|\det|^{s_r}\sigma_r\rtimes\pi_{V_0},
\]
then we have
\[
    \Hom_{\RG_V}((|\det|^s\sigma_V\rtimes\pi_{V_0'})\wh{\otimes} \pi_V,1_{\RG_V}) = 0
\]
for $\pi_{V_0'}\in \Pi_{\FMG}(\RG_{V_0})$ when $\Re(s)>\Re(s_1)$;
    \item (Vanishing results for Jacobi groups)
    When $\RG_V=\Sp(V),\wt{\Sp}(V),\RU(V)$, let $\pi_V^J=\wt{\pi}_V\otimes \omega_{V,\psi_F}\in \Pi_{\FMG}^{\psi_F}(\RG_{V}^J)$, where 
    \[
\wt{\pi}_V=|\det|^{s_1}\sigma_1\times\cdots\times|\det|^{s_r}\sigma_r\rtimes\wt{\pi}_{V_0},
\]
where $s_i,\sigma_i$ are as in (3) and $\wt{\pi}_{V_0}$ is an irreducible tempered representaiton of $\wt{\RG}_{V_0}$, then we have
\[
    \Hom_{\RG_V^J}((|\det|^s\sigma_V\rtimes\pi_{V_0'}^J)\wh{\otimes} \pi_V^J,1_{\RG_V^J}) = 0
\]
for $\pi_{V_0'}^J\in \Pi_{\FMG}^{\psi_F^{-1}}(\RG_{V_0}^J)$ when $\Re(s)>\Re(s_1)$.
\end{enumerate}
\end{lem}

\begin{proof}
Points (1)(2) are well-known. From Lemma \ref{lem: reciprocity}, 
\[  
\begin{aligned}
&\Hom_{\RG_V}((|\det|^s\sigma_V\rtimes\pi_{V_0'})\wh{\otimes} \pi_V,1_{\RG_V})\\
=&\Hom_{\RG_V}((|\det|^s\sigma_V\rtimes\pi_{V_0'}),\pi_V^{\vee})\\
=&\Hom_{\RG_V}(|\det|^s\sigma_V\rtimes\pi_{V_0'},(|\det|^{s_1}\sigma_1)^{\tau}\times \cdots\times (|\det|^{s_r}\sigma_{l})^{\tau}\rtimes \pi_{V_0})
\end{aligned}
\]
where $\tau$ is the complex conjugation when $E\neq F$ and $\tau$ is the trivial action when $E=F$. From Theorem \ref{thm: vanishing in appendix}(2), this space is equal to zero, then we proved Point (3). From the classification in Lemma \ref{lem: classification of representations 2}, $\pi_{V_0'}^J=\wt{\pi}_{V_0'}\otimes\omega_{V_0,\psi_F}$ and from Lemma \ref{lem: compatibility with tensor Weil}, 
\[
|\det|^s\sigma_V\rtimes\pi_{V_0'}^J=(|\det|^s\sigma_V\rtimes \wt{\pi}_{V_0'})\otimes\omega_{V,\psi_F}.
\] Then from Lemma \ref{lem: coinvariant tensor weil}
\[
    \Hom_{\RG_V^J}((|\det|^s\sigma_V\rtimes\wt{\pi}_{V_0'})\wh{\otimes} \pi_V\wh{\otimes}(\omega_{V,\psi_F}\otimes \omega_{V,\psi_F^{-1}}),1_{\RG_V^J}) = 
    \Hom_{\RG_V}((|\det|^s\sigma_V\rtimes\pi_{V_0'})\wh{\otimes} \pi_V,1_{\RG_V}).
\]
This reduces Point (4) to Point (3).
\end{proof}

\begin{proof}[Proof for Lemma \ref{lem:  inequalities}(1)]
 We prove Point (1) by
analyzing the complement $\CZ=\CX-\CU$ of the open orbit. From the left exactness (Lemma \ref{lem: three properties}), there is an exact sequence
\[
0\to \Hom_{\RH^+}(\Gamma^{\CS}_{\CZ}(\CX,\CE_s),1_{\RH^+}) \to \Hom_{\RH^+}(\Gamma^{S}(\mathcal{X},\mathcal{E}_s),1_{\RH^+}) \to \Hom_{\RH^+}(\Gamma^{S}(\mathcal{U},\mathcal{E}_s),1_{\RH^+}). 
\]

Hence, to prove the inequality in (1), it suffices to prove that
\begin{equation}\label{equ: hom closed=0}
\Hom_{\RH^+}(\Gamma^{\CS}_{\CZ}(\CX,\CE_s),1_{\RH^+})=0
\end{equation}
under the given conditions. 

One can study the structure of $\Gamma^{\CS}_{\CZ}(\CX,\CE_s)$ with the Borel's lemma in \cite{chen2020schwartz}.
\begin{lem}\label{lem: borel}
    There is a complete descending filtration $\{\Gamma_{\CZ}^{\CS}(\CX,\CE_s)_k\}_{k\geqslant 0}$ of  $\Gamma^\mathcal{S}_{\mathcal{Z}}\left(\mathcal{X}, \mathcal{E}_s\right)$ 
such that the graded pieces
\[ \Gamma_{\CZ}^{\CS}(\CX,\CE_s)_k/\Gamma_{\CZ}^{\CS}(\CX,\CE_s)_{k+1}=\Gamma^\mathcal{S}\left(\mathcal{Z}, \Sym^k\mathcal{N}_{\mathcal{Z}|\mathcal{X}}^\vee\otimes\mathcal{E}_s\big|_\mathcal{Z}\right), \quad k\in\mathbb{N}.
\]
where $\CN_{\CZ|\CX}$ is the conormal bundle of $\CZ$ in $\CX$.
\end{lem}
Recall that in Lemma \ref{lem: stabilizer group}, we computed
\[
\RS_{\gamma'}=\begin{cases}
  \Delta(\Res_{E/F}\GL(X_c)\times \RG_{V_0})\rtimes \RN_V  & \text{ in (B)(FJ 1) cases},\\
  \Delta(\Res_{E/F}\GL(X_c)\times \RG_{V_0^+}^J)\rtimes \RN_{W}^J  & \text{in (FJ 2) cases},\\
\end{cases}
\]
where $V_0,W_0^+,V_0^+$ are given as in Lemma \ref{lem: complement}, and we may assume  
    (1) $V_0\subset W$ and $\dim_E W=\dim_E V_0+1$ in (B) cases;
    (2) $W=V_0$ in (FJ 1) cases;
    (3) $V_0^+\subset V$ and $\dim_E V=\dim_E V_0^++2$ in (FJ 2) cases.

From Lemma \ref{lem: conormal bundle}, we compute the fiber of the bundle at $\gamma'$ as a representation of $\RS_{\gamma'}$
\[
\mathrm{Fib}_{\gamma'}(\Sym^k\CN_{\CZ|\CX}^{\vee}\otimes \CE_s|_{\CZ})=\begin{cases}
\delta_{\RP^+}^{1/2}\otimes((|\det|^{s+1/2}\sigma_{X^+}\otimes \Sym^k\rho)\boxtimes \pi_W^{\gamma'}|_{\RG_{V_0}})\boxtimes \pi_V& \text{ in (B) cases},\\
     \pi_V\boxtimes \delta_{\RP^+}^{1/2}\otimes((|\det|^{s+1/2}\sigma_{X^+}\otimes \Sym^k\rho)\boxtimes \pi_{W}^{\gamma'}|_{\RG_{W^+_0}})&\text{ in (FJ 1) cases},\\
     \delta_{\RP^+}^{1/2}\otimes((|\det|^{s+1/2}\sigma_{X^+}\otimes \Sym^k\rho)\boxtimes \pi_{V}^{\gamma'}|_{\RG_{V_0^+}^J})\boxtimes\pi_W&\text{ in (FJ 2) cases},
\end{cases}
\]
where $\rho=
\mathrm{Fib}_{\gamma'}(\CN_{\CZ|\CX}^{\vee})$ is the fiber of the conormal bundle $\CN_{\CZ|\CX}^{\vee}$ at $\gamma'$ given in Lemma \ref{lem: conormal bundle}. The representations obtained from restriction are Fr\'echet of moderate growth but not necessarily Casselman-Wallach.

We first detailed the cases when $\CZ$ has a single orbit $[\gamma']$ in Lemma \ref{lem: complement}. 
By comparing the unipotent part, we obtain that
\[
\delta_{\RP_{W^+,X^+}}(a)\delta_{\RP_{V,X_c}}^{-1}(a^{\gamma'})=|\det(a)|, 
\]
 \[
 \delta_{\RP_{W^+,X^+}^J}(a)\delta_{\RP_{V,X_c}}^{-1}(a^{\gamma'})=|\det(a)|,
 \] 
 \[
 \delta_{\RP_{V^+,X^+}}(a)\delta^{-1}_{\RP_{W,X_c}^J}(a^{\gamma'})=|\det(a)|,
 \]
 for $a\in \GL(X^+)(E)$. Then we have
\[
\Gamma_{\CZ}^{\CS}(\CX,\CE_s)_k/\Gamma_{\CZ}^{\CS}(\CX,\CE_s)_{k+1}=
\begin{cases}
    (|\det|^{s+1/2}\sigma_{X^+}\otimes \Sym^k\rho)\rtimes \pi_W^{\gamma'}|_{\RG_{V_0}}\boxtimes \pi_V& \text{ in (B) cases},\\
     \pi_V\boxtimes (|\det|^{s+1/2}\sigma_{X^+}\otimes \Sym^k\rho)\rtimes \pi_{W}^{\gamma'}|_{\RG_{W_0^+}}&\text{ in (FJ 1) cases},\\
(|\det|^{s+1/2}\sigma_{X^+}\otimes \Sym^k\rho)\rtimes \pi_{V}^{\gamma'}|_{\RG_{V_0^+}^J}\boxtimes \pi_W     &\text{ in (FJ 2) cases}.\\
\end{cases}
\]

Therefore, from Lemma \ref{lem: three properties}(3)(4) and the assumptions on $\Re(s)$, we have
\[
\Hom_{\RH^+}(\Gamma^{\CS}_{\CZ}(\CX,\CE_s)_k/\Gamma_{\CZ}^{\CS}(\CX,\CE_s)_{k+1},1_{\RH^+})=0, \quad k=0,1,\cdots
\]
From Lemma \ref{lem: three properties}(1)(2), we have (\ref{equ: hom closed=0}), so we completed the proof for Point (1).

When $\CZ$ has two orbits, from Lemma \ref{lem: complement}, it is either the special orthogonal Bessel cases or almost equal-rank Fourier-Jacobi cases. In (B) cases when $\RG_V=\SO(V)$, it was computed in \cite{chen2021local} that 
\[
\Gamma_{\CZ}^{\CS}(\CX,\CE_s)_k= ((|\det|^{s+1/2}\sigma_V\otimes \Sym^k\rho)\rtimes \pi_W|_{\RG_{V_0}}\boxtimes \pi_V)^{\oplus 2}.
\]

In (FJ 2) cases when $\dim X^+=1$, notice $\CZ$ has an open orbit $[\gamma']$ and a closed orbit $[\gamma]$, which is a single point.
\[
0\to \Gamma^{\CS}([\gamma'],\Sym^k\mathcal{N}_{\mathcal{Z}|\mathcal{X}}^\vee\otimes\mathcal{E}_s\big|_\mathcal{Z})\to \Gamma_{\CZ}^{\CS}(\CX,\CE_s)_k\to \Gamma_{[\gamma^{-1}]}^{\CS}(\CZ,\CE_s)\to 0.
\]

Using the method for single orbit situation, one can show that 
\[
\Hom_{\RH^+}(\Gamma^{\CS}([\gamma'],\Sym^k\mathcal{N}_{\mathcal{Z}|\mathcal{X}}^\vee\otimes\mathcal{E}_s\big|_\mathcal{Z}),1_{\RH^+})=0,\quad \Hom_{\RH^+}(\Gamma_{[\gamma^{-1}]}^{\CS}(\CZ,\CE_s),1_{\RH^+})=0.
\]
This special case is detailed in \cite{chen2022FJ}.
\end{proof}

To prove Point (2), we introduce the following property of discrete series  representations.

\begin{lem}(\cite[Lemma 5.4.9]{chen2021local})\label{lem: filtered structure of discrete series representation}
$|\det|^sD_m|_{\RR_{1,1}}$ has a subrepresentation (the underlying space is not necessarily closed) isomorphic to \[
\Ind_{\BR^{\times}\times 1}^{\CS,\RR_{1,1}}(|\cdot|^{s+\frac{m+1}{2}}\sgn^{m+1}).
\]
Moreover, there is a complete descending filtration with graded pieces isomorphic to
\[
|\det|^{k+s+\frac{m}{2}}\sgn(\det)^k|_{\RR_{1,1}},\text{ for } k=1, 2,\cdots.
\]
\end{lem}
Then from the exactness of Schwartz induction (\cite[Proposition 7.1]{chen2020schwartz}) and exactness of tensor product with nuclear Fr\'echet space (\cite[Lemma A.3]{casselman2000bruhat}), the embedding
\[
\Ind_{\BR^{\times}\times 1}^{\CS,\RR_{1,1}}(|\cdot|^{s+\frac{m}{2}}\sgn^{m+1}) \hookrightarrow D_m|_{\RR_{1,1}}
\]
induces an embedding 
\[
\Ind^{\CS,\RH^+}_{\RP^+\cap \RH^+}(\delta_{\RP^+}^{1/2}\otimes \Ind_{\BR^{\times}\times 1}^{\RR_{1,1}}(\chi)\boxtimes \pi_V\boxtimes \pi_W)\hookrightarrow 
\Ind^{\CS,\RH^+}_{\RP^+\cap \RH^+}(\delta_{\RP^+}^{1/2}\otimes |\det|^sD_m\boxtimes \pi_V\boxtimes \pi_W)=\Gamma^{S}(\mathcal{U},\mathcal{E}_s).
\]
\begin{num}
 \item \label{equ: complete de open}   
 The quotient\[\Gamma^{S}(\mathcal{U},\mathcal{E}_s)/
\Ind^{\CS,\RH^+}_{\RP^+\cap \RH^+}(\delta_{\RP^+}^{1/2}\otimes \Ind_{\BR^{\times}\times 1}^{\RR_{1,1}}(\chi)\boxtimes \pi_V\boxtimes \pi_W)\] has a
complete descending filtration with graded pieces isomorphic to
\[
\Ind^{\CS,\RH^+}_{\RP^+\cap \RH^+}(\delta_{\RP^+}^{1/2}\otimes |\det|^{k+s+\frac{m}{2}}\sgn(\det)^k|_{\RR_{1,1}}
\boxtimes \pi_V\boxtimes \pi_W).
\]
\end{num}
\begin{lem}
For a character $\chi$ of $\BR^{\times}$, we have
\[
\RI'(\chi,\pi_V,\pi_W)=\Ind^{\CS,\RH^+}_{\RP^+\cap \RH^+}(\delta_{\RP^+}^{1/2}\otimes \Ind_{\BR^{\times}\times 1}^{\RR_{1,1}}(|\cdot|^{-1/2}\chi))\boxtimes \pi_V\boxtimes \pi_W).
\]
\end{lem}
\begin{proof}
Since both $\RG^+(F)$, $\RH^+(F)$ are unimodular, $\RH^+(F)\bs\RG^+(F)$ has a $\RG^+(F)$-invariant measure. Since $\RH^+(F)\RP^+(F)$ is open in $\RG(F)$  (Lemma \ref{lem: stabilizer group}), this measure induces an $\RP^+(F)$-invariant measure on $\RP^+\cap \RH^+(F)\bs \RP^+(F)$. Hence, \[
\delta_{\RP^+}|_{\RP^+\cap \RH^+}=\delta_{\RP^+\cap \RH^+}.
\]
Then the equality follows from the fact that
\[
\delta_{\RR_{1,1}}(\diag(a,1))=|a|.
\]
\end{proof}

\begin{defin}
For a character $\chi$ of $\BR^{\times}$, we have
\[
\RI''(\chi,\pi_V,\pi_W)=
\Ind^{\CS,\RH^+}_{\RP^+\cap \RH^+}(\delta_{\RP^+}^{1/2}\chi(\det)\boxtimes \pi_V\boxtimes \pi_W)
\]
\end{defin}
\begin{lem}\label{lem: Iprimeprime}
  \[
\RI''(\chi,\pi_V,\pi_W)=\begin{cases}
(\chi\rtimes \Ind_{\RG_W}^{\CS,\RG_{W\oplus Ez_0}}(\pi_W))\boxtimes \pi_V
&
\text{ in (B) cases},    \\
   \pi_V\boxtimes(\chi\rtimes \Ind_{\RG_W^J}^{\CS,\RG_{W\oplus H_1}}(\pi_W))&
\text{ in (FJ 1) cases},
 \\
    \chi\rtimes \Ind_{\RG_V}^{\RG_{V}^J}(\pi_V))\boxtimes \pi_W & \text{ in (FJ 2) cases}.
\end{cases}
\]  
\end{lem}
\begin{proof}
This follows from the fact that $\RG_{W\oplus Ez_0}$, $\RG_{W}$ (in (B) cases), $\RG_{W\oplus H_1}$, $\RG_W^J$ ( in (FJ 1) cases), $\RG_{V}^J$, $\RG_V$ (in (FJ 2) cases) are unimodular. 
\end{proof}
\begin{proof}[Proof for Lemma \ref{lem: inequalities}(2)]
We prove Point (2) by analyzing the open orbit $\CU$.

When $\sigma_{X^+}=|\cdot|^s\chi$, from Corollary \ref{cor: mirabolic structure}, we have $\CU=\RH(F)\bs\RP^+\cap \RH^+(F)=\{1\}$, so
\[
\dim\Hom_{\RH^+}(\Gamma^{S}(\mathcal{U},\mathcal{E}_s),1_{\RH^+})=m(\pi_V\boxtimes \pi_W)
\]

When $\sigma_{X^+}=|\det|^sD_m$,   
 from Corollary \ref{cor: mirabolic structure},
$\CU=\RH(\BR)\bs\RP^+\cap \RH^+(\BR)=\{\RN_{X^+}(\BR)\bs \RR_{X',X^+}(\BR)\}=\BR^{\times}$. From Lemma \ref{lem: Iprimeprime} and  Lemma \ref{lem: inequalities}(3)(4),
\[
\Hom_{\RH^+}(\RI''(|\det|^{s+k+\frac{m}{2}}\sgn(\det)^k,\pi_V,\pi_W)),\BC)=0, \quad \text{ for }k=1,2,\cdots
\]
then from (\ref{equ: complete de open}) Lemma \ref{lem: three properties}(1)(2), we have
\[
\dim\Hom_{\RH^+}(\Gamma^{S}(\mathcal{U},\mathcal{E}),1_{H^+})\leqslant m(\pi_V\boxtimes \RI'(|\cdot|^{s+{\frac{m}{2}}}\sgn^{m+1},\pi_V,\pi_W)).
\]
\end{proof}

\subsection{Proof for Theorem \ref{thm: main intro}(1)} \label{section: proof for (1)}

Theorem \ref{thm: main intro}(1) was proved  in   \cite[Proposition 6.1]{xue2020bessel} for unitary Bessel cases, and  in \cite[Lemma 5.1.3]{chen2021local} for special orthogonal Bessel cases using a Schwartz homology analog of the proof for Theorem \ref{thm: main 2}(2).

\begin{defin}
For $V\in \Pi_{\FMG}(G)$, the Schwartz homology $\RH_i^{\mathcal{S}}(G,V)$ is defined to be the $i$-th left derived functors of the  $G(\BR)$-coinvariant functor $V\mapsto V_{G}$, where $V_{G}$ is defined in Definition \ref{def: Hausdorff coinvariant}(1). In particular, $\RH_0^{\mathcal{S}}(G,V)=V_{G}$.
\end{defin}
As in the previous section, we have
\[
\RI(\sigma_{\udl{s}},\pi_V, \pi_W)=\Gamma^{\CS}(\RP^+(F)\bs \RG^+(F),\CE_{\udl{s}}),
\]
where $\CE_{\udl{s}}$ the bundle
\[
\RP^+(F)\bs (\RG^+(F)\times \delta_{\RP^+}^{1/2}\otimes l(\sigma_{\udl{s}},\pi_V,\pi_W))
\]
with the $\RP^+(F)$-action  given in Proposition \ref{pro: geometric definition}.

The following lemma is the counterpart of Lemma \ref{lem: inequalities}
\begin{lem} \label{lem: equalities}
We have
\begin{enumerate}
    \item 
$H_i(\RH^+(F),\Gamma^{S}(\mathcal{U},\mathcal{E}_{\udl{s}}))=H_i(\RH^+(F),\Gamma^{S}(\mathcal{X},\mathcal{E}_{\udl{s}}))$ for $\udl{s}$ in general positions;
\item
$H_i(\RH^+(F),\Gamma^{\CS}(\CU,\CE_s))=H_i(\RH(F),\pi_V\boxtimes \pi_W\otimes \xi^{-1})$ for $\udl{s}$ in general positions.
\end{enumerate}       
\end{lem}

The idea for Lemma \ref{lem: equalities} proof is to replace Lemma \ref{lem: three properties} in the proof for Theorem \ref{thm: main intro}(2) with  properties of the Schwartz homology in Lemma \ref{lem: two property Schwartz}.

\begin{lem}\label{lem: two property Schwartz}
\begin{enumerate}
\item (Long exact sequence)
For an exact sequence 
\[
0\to \pi_1\to \pi_2\to \pi_3\to 0
\]
of $\RH^+(F)$-representations, there is a long exact sequence
\[
\to H_i^{\CS}(\RH^+(F),\pi_1)\to H_i^{\CS}(\RH^+(F),\pi_2)\to H_i^{\CS}(\RH^+(F),\pi_3)\to H_{i-1}^{\CS}(\RH^+(F),\pi_1) \to \cdots
\]
\item (Commutes with inverse limit)
Given a projective system $\{\pi_{\alpha}\}_{I}$ with surjective $\pi_{\alpha}\to \pi_{\beta}$ for $\alpha>\beta$, we have
\[
H_i^{\CS}(\RH^+(F),\varprojlim_{\alpha} \pi_{\alpha})=\varprojlim_{\alpha} H_i^{\CS}(\RH^+(F),\pi_{\alpha}).
\]
\item (Vanishing result for reductive groups) In the setting of Lemma \ref{lem: three properties}(3), we have
\[
    H_i^{\CS}(\Delta\RG_V(F),(|\det|^s\sigma_V\rtimes\pi_{V_0})\boxtimes \pi_V) = 0
\]
for $s\in\BC$ in general positions.
\item (Vanishing results for Jacobi groups)
 In the settings of Lemma \ref{lem: three properties}(4), we have
\[
    H_i^{\CS}(\Delta\RG_V^J(F),((|\det|^s\sigma_V\rtimes\pi_{V_0})\wh{\otimes} \pi_V,1_{\RG_V^J}) = 0
\]
for $s\in \BC$ in general positions.
\end{enumerate}
\end{lem}
\begin{proof}
Point (1) is a property of derived functors. Point (2) was proved in \cite[Proposition 2.13]{xue2020bessel}. Point (3) follows from the vanishing criterion on infinitesimal characters (\cite{xue2020bessel}) using the arguments in \cite[\S 3]{xue2020bessel} or \cite[Lemma 5.2.7]{chen2021local}. Using the arguments in the proof for \ref{lem: three properties}, Point (4) follows from Point (3). 
\end{proof}

\begin{proof}[Proof for Lemma \ref{lem: equalities}(1)]
Following the computation  complete descending filtration $\Gamma_{\CZ}^{\CS}(\CX,\CE_{\udl{s}})_k$ of $\Gamma_{\CZ}^{\CS}(\CX,\CE_{\udl{s}})=\Gamma^{\CS}(\CX,\CE_{\udl{s}})/\Gamma^{\CS}(\CU,\CE_{\udl{s}})$   
with 
\[
\Gamma_{\CZ}^{\CS}(\CX,\CE_{\udl{s}})_k/
\Gamma_{\CZ}^{\CS}(\CX,\CE_{\udl{s}})_{k+1}=
\begin{cases}    (|\det|^{s+1/2}\sigma_{\udl{s}}\otimes \Sym^k\rho)\rtimes \pi_W^{\gamma'}|_{\RG_{V_0}}\boxtimes \pi_V& \text{ in (B) cases},\\
     \pi_V\boxtimes (|\det|^{1/2}\sigma_{\udl{s}}\otimes \Sym^k\rho)\rtimes \pi_{W}^{\gamma'}|_{\RG_{W_0^+}}&\text{ in (FJ 1) cases},\\
(|\det|^{1/2}\sigma_{\udl{s}}\otimes \Sym^k\rho)\rtimes \pi_{V}^{\gamma'}|_{\RG_{V_0^+}^J}\boxtimes \pi_W     &\text{ in (FJ 2) cases}.
\end{cases}
\]
Then from Lemma \ref{lem: two property Schwartz}(3)(4),
\[
H_i(\RH^+(F),\Gamma_{\CZ}^{\CS}(\CX,\CE_{\udl{s}})_k/
\Gamma_{\CZ}^{\CS}(\CX,\CE_{\udl{s}})_{k+1})=0,\quad i=0,1,\cdots,\quad k=0,1,\cdots
\]
Hence, from Lemma \ref{lem: two property Schwartz}(1), we have
\[
H_i(\RH^+(F),\Gamma_{\CZ}^{\CS}(\CX,\CE_{\udl{s}})/\Gamma_{\CZ}^{\CS}(\CX,\CE_{\udl{s}})_k)=0,\quad i=0,1,\cdots,\quad k=0,1,\cdots
\]
Therefore, from Lemma \ref{lem: two property Schwartz}(2), we have
\[
H_i(\RH^+(F),\Gamma_{\CZ}^{\CS}(\CX,\CE_{\udl{s}}))=0,\quad i=0,1,\cdots
\]
Then from Lemma \ref{lem: two property Schwartz}(1), we obtain Lemma \ref{lem: equalities}(1).
\end{proof}

The idea for the proof for Lemma \ref{lem: equalities}(2) is also parallel to that for Lemma \ref{lem: inequalities}(2). We need to study the structure of $\sigma_{\udl{s}}/\Ind_{\RN_{0, X^+}}^{\CS,\RR_{X',X^+}}(\psi_{0,X^+}^{-1})$ with Lemma \ref{lem: XC}.

We first introduce some notations. 
\begin{itemize}
    \item We denote by $\RP_{0,n}$ a Borel subgroup of $\GL_{0,n}$ and denote by $\RP_{0,n}$ the unipotent radical of $\RP_{0,n}$;
\item  We denote by $\RP_{a,b,c}$ the parabolic subgroup of $\GL_{a+b+c}$ stabilizing a filtration
\[
X_a\subset X_{a+b}\subset X_{a+b+c},
\]
where $\dim X_i=i$, and we denote by $\RL_{a,b,c}$ its Levi subgroup.
\end{itemize}

\begin{lem}\label{lem: XC}(\cite[Proposition 5.1]{xue2020bessel}\cite[Proposition 5.3.4]{chen2021local})
There is a  $\RR_{X',X^+}$-equivariant embedding from $\Ind_{\RN_{0, X^+}}^{\CS,\RR_{X',X^+}}(\psi_{0,X^+}^{-1})$
to $\sigma_{\udl{s}}$. The quotient $\sigma_{\udl{s}}/\Ind_{\RN_{0, X^+}}^{\CS,\RR_{X',X^+}}(\psi_{0,X^+}^{-1})$  admits an $R_{X',X^+}$-stable complete filtration whose graded pieces have the shape 
\[
\Ind_{\RP_{a,b,c}}^{\CS,\RR_{\dim X',1}}(\tau_a\boxtimes \tau_b\boxtimes \tau_c),
\]
where $a+b+c=\dim X^+$, $a+b\neq 0$ and 
 the $\RP_{a,b,c}$-representation $\tau_a\boxtimes \tau_b\boxtimes \tau_c$ is regarded as the inflation  from $\RL_{a,b,c}$-representation $\tau_a\boxtimes \tau_b\boxtimes \tau_c$.
\begin{enumerate}
    \item $\tau_a=\Ind_{\RP_{0,a}}^{\CS,\GL_a}(\sgn^{m_1}|\cdot|^{s_{i_1}+k_1}\boxtimes \cdots \boxtimes \sgn^{m_{a}}|\cdot|^{s_{i_a}+k_a})$
    where $1\leqslant i_1, \cdots,i_a\leqslant t+1$ are integers, $l_1,\cdots,l_a\in\BZ$ and $k_1,\cdots,k_a\in\frac{1}{2}\BZ$;
    \item $\tau_b=\tau_b'\otimes \rho$ where $\tau_b'$ is a representation of the same form as $\tau_a$ and $\rho$ is a finite-dimensional representation of $\GL_b(\BR)$;
    \item $\tau_c=\Ind_{\RN_{0,c}}^{\RR_{c-1,1}}(\psi_{c}^{-1})$.
\end{enumerate}
\end{lem}
\begin{proof}
Using this lemma and the exactness of Schwartz induction and the exactness of tensor product with nuclear Fr\'echet space, we obtain a 
complete descending filtration  of 
\[
\Gamma_o=\Gamma^{\CS}(\CU,\CE_{\udl{s}})/\Ind_{\RP^+\cap \RH^+}^{\CS,\RH^+}(\Ind_{\RN_{0,X^+}}^{\CS,\RR_{X',X^+}}(\psi_{0,X^+}^{-1})\boxtimes \pi_V\boxtimes \pi_W)
\]
with graded pieces isomorphic to
\[
\begin{cases}
(\sgn^{m_1}|\cdot|^{s_{i_1}+k_1'}\rtimes \pi)\boxtimes \pi_V
&
\text{ in (B) cases},    \\
   \pi_V\boxtimes(\sgn^{m_1}|\cdot|^{s_{i_1}+k_1'}\rtimes \pi)&
\text{ in (FJ 1) cases},
 \\(\sgn^{m_1}|\cdot|^{s_{i_1}+k_1'}\rtimes \pi)\boxtimes \pi_W & \text{ in (FJ 2) cases}
\end{cases}
\]
for certain Fr\'echet representation $\pi$ of moderate growth and $k_1'\in \frac{1}{2}\BZ$, then from Lemma \ref{lem: two property Schwartz}(3)(4), we have
the Schwartz homology of the graded piece all equal to zero. Then Lemma \ref{lem: two property Schwartz}(1)(2) implies Lemma \ref{lem: equalities}.
\end{proof}

\begin{proof}[Proof for Theorem \ref{thm: main intro}(1)]
Then we prove Theorem \ref{thm: main intro}(1) from Lemma \ref{lem: equalities}. 

By definition
\[
m(\RI(\sigma_{\udl{s}},\pi_V,\pi_W))=\dim \Hom_{\RH^+}(\RI(\sigma_{\udl{s}},\pi_V,\pi_W),1_{\RH^+})=\dim \Hom_{\BC}(\RI(\sigma_{\udl{s}},\pi_V,\pi_W)^{\mathrm{Haus}}_{\RH^+},1),
\]
and
\[
m(\pi_V\boxtimes
\pi_W)=\dim \Hom_{\RH}(\pi_V\boxtimes\pi_W,\xi^{-1})=\dim \Hom_{\BC}((\pi_V\boxtimes \pi_W\otimes \xi)^{\mathrm{Haus}}_{\RH},1).
\]

Hence, it suffices to show that 
\begin{equation}\label{equ: maximal Hausdorff quotient}
\RI(\sigma_{\udl{s}},\pi_V,\pi_W)_{\RH^+}^{\mathrm{Haus}}=(\pi_V\boxtimes \pi_W\otimes \xi^{-1})_{\RH}^{\mathrm{Haus}}.
\end{equation}

By definition, 
\[
\RH_0^{\CS}(\RH^+(F),\RI(\sigma_{\udl{s}},\pi_V,\pi_W))=\RI(\sigma_{\udl{s}},\pi_V,\pi_W)_{\RH^+},\quad \RH_0^{\CS}(\RH(F),\pi_V\boxtimes \pi_W\otimes \xi^{-1})=(\pi_V\boxtimes \pi_W\otimes \xi^{-1})_{\RH}.
\]
Then from Lemma \ref{lem: equalities},
\[
\RI(\sigma_{\udl{s}},\pi_V,\pi_W)_{\RH^+}=(\pi_V\boxtimes \pi_W\otimes \xi^{-1})_{\RH}
\]
as topological spaces, then their maximal Hausdorff quotients are the same, so we have (\ref{equ: maximal Hausdorff quotient}). 
\end{proof}

\subsection{The first inequality}\label{section: first inequality}
In this section, we prove the first inequality $m(\pi_V\boxtimes \pi_W)\geqslant m(\RI(\sigma_{X^+},\pi_V,\pi_W))$ in Conjecture \ref{conjecture: main} from Theorem \ref{thm: main intro}(1)(2) using a modification of the mathematical induction in \cite{moeglin2012conjecture}. 

From the classification of tempered representations of general linear groups (\cite{knapp1982classification}), every tempered representation $\sigma$ of a general linear group is in the form of 
\[
|\det|^{s_1}\sigma_1\times\cdots\times |\det|^{s_l}\sigma_l
\]
where $s_i\in \BC$,  $\sigma_{i}$ is either equal to an  one-dimensional unitary $\GL_1(E)$-representation,  or a discrete series representation of $\GL_2(\BR)$ when $E=\BR$. Then we can prove the first inequality in the basic cases with the following theorem.

The proof follows from a modification of the proof in \cite[\S 1.4-\S 1.6]{moeglin2012conjecture}, which proves the basic form of the first inequality, and then applies 
mathematical induction to reduce the general inequalities to those in the basic forms.

\begin{thm}\label{thm: the first inequality}
Let 
\[
\sigma_V=|\cdot|^{s_{V,1}}\sigma_{V,1}\times \cdots \times |\cdot|^{s_{V,l_V}}\sigma_{V,l_V},\quad \sigma_W=|\cdot|^{s_{W,1}}\sigma_{W,1}\times \cdots \times |\cdot|^{s_{W,l_W}}\sigma_{W,l_W},
\]
where $\sigma_{V,i},\sigma_{W,i}$ are either equal to  unitary characters of  $\GL_1(E)$,  or discrete series representations of $\GL_2(\BR)$  when $E=\BR$, and $\Re(s_{V,i}),\Re(s_{W,i})\geqslant 0$, then for irreducible tempered representations $\pi_{V_0},\pi_{W_0}$ and $\pi_V=\sigma_V\rtimes \pi_{V_0},\pi_W=\sigma_W\rtimes \pi_{W_0}$
\[
m(\pi_V\boxtimes \pi_W)\leqslant m(\pi_{V_0}\boxtimes \pi_{W_0}).
\]    
\end{thm}

\begin{proof}[Proof for "the first inequality" (Theorem \ref{thm: the first inequality})]
Then we prove "the first inequality" follows from the basic forms. From Theorem \ref{thm: main intro}, we may assume $(V,W)$ is a basic admissible pair. The mathematical induction can be completed following a modification of the proof in \cite[\S 5.4]{chen2021local}. To put more emphasis on the induction, we abuse to notions and take $m(\pi_V\boxtimes \pi_W)=m(\pi_W\boxtimes \pi_V)$. The precise orders are given in a similar mathematical induction in next section.

We may assume
 \[
\Re(s_{V,1})+\frac{m_{V,1}}{2}\geqslant\cdots \geqslant \Re(s_{V,l_V})+\frac{m_{V,l_V}}{2}\geqslant 0, \text{and}
\]
\[
\Re(s_{W,1})+\frac{m_{V,l_V}}{2}\geqslant\cdots \geqslant \Re(s_{W,l_W})+\frac{m_{V,l_V}}{2}\geqslant 0.
\]

 We  prove the inequality in the theorem by mathematical induction on the lexicographical order of $(N(\sigma_V,\sigma_W),M(\sigma_V,\sigma_W))\in \BN^2$, where
\[
N(\sigma_V,\sigma_W)=\sum_{i=1}^{l_V}n_{V,i}+\sum_{i=1}^{l_W}n_{W,i},
\]
and $M(\sigma_V,\sigma_W)$ is equal to the largest $i$ such that
\[
\Re(s_{W,i})+\frac{m_{W,i}}{2}>\Re(s_{V,1})+\frac{m_{V,1}}{2}
\]
when $\mathrm{LC}(\pi_W)>\mathrm{LC}(\pi_V)$,  the largest $i$ such that 
\[
\Re(s_{V,i})+\frac{m_{V,i}}{2}>\Re(s_{W,1})+\frac{m_{W,1}}{2}
\]
when $\mathrm{LC}(\pi_V)>\mathrm{LC}(\pi_W)$, and $0$ when $\mathrm{LC}(\pi_V)=\mathrm{LC}(\pi_W)$.

We set 
\[
\sigma_V'=|\cdot|^{s_{V,2}}\sigma_{V,1}\times \cdots \times |\cdot|^{s_{V,l_V}}\sigma_{V,l_V},\quad \sigma_W'=|\cdot|^{s_2}\sigma_{W,1}\times \cdots \times |\cdot|^{s_{W,l_W}}\sigma_{W,l_W}
\]
\[
\pi_V'=\sigma_V'\rtimes \pi_{V_0},\quad \pi_{W}'=\sigma_{W}'\rtimes \pi_{W_0}
\]

If $N(\sigma_V,\sigma_W)=0$, $\pi_V=\pi_{V_0}$ and $\pi_{W}=\pi_{W_0}$, then
\[
m(\pi_{V}\boxtimes \pi_{W})=m(\pi_{V_0}\boxtimes \pi_{W_0}).
\]

When $N(\sigma_V,\sigma_W)>0$,
\begin{enumerate}
    \item If $\mathrm{LC}(\pi_V)>\mathrm{LC}(\pi_W)$, 
    \begin{enumerate}
        \item When $n_{V,1}=1$, from  Theorem \ref{thm: main intro}(2), we have 
        \[        m(\pi_V',\pi_W)\geqslant m(\pi_V\boxtimes \pi_W)
        \]
        and 
    \[
N(\sigma_V',\sigma_W)=N(\sigma_V,\sigma_W)-1.
    \]
    \item When $n_{V,1}=2$, from Theorem \ref{thm: main intro}(2), we have \[m((|\cdot|^{s_{V,1}+\frac{m_{V,1}}{2}}\sgn^{m_{V,1}+1} \rtimes\pi_V')\boxtimes\pi_W)\leqslant m(\pi_V\boxtimes \pi_W)\]
    and 
    \[
N(|\cdot|^{s_{V,1}+\frac{m_{V,1}}{2}}\sgn^{m_{V,1}+1} \rtimes\sigma_V',\sigma_W)=N(\sigma_V,\sigma_W)-1.
    \]
    \end{enumerate}
    \item If $\mathrm{LC}(\pi_V)<\mathrm{LC}(\pi_W)$, from Theorem \ref{thm: main intro}(1), we can fix $s'\in \BC$ with $\Re(s')=0$ such that 
    \[
m(\pi_V\boxtimes(|\cdot|^{s'}\rtimes\pi_W))=m(\pi_V,\pi_W).
    \]
    \begin{enumerate}
        \item If $M(\sigma_V,\sigma_W)>1$ and $n_{W,1}=1$, by applying Theorem \ref{thm: main intro}(2), we obtain 
        \[m(\pi_V,|\cdot|^{s'}\rtimes \pi_W')\leqslant m(\pi_V\boxtimes |\cdot|^{s'}\rtimes \pi_W')
        \]
        and we have
        \[
N(\sigma_V,|\cdot|^{s'}\rtimes \sigma_W')=N(\sigma_V,\sigma_W),\quad 
M(\sigma_V,|\cdot|^{s'}\rtimes \sigma_W')=M(\sigma_V,\sigma_W)-1.
        \]
        \item If $M(\sigma_V,\sigma_W)=1$ and $n_{W,1}=1$, we can apply Theorem \ref{thm: main intro}(2) twice and obtain that
        \[
    m(\pi_V\boxtimes (|\cdot|^{s_1}\rtimes\pi_W))\leqslant m(\pi_V'\boxtimes(|\cdot|^{s'}\rtimes\pi_W'))
        \]
        and we have
        \[
N(\sigma_V',|\cdot|^{s'}\rtimes \sigma_W')=N(\sigma_V,\sigma_W)-1.
        \]
    \item  When $n_{W,1}=2$, from Theorem \ref{thm: main intro}(2), we have
    \[
m(\pi_V\boxtimes(|\cdot|^{s'}\rtimes\pi_W))\leqslant
m(\pi_V\boxtimes(|\cdot|^{s'}\rtimes |\cdot|^{s_{W,1}+\frac{m_{W,1}}{2}}\sgn^{m_{W,1}+1} \rtimes \pi_W'))
    \] 
    and $\pi_V\boxtimes(|\cdot|^{s'}\rtimes |\cdot|^{s_{W,1}+\frac{m_{W,1}}{2}}\sgn^{m_{W,1}+1} \rtimes \pi_W')$ is in case (a)(b) with
    \[N(\sigma_V\boxtimes(|\cdot|^{s'}\times |\cdot|^{s_{W,1}+\frac{m_{W,1}}{2}}\sgn^{m_{W,1}+1} \rtimes \sigma_W'))=N(\sigma_V,\sigma_W), \text{ and}\]
    \[    M(\sigma_V\boxtimes(|\cdot|^{s'}\rtimes |\cdot|^{s_{W,1}+\frac{m_{W,1}}{2}}\sgn^{m_{W,1}+1} \rtimes \sigma_W')\leqslant M(\sigma_V,\sigma_W)
    \]
    
    \end{enumerate}
 \end{enumerate}
\end{proof}

\subsection{The second inequality}\label{section: second inequality}
In this section, we prove the counterpart of the "second inequality" in \cite[\S 1.6]{moeglin2012conjecture}, that is, $m(\pi_V\boxtimes \pi_W)\leqslant m(\RI(\sigma_{X^+},\pi_V,\pi_W))$ in the setting of Conjecture \ref{conjecture: main}. It is not hard to imagine that the proof for the second inequality is essentially  applying Theorem \ref{thm: main intro} twice. Still, there are two details we need to pay attention to:
\begin{enumerate}
\item We need  to ensure the pair of representations always associates to a basic admissible pair every time we apply Theorem \ref{thm: main intro}(1).
    \item Since we have not proven the inequality for finite-length $\pi_V$ and $\pi_W$, we need to ensure our representation on the  $\RG_V$-block of the inducing datum is always irreducible when applying Theorem \ref{thm: main intro}(3).
\end{enumerate}

For point (1), we follow the idea in \cite{jiang2010uniqueness}\cite{liu2013uniqueness} and consider the parabolic induction with a spherical principal series representation (Theorem \ref{thm: main intro}). For point (2), we choose an appropriate parameter of the spherical principal series representation to ensure the irreducibility of the parabolic induction.
\begin{thm}\label{thm: second inequality}
Given irreducible Casselman-Wallach representations $\pi_{V_0},\pi_{W_0}$ and generic Casselman-Wallach  representations $\sigma_V,\sigma_W$, we set 
\[
\pi_V=\sigma_{V}\rtimes \pi_{V_0},\quad
\pi_W=\sigma_{W}\rtimes \pi_{W_0},
\]
then
\[
m(\pi_V\boxtimes \pi_W)\geqslant m(\pi_{V_0}\boxtimes \pi_{W_0})
\]    
\end{thm}
\begin{proof}
We consider $W^+=(X^+\oplus Y^+)\oplus^{\perp} W$ as in Section \ref{section: double coset}. 
Since $\pi_W$ is  Casselman-Wallach, from Theorem \ref{thm: main intro}(1), we can find a spherical principal series representation $\sigma_{\udl{s}}$ such that 
\begin{equation}
m(\pi_V\boxtimes\pi_W)=m(\RI(\sigma_{\udl{s}},\pi_V,\pi_W))
\end{equation}
 
From Theorem \ref{thm: main intro}(3), we have
\begin{equation}
\begin{aligned}
m(\RI(\sigma_{\udl{s}},\pi_V,\pi_W))=m((\sigma_{\udl{s}}\rtimes\pi_W)\boxtimes\pi_V)\geqslant m(\pi_{V}\boxtimes \pi_{W_0}),& \text{ in (B) cases};\\
m(\RI(\sigma_{\udl{s}},\pi_V,\pi_W))=m(\pi_V\boxtimes(\sigma_{\udl{s}}\rtimes\pi_W))\geqslant m(\pi_{V}\boxtimes \pi_{W_0}),& \text{ in (FJ 1) cases};\\
m(\RI(\sigma_{\udl{s}},\pi_V,\pi_W))=m((\sigma_{\udl{s}}\rtimes\pi_V)\boxtimes\pi_W)\geqslant m(\pi_{V_0}\boxtimes \pi_{W}),& \text{ in (FJ 2) cases}.\\
\end{aligned}
\end{equation}

From Theorem \ref{thm: main intro}(1), we can find a spherical principal series representation $\sigma_{\udl{s}'}$ such that 
\[
\begin{aligned}
    m(\pi_{V}\boxtimes \pi_{W_0})=m((\sigma_{\udl{s}}\rtimes \pi_{W_0})\boxtimes \pi_V),& \text{ in (B) cases};\\
   m(\pi_{V}\boxtimes \pi_{W_0})=m(\pi_V\boxtimes (\sigma_{\udl{s}}\rtimes \pi_{W_0})),& \text{ in (FJ 1) cases}; \\
    m(\pi_{V_0}\boxtimes \pi_{W})=m((\sigma_{\udl{s}}\rtimes\pi_{V_0})\boxtimes  \pi_{W}),& \text{ in (FJ 2) cases}.
\end{aligned}
\]

From Theorem \ref{thm: main intro}(1), we can find $s\in \BC$ such that 
\[
\begin{aligned}
m((\sigma_{\udl{s}}\rtimes \pi_{W_0})\boxtimes \pi_V)=m((|\cdot|^s\rtimes\pi_V)\boxtimes (\sigma_{\udl{s}}\rtimes \pi_{W_0})),& \text{ in (B) cases};\\
  m(\pi_V\boxtimes (\sigma_{\udl{s}}\rtimes \pi_{W_0}))=m(\pi_V\boxtimes (|\cdot|^s\times\sigma_{\udl{s}}\rtimes \pi_{W_0})),& \text{ in (FJ 1) cases (RHS is a (FJ 2) case)}; \\
    m((\sigma_{\udl{s}}\rtimes\pi_{V_0})\boxtimes  \pi_{W})=m((|\cdot|^s\times\sigma_{\udl{s}}\rtimes\pi_{V_0})\boxtimes  \pi_{W}),& \text{ in (FJ 2) cases (RHS is a (FJ 1) case)}.
\end{aligned}
\]

Similarly, 
\begin{equation}
m(\pi_{V_0}\boxtimes\pi_{W_0})=m(\RI(\sigma_{\udl{s}'},\pi_{V_0},\pi_{W_0})).
\end{equation}


From \cite[Theorem 1.1]{speh1980reducibility} and Langlands classification, we may also assume $\sigma_{\udl{s}'}\rtimes\pi_{W_0}$ is irreducible.

From Theorem \ref{thm: main intro}(3), we have
\begin{equation}
\begin{aligned}
m((|\cdot|^s\rtimes\pi_V)\boxtimes (\sigma_{\udl{s}'}\rtimes \pi_{W_0}))\geqslant m((\sigma_{\udl{s}'}\rtimes\pi_{W_0})\boxtimes \pi_{V_0}),& \text{ in (B) cases};\\
m(\pi_V\boxtimes (|\cdot|^s\times\sigma_{\udl{s}}\rtimes \pi_{W_0}))\geqslant m(\pi_{V_0}\boxtimes (|\cdot|^s\times\sigma_{\udl{s}}\rtimes\pi_{W_0}),& \text{ in (FJ 1) cases};\\
m((|\cdot|^s\times\sigma_{\udl{s}}\rtimes\pi_{V_0})\boxtimes  \pi_{W})\geqslant m((|\cdot|^s\times\sigma_{\udl{s}}\rtimes\pi_{V_0})\boxtimes  \pi_{W_0}),& \text{ in (FJ 2) cases}.\\
\end{aligned}
\end{equation}

Recall that our choice of $s$ and $\udl{s}'$ are in general positions, then we may assume
\begin{equation}
\begin{aligned}
m((\sigma_{\udl{s}'}\rtimes\pi_{W_0})\boxtimes \pi_{V_0})=m(\pi_{V_0},\pi_{W_0}),&   \text{ in (B) cases};\\
m(\pi_{V_0}\boxtimes (|\cdot|^s\times\sigma_{\udl{s}}\rtimes\pi_{W_0})=m(\pi_{V_0},\pi_{W_0}),&     \text{ in (FJ 1) cases};\\
m((|\cdot|^s\times\sigma_{\udl{s}}\rtimes\pi_{V_0})\boxtimes  \pi_{W_0})=m(\pi_{V_0},\pi_{W_0}),&     \text{ in (FJ 2) cases}.
\end{aligned}
\end{equation}

In conclusion,
\[
m(\pi_V\boxtimes\pi_W)\geqslant m(\pi_{V_0}\boxtimes\pi_{W_0}).
\]

This completes the proof for "the second inequality".

    
\end{proof}

\section{The local Gan-Gross-Prasad conjecture}\label{section: GGP}
In this section,  we introduce the setting of the local Gan-Gross-Prasad conjecture and give a uniform proof from the tempered cases assuming Theorem \ref{thm: conj 3}.   The further reduction to tempered basic cases is given in Section \ref{section: 5.3}.

\subsection{Statement of conjecture}

When $F$ is Archimedean, $\RG(F)$ is one of the following in the Bessel cases
\begin{equation}\label{equ: GGP Bessel cases}
\begin{aligned}
&\RU(p+r+1,q+r)\times \RU(p,q),  \SO(p+r+1,q+r)\times\SO(p,q),\SO(n+2r+1,\BC)\times \SO(n,\BC),\\
\end{aligned}
\end{equation}
or one of the following in the Fourier-Jacobi cases 
\begin{equation}\label{equ: GGP FJ}
\begin{aligned}
&\RU(p+r,q+r)\times(\RU(p,q)\rtimes \CH_{2p+2q+1}(\BR)),\Sp(2n+2r,\BR)\times(\Sp(2n,\BR)\rtimes \CH_{2n+1}(\BR)), \\
&\Mp(2n+2r,\BR)\times(\Mp(2n,\BR)\rtimes \CH_{2n+1}(\BR)),\Sp(2n+2r,\BC)\times(\Sp(2n,\BC)\rtimes \CH_{2n+1}(\BC)),
\end{aligned}
\end{equation}
where $r\geqslant 0$. We define $\RH,\xi$ as in Section \ref{section: Bessel Fourier-Jacobi}.  We  give a uniform proof for the conjecture and its variation when $r<0$ (FJ 2), and in the section, we give a uniform statement of the conjecture.

\subsubsection{Vogan $L$-packet} 
When $F$ is Archimedean, the Weil-Deligne group $\mathrm{WD}_F$ is equal to the Weil group \[\CW_F=\begin{cases}
\BC^{\times}&\text{ when }F=\BC,\\
\BC^{\times}\cup \BC^{\times}j&\text{ when }F=\BR,
\end{cases}\]
where $j^2=-1$. For a connected reductive group $\RG$, R. Langlands associated an $L$-packet $\Pi_{\varphi}(\RG)$ to every $L$-parameter $\varphi:\CW_F\to {}^L\RG$.  For every $\alpha\in H^1(F,\RG)$, we define $\RG_{\alpha}$ be the inner twist of $\RG$ with $\alpha$, and we call $\RG_{\alpha}$ a \textbf{pure inner form} of $\RG$. From the definition of the Langlands groups (which is only dependent on the dual group $\wh{\RG}$ and its splitting), we have ${}^L\RG={}^L\RG_{\alpha}$. Following the work of Vogan (\cite{vogan1993local}), we can define the \textbf{Vogan $L$-packet} associate to $\varphi$ as
\begin{equation}\label{equ: define Vogan}    
\Pi_{\varphi}^{\Vogan}=\bigsqcup_{\alpha\in H^1(F,\RG)}\Pi_{\varphi}(\RG_{\alpha}).
\end{equation}

A \textbf{Whittaker datum} $\Fw$ for a reductive group $\RG$ is a triple $(\RG^*,\RB^*,\psi^*)$ where $\RG^*$ is a quasi-split pure inner form of $\RG$, $\RB^*$ is a Borel subgroup of $\RG^*$, and $\psi^*$ is a generic character of the unipotent radical of $\RB^*$. An $L$-parameter is called \textbf{generic} if $\Pi_{\varphi}^{\Vogan}$ contains a generic representation (with respect to a Whittaker datum),  and this is independent of the choice of Whittaker datum (\cite[\S 18]{gan2012symplectic}).

It was conjectured by Vogan and known over Archimedean local fields (\cite[Theorem  6.3]{vogan1993local}), that, fixing a Whittaker datum of $(\RG_{\alpha})_{\alpha\in H^1(F, \RG)}$, there is a bijection 
\begin{equation}
\pi\in\Pi^{\Vogan}_{\varphi}\longleftrightarrow \chi_{\pi}\in\widehat{\CS}_{\varphi}.
\end{equation}
Here $\widehat{\CS}_{\varphi}$ is the set of character of  component group $\CS_{\varphi}$ of $\RS_{\varphi}$, where $\RS_{\varphi}$ is the centralizer of the image $\Im(\varphi)$ in the dual group $\wh{\RG}$. 

 In the Bessel cases,  $\RG_V=\RU(V),\SO(V)$, so $\RG_V,\RG_W$ are connected and reductive. 

We choose the Whittaker datum as in \cite[\S 12]{gan2012symplectic} and obtain a bijection (\ref{equ: isomorphism component group}). Gan, Gross, and Prasad suggested that one may consider the relevant Vogan packet
\[
\Pi^{\Vogan}_{\varphi,\rel}=\bigsqcup_{\alpha\in H^1(F,\RH)}\Pi_{\varphi}(\RG_{\alpha}).
\]


In Fourier-Jacobi cases, we define $\wt{\RG}_W$ as in (\ref{equ: def of tilde GW}).  Given $L$-parameter of $\wt{\RG}_W:
\wt{\varphi}_W:\CW_F\to {}^L\wt{\RG}_W$.

\begin{enumerate}
    \item When $\RG_W(F)\neq \Sp(2n,\BR)$,  $\wt{\RG}_W$ is reductive, then we define $\Pi^{\Vogan}_{\wt{\varphi}_W}$ as in (\ref{equ: define Vogan}).
\item 
When $\RG_W(F)=\Sp(2n,\BR)$, the parameter $\wt{\varphi}_W:\CW_{\BR}\to \wt{\RG}_W=\Sp(\dim_EW,\BC)$ also  gives an $L$-parameter 
\begin{equation}\label{equ: correspondence for metaplectic}
\varphi_{W'}:\CW_{\BR}\to {}^L\RG_{W'}=\Sp(\dim_EW,\BC)
\end{equation}
of $\SO(W')$ where $W'$ is a quadratic space over $\BR$ with $\dim_{\BR} W'=\dim_{\BR} W+1$ and $\disc(W')=1$. 
The Shimura-Waldspurger correspondence (\cite{adams1998genuine}) gives an isomorphism 
\[
\theta_W:\Pi^{\mathrm{irr}}_{\CaWa}(\SO(W'))\to \Pi^{\mathrm{irr,genu}}_{\CaWa}(\Mp(W))
\]

The \textbf{Vogan $L$-packet for metaplectic groups} is defined as  \[
\Pi^{\Vogan}_{\wt{\varphi}_W}=\theta_W(\Pi^{\Vogan}_{{\varphi}_{W'}}).
\]


 The Vogan $L$-packet for metaplectic groups also corresponds to characters of the component group via the following bijections:
\begin{equation}
\wt{\pi}_W\in\Pi^{\Vogan}_{\wt{\varphi}_W}\longleftrightarrow \theta_{W}^{-1}(\wt{\pi}_W)\in \theta_W^{-1}(\Pi^{\Vogan}_{\wt{\varphi}_W})\longleftrightarrow\chi_{\wt{\pi}_W}=\chi_{\theta_W^{-1}(\wt{\pi}_W)}\in\widehat{\CS}_{\wt{\varphi}_W}
\end{equation}
\end{enumerate}

With the metaplectic cases discussed in Section \ref{section: subgroup}, for a  representation $\wt{\pi}_W\in \Pi_{\wt{\varphi}_W}^{\Vogan}$, the representation  $\pi_{W}\otimes\omega_{W,\psi_F}$ factors through a representation in  $\Pi_{\CaWa}(\RG_W\rtimes \CH(W))$, which we also denote by $\pi_{W}\otimes\omega_{W,\psi_F}$.

\subsubsection{Distinguished character} Following \cite{gan2012symplectic},  we consider the standard representation $\std_{V}$ of ${}^L\RG_V$ and we call $\varphi_V^{\ss}=\std_V\circ \varphi_V$ the \textbf{semisimplification} of $\varphi_V$. We set 
\[
\RM_V=\begin{cases}
 \std_V \oplus \BC & \text{ when } \RG_V=\Sp(V),\\
  \std_V & \text{ otherwise.}
\end{cases}
\]

For an $L$-parameter $\varphi_W$ of $\RG_W$ (here $\RG_W$ does not necessary be of the same type as $\RG_V$), we define $\varphi^{\ss}_{W}$ and $\RM_W$ accordingly.

We define the following distinguished characters in Conjecture \ref{conj: Bessel intro} and Conjecture \ref{conj: intro FJ} in the Archimedean cases as in \cite{gan2012symplectic}.
\begin{equation}\label{equ: definition of character}
\begin{aligned}
&\chi_{\varphi_V,W}\left(s_V\right)=
\sgn_{\varphi_V,W}(s_V)\cdot
\varepsilon\left(\frac{1}{2}, \mathrm{M}_V^{s_V=-1} \otimes \mathrm{M}_W, \psi_F\right),\quad s_V\in \CS_V, \text{ and}\\
&\chi_{\varphi _W,V}\left(s_W\right)=\sgn_{V,\varphi_W}(s_W)\cdot\varepsilon\left(\frac{1}{2}, \mathrm{M}_W^{s_W=-1} \otimes \mathrm{M}_V, \psi_F\right),\quad s_W\in \CS_W.
\end{aligned}
\end{equation} 
where the space $\RM_V^{s_V=-1}$, $\RM_W^{s_W=-1}$ are respectively the $s_V=(-1)$-eigenspace of $\RM_V$, the $s_W=(-1)$-eigenspace of $\RM_W$, and $\varepsilon(\dots)$ is the local root number defined by Rankin-Selberg integral (\cite{jacquet2009archimedean}). The sign characters are defined as the following
\[
\sgn_{\varphi_V,W}(s_V)=\begin{cases}
  1 & \text{ when }\RG_V=\RU(V),\\
  \operatorname{det}\left(-\mathrm{Id}_{\mathrm{M}_V^{s_V=-1}}\right)^{\frac{\operatorname{dim} \mathrm{M}_W}{2}} \cdot \operatorname{det}\left(-\mathrm{Id}_{\mathrm{M}_W}\right)^{\frac{\operatorname{dim} \mathrm{M}_V^{s_V=-1}}{2}} & \text{ otherwise}.
\end{cases}
\]
\[
\sgn_{V,\varphi_W}(s_W)=\begin{cases}
 1& \text{when }\RG_V=\RU(V)\\
 \operatorname{det}\left(-\operatorname{Id}_{\mathrm{M}_W^{s_W=-1}}\right)^{\frac{\operatorname{dim} \mathrm{M}_V}{2}} \cdot \operatorname{det}\left(-\operatorname{Id}_{\mathrm{M}_V}\right)^{\frac{\operatorname{dim}{\mathrm{M}_W^{s_W=-1}}}{2}}  & \text{otherwise.}
\end{cases}
\]

\begin{lem}\label{lem: no self-dual sub}
When $\RG_V=\RU(V),\SO(V),\Sp(V),\Mp(V)$, suppose
\[
\varphi_V^{\ss}=\varphi_V^{\GL}\oplus \varphi_{V_0}^{\ss}\oplus (\varphi_V^{\GL})^{\vee},
\]
where $\varphi^{\GL}_V$ has no self-dual subrepresentation, 
then there is an isomorphism
\[
\CS_{\varphi_V}\to \CS_{\varphi_{V_0}}.
\]
In particular, when $\varphi_{V_0},\varphi_V$ are generic $L$-parameters,
\[
|\Pi_{\varphi_{V_0}}^{\Vogan}|=|\Pi_{\varphi_V}^{\Vogan}|.
\]
\end{lem}
\begin{proof}
  When $\RG_V=\RU(V),\SO(V),\Sp(V)$, the result follows from \cite[\S 4]{gan2012symplectic}. When $\RG_V=\Mp(V)$, an $L$-parameter $\varphi_V:\CW_F\to {}^L\RG_V=\Sp(\dim V,\BC)$ can be treated as $L$-parameter $\varphi_{V'}$ of $\SO(V')$ for a quadratic space  $V'$ with dimension equal to $\dim V+1$ and discriminant equal to $1$.
\end{proof}
Then we check by definition that we have the following lemma.
\begin{lem}
For $L$-parameters $\varphi_V,\varphi_{V_0},\varphi_W,\varphi_{W_0}$ of $\RG_V,\RG_{V_0},\RG_W,\RG_{W_0}$ satisfying
\[
\varphi_V^{\ss}=\varphi_V^{\GL}\oplus \varphi_{V_0}^{\ss}\oplus (\varphi_V^{\GL})^{\vee},\quad \varphi_W^{\ss}=\varphi_W^{\GL}\oplus \varphi_{W_0}^{\ss}\oplus (\varphi^{\GL}_W)^{\vee}
\]    
such that $\varphi_V^{\GL},\varphi_W^{\GL}$ have no self-dual subrepresentation, then for $s_V\in \CS_{\varphi_{V_0}}=\CS_{\varphi_V}$ and $s_W\in \CS_{\varphi_{W_0}}=\CS_{\varphi_W}$
\[
\chi_{\varphi_V,W}(s_V)=\chi_{\varphi_{V_0},W_0}(s_V),\quad 
\chi_{\varphi_W,V}(s_W)=\chi_{\varphi_{W_0},V_0}(s_W).
\]
\end{lem}


\subsubsection{Results in this article} This article gives uniform proof for these conjectures from the tempered basic cases by modifying the proof in \cite{moeglin2012conjecture}.

\begin{thm}\label{thm: GGP hold}
When $F$ is Archimedean, if  Conjecture \ref{conj: Bessel intro} and Conjecture \ref{conj: intro FJ} holds in basic cases for tempered $L$-parameters, then it holds in general.
\end{thm}

In most situations, the local Gan-Gross-Prasad conjecture in tempered basic cases has been proved case-by-case using various approaches.

For Bessel models, Beuzart-Plessis proved the multiplicity-one part of the conjecture for unitary groups in \cite{beuzart2019local} using  the local trace formula.  Xue proved the whole conjecture for unitary Bessel cases using theta correspondence in \cite{xue1bessel}. Luo proved the multiplicity-one part of the conjecture for special orthogonal groups in  \cite{luothesis}. The author and Luo proved the epsilon-dichotomy part of the conjecture for special orthogonal groups in \cite{chen2022local} by simplifying Waldspurger's proof in non-Archimedean cases. 

For Fourier-Jacobi models, Xue proved the skew-unitary cases in \cite{xuefourier}. There is an ongoing project of the author with  Chen and  Zou for the symplectic-metaplectic tempered basic cases. 

\subsection{Reduction to tempered basic cases}
In this section, we reduce Conjecture \ref{conj: Bessel intro} and \ref{conj: intro FJ} to the tempered cases based on Theorem \ref{thm: conj 3} using a \textit{structure theorem}. With this reduction, the proof for Theorem \ref{thm: GGP hold} can be completed with the reduction to basic cases using Theorem \ref{thm: main intro}(1).

The  \textit{structure theorem} studies the representations in generic packets and expresses them as Schwartz induction from inducing data of specific types. We need the structure theorem for connected reductive groups $\RG_V,\RG_W$ in all cases and for $\wt{\RG}_W$ in the Fourier-Jacobi cases.

\begin{thm}\label{thm: structure}
When $\RG_V=\RU(V),\SO(V),\Sp(V),\wt{\Sp}(V)$, for every generic $L$-parameter $\varphi_V$ of $\RG_V$, there is a decomposition 
\[
\varphi_V^{\ss}=\varphi^{\GL}_{V}\oplus \varphi_{V_0}^{\ss}\oplus (\varphi^{\GL}_{V})^{\vee}
\]
such that $\varphi^{\GL}_V$ has no self-dual subrepresentation and $\varphi_{V_0}$ is tempered. Moreover, there is a bijection
\[
\Pi_{\varphi_{V_0}}^{\Vogan}\to \Pi_{\varphi_V}^{\Vogan}
\]\[
{\pi}_{V_0}\mapsto \sigma_V\rtimes \pi_{V_0}
\]
for  the general linear group representation $\sigma_V$ with $L$-parameter  $\varphi_V^{\GL}$. Moreover, the isomorphism is compatible with the isomorphism in Lemma \ref{lem: no self-dual sub}.
\end{thm}

\begin{proof}[Proof for Theorem \ref{thm: structure}]
We check the case when $F=\BR$ and complex cases can be treated similarly. Based on the representation theory of $\CW_{\BR}$, $\varphi_V^{\ss}$ can be decomposed into a direct sum
\[
\bigoplus_{\substack{k=1,2\\
i\in \RI_k}}|\cdot|^{s_{V,i}}\varphi_{m_{V,i}}^{(k_{V,i})},
\]
 where $\RI_1\coprod\RI_2=\{1,\cdots,l\}$, the unitary character
$\varphi_{m_{V,i}}^{(1)}$ of
$\CW_{\BR}=\BC^{\times}\cup \BC^{\times} j\ (j^{2}=-1)$ is defined by
\[
\varphi_{m_{V,i}}^{(1)}(z)=1,\quad \varphi_{m_{V,i}}^{(1)}(z\cdot j)=(-1)^{m_{V,i}}, \quad z\in \BC,
\]
and $\varphi_{m_{V,i}}^{(2)}\ (m_{V,i}\in \BN)$ is the two-dimensional representation of $\CW_{\BR}$ with basis $u,v$ satisfying
\[
\begin{aligned}
\varphi_{m_{V,i}}^{(2)}(z)u=u, &\quad  \varphi_{m_{V,i}}^{(2)}(z\cdot j)u=(-1)^{m_{V,i}}v,\\
\varphi_{m_{V,i}}^{(2)}(z)v=v, &\quad  \varphi_{m_{V,i}}^{(2)}(z\cdot j)v=u,\ z\in \BC.
\end{aligned}
\] 

We take 
\[
\varphi_V^{\GL}=\bigoplus_{\substack{\Re(s_{V,i})>0\\
k=1,2,
i\in \RI_k}}|\cdot|^{s_{V,i}}\varphi_{m_{V,i}}^{(k_{V,i})}, \quad
\varphi_{V_0}^{\ss}=\bigoplus_{\substack{\Re(s_{V,i})=0\\
k=1,2, i\in\RI_k}}|\cdot|^{s_{V,i}}\varphi_{m_{V,i}}^{(k_{V,i})}
\]
It is obvious that $\varphi^{\GL}$ has no self-dual subrepresentation and based on the invariant bilinear form of the standard representation of ${}^L\RG$, we have
\[
\varphi_V^{\ss}=\varphi^{\GL}_{V}\oplus \varphi_{V_0}^{\ss}\oplus (\varphi^{\GL}_{V})^{\vee}.
\]

When $\varphi_V$ is generic, there exists a generic representation $\pi_V\in \Pi^{\Vogan}_{\varphi_V}$.
From \cite[Theorem 6.2, equivalence of (a) and (e)]{vogan1978gelfand}, we obtain a relation between the $s_{V,i},m_{V,i}$, which is equivalent to 
\[
L(s,\varphi_V,\rho)\text{ is holomorphic at }s=1.
\] 
Here
 \[
 \rho=\begin{cases}
\mathrm{As}^{(-1)^{\dim_E V}}    &\text{ when } \RG_V  \text{ is the unitary group,}\\
  \Ad  & \text{ otherwise.}
 \end{cases}
 \]

Now we take a representation $\pi_V$ in the Vogan $L$-packet $\Pi_{\varphi_V}^{\Vogan}$. 

When $\RG_V=\RU(V),\SO(V),\Sp(V)$, from the Langlands classification,  $\pi_V$ is the quotient of 
\begin{equation}\label{equ: langlands classification for piV}
|\det|^{s_{V,1}}\sigma_{V,1}\times\cdots \times |\det|^{s_{V,l}}\sigma_{V,l} \rtimes \pi_{V_0}
\end{equation}
for $\sigma_{V,i}=\begin{cases}
   \sgn^{m_{V,i}} &\text{ when } k_i=1 \\
    D_{m_{V,i}} &\text{ when }k_i=2
\end{cases}$, and  $\pi_{V_0}$ is tempered.

To show $\pi_V=\sigma_V \rtimes\pi_{V_0}$, it  suffices to verify that the representation is irreducible.  This was proved in \cite{xue2020bessel} when $\rho$ is the Asai $L$-function and proved in \cite[Lemma 4.0.4]{chen2021local} when $\RG_V$ is the adjoint $L$-function.

When $\RG_V=\wt{\Sp}(V)$,  recall the Vogan $L$-packet is defined by
\[
\Pi_{\varphi_{V}}^{\Vogan}=\theta_V(\Pi_{\varphi_{V'}}^{\Vogan}),
\] 
where $V', \varphi_{V'}$ are defined as in (\ref{equ: correspondence for metaplectic}).

Based on the result for $\varphi_{V'}$, there is a decomposition 
\[
\varphi_{V'}^{\ss}=\varphi_V^{\GL}\oplus \varphi_{V_0'}^{\ss}\oplus (\varphi_V^{\GL})^{\vee}.
\]

Then for every representation $\pi_V\in\Pi_{\varphi_{V}}^{\Vogan}$,  
\[
\theta_V^{-1}(\pi_V)=\sigma_V\rtimes \pi_{V_0'}.
\]

From \cite[Theorem 5.1]{adams1998genuine}, we have
\[
\pi_V=\wt{\sigma}_V\rtimes \theta_{V_0}^{-1}(\pi_{V_0'})=\sigma_V\rtimes \theta_{V_0}^{-1}(\pi_{V_0'}),
\]
where $\wt{\sigma}_V$ is the $\wt{\GL}$-representation correspond to $\sigma_V$  in Lemma \ref{lem: classification of representations}(1). This completes the proof for the lemma.
\end{proof}

\begin{pro}
When $F$ is Archimedean, if  Conjecture \ref{conj: Bessel intro} and Conjecture \ref{conj: intro FJ} holds for tempered $L$-parameters, then it holds in general.
\end{pro} \begin{proof}[Proof for Theorem \ref{thm: GGP hold}]

We first reduce Conjecture \ref{conj: Bessel intro} and \ref{conj: intro FJ} to the tempered cases.
 \begin{itemize}
     \item In the Bessel cases, given generic $L$-parameters $\varphi_{V}$ of $\RG_V$ and $\varphi_W$ of $\RG_W$. If we identify  the component group $\CS_{\varphi_{V_0}\times \varphi_{W_0}}$ with $\CS_{\varphi_V\times \varphi_W}$  via the isomorphisms \[
    \Pi_{\varphi_{V_0}}^{\Vogan}\to \Pi_{\varphi_V}^{\Vogan},\quad\Pi_{\varphi_{W_0}}^{\Vogan}\to \Pi_{\varphi_W}^{\Vogan}
    \]
    given by Theorem \ref{thm: structure}, from Lemma \ref{lem: no self-dual sub}, $\chi_{\varphi_{V_0}\times \varphi_{W_0}}=\chi_{\varphi_{V}\times \varphi_{W}}$. Then, with Theorem \ref{thm: GGP hold}, we reduce Conjecture \ref{conj: Bessel intro} to the tempered cases for  $\varphi_{V_0}, \varphi_{W_0}$.
\item In Fourier-Jacobi cases, given generic $L$-parameters $\varphi_{V}$  of $\RG_V(F)$, and $\wt{\varphi}_W$ of $\wt{\RG}_W(F)$.  If we identify  the component group $\CS_{\varphi_{V_0}}$ with $\CS_{\varphi_V}$  and identify $\CS_{\wt{\varphi}_{W_0}}$ with $\CS_{\varphi_V\times \varphi_W}$  via the isomorphisms \[
    \Pi_{\varphi_{V_0}}^{\Vogan}\to \Pi_{\varphi_V}^{\Vogan},\quad\Pi_{\wt{\varphi}_{W_0}}^{\Vogan}\to \Pi_{\wt{\varphi}_W}^{\Vogan}
    \]
    given by Theorem \ref{thm: structure}, from Lemma \ref{lem: no self-dual sub}
    \[\chi_{\varphi_{V_0}\times \varphi_{W_0}}=\chi_{\varphi_{V}\times \varphi_{W}}.
    \]
    
    Then, with Theorem \ref{thm: conj 3} and Lemma \ref{lem: compatibility with tensor Weil}, we reduce Conjecture \ref{conj: intro FJ} to the tempered cases for  $\varphi_{V_0}, \wt{\varphi}_{W_0}$.
 \end{itemize}
\end{proof}

Therefore, to complete the proof for Theorem \ref{thm: GGP hold}, it remains to prove that the tempered basic cases imply the tempered cases.
 \begin{pro}
When $F$ is Archimedean, if  Conjecture \ref{conj: Bessel intro} and Conjecture \ref{conj: intro FJ} holds  in basic cases for tempered $L$-parameters, then
Conjecture \ref{conj: Bessel intro} and Conjecture \ref{conj: intro FJ} holds for tempered $L$-parameters. 
 \end{pro}
 
\begin{proof}
 From Lemma \ref{lem: no self-dual sub}, we choose $(s_1,\cdots,s_{r^+})\in \BC^{r^+}$ such that $s_i+s_j\neq 0$, $1\leqslant i<j\leqslant r^+$, there are isomorphisms
\[
\begin{aligned}
\Pi^{\Vogan}_{\varphi_W}\to \Pi^{\Vogan}_{\varphi_{W,\udl{s}}},&\quad \pi_W\mapsto \sigma_{\udl{s}}\rtimes \pi_W \text{ in (B) cases},\\
\Pi^{\Vogan}_{\wt{\varphi}_W}\to \Pi^{\Vogan}_{\wt{\varphi}_{W,\udl{s}}},&\quad \wt{\pi}_W\mapsto \sigma_{\udl{s}}\rtimes \wt{\pi}_W \text{ in (FJ 1) cases},\\
\Pi^{\Vogan}_{\varphi_V}\to \Pi^{\Vogan}_{\varphi_{V,\udl{s}}},&\quad \pi_V\mapsto \sigma_{\udl{s}}\rtimes \pi_V \text{ in (FJ 2) cases},
\end{aligned}
\]
where $\varphi_{W,\udl{s}}^{\ss}=\varphi_{\udl{s}}\oplus \varphi_W^{\ss}\oplus\varphi_{\udl{s}}^{\vee}$ for $\varphi_{\udl{s}}=\oplus_{i=1}^{r^+}|\cdot|^{s_i}$.

From Theorem \ref{thm: main intro}(1), we can find  $\udl{s}=(s_1,\cdots,s_{r^+})\in (i\BR)^{r^+}$ such that $s_i+s_j\neq 0$, for $1\leqslant i<j\leqslant r^+$, and
\begin{equation}\label{equ: reduction to codimension-one}
m(\pi_V\boxtimes\pi_W)=m(\RI(\sigma_{\underline{s}},\pi_V,\pi_W))
\end{equation}
for all $\pi_V\in \Pi_{\varphi_V}^{\Vogan}$ and (i) $\pi_W\in \Pi_{\varphi_W}^{\Vogan}$ in (B) cases, (ii) $\pi_W=\wt{\pi}_W\otimes \omega_{W,\psi}$, $\wt{\pi}_W\in \Pi_{\wt{\varphi}_W}^{\Vogan}$ in (FJ 1)(FJ 2) cases. 

Notice that
\[m(\RI(\sigma_{\underline{s}},\pi_V,\pi_W))=\begin{cases}
    m((\sigma_{\udl{s}}\rtimes \pi_W)\boxtimes\pi_V)& \text{ in (B) cases},\\
    m(\pi_V\boxtimes (\sigma_{\udl{s}}\rtimes \pi_W))& \text{ in (FJ 1) cases},\\
    m((\sigma_{\udl{s}}\rtimes \pi_V)\boxtimes \pi_W)& \text{ in (FJ 2) cases}. 
\end{cases}\]
Since $(\varphi_{W,\udl{s}},\varphi_V)$ in (B) cases, $(\varphi_V,\varphi_{W,\udl{s}})$ in (FJ 1) cases and $(\varphi_{V,\udl{s}},\varphi_{W})$ in (FJ 2) cases are tempered $L$-parameters for basic triples. Then Conjecture \ref{conj: Bessel intro} and \ref{conj: intro FJ} for tempered $L$-parameters $(\varphi_V,\varphi_W)$ follows from the tempered basic cases.
 \end{proof}

\appendix
\section{A Vanishing Lemma}\label{section: vanishing app}
\smallskip
\begin{center}by\textsc{ Cheng Chen, Rui Chen, Jialiang Zou}\end{center}

In this section, we prove Theorem \ref{thm: vanishing in appendix}(2), a vanishing lemma for $\Hom$-functor, which implies Lemma \ref{lem: three properties}(3) with the reciprocity law (Lemma \ref{lem: reciprocity}). We follow the methods of M\oe glin and Waldspurger in \cite[p10]{moeglin2012conjecture} for special orthogonal groups over non-Archimedean local fields. These arguments hold for other classical groups over non-Archimedean local fields as well. 

Let $F=\BR,\BC$, and $E$ be an algebraic field extension of $F$. For a given $\epsilon$-hermitian space $V$ over $E$. We define $\RG_V$ as in Definition \ref{def: GV}, then 
\[
\RG_V=\RU(V), \SO(V), \RO(V), \Sp(V), \text{ or }\Mp(V).
\]

From Langlands classification, every irreducible Casselman-Wallach representation of $\RG_V(F)$ is the unique quotient  of
  a normalized parabolic induction
\begin{equation}\label{equ: Langlands  quotient}
|\det|^{s_1}\sigma_{1}\times\cdots\times|\det|^{s_l}\sigma_{l}\rtimes\pi_{V_0},
\end{equation}
 where $\sigma_{1},\cdots,\sigma_{l}$ are irreducible tempered representations of $\GL_{n_i}(E)$,  $\pi_{V_0}$ is an irreducible tempered representation of $\RG_{V_0}(F)$, $s_1,\cdots,s_l\in \BR$ satisfying $s_1\geqslant \cdots \geqslant s_l>0$. We denote by $\mathrm{LQ}(|\det|^{s_1}\sigma_{1}\times\cdots\times|\det|^{s_l}\sigma_{l}\rtimes\pi_{V_0})$ the unique quotient of $|\det|^{s_1}\sigma_{1}\times\cdots\times|\det|^{s_l}\sigma_{l}\rtimes\pi_{V_0}$, which is known as the Langlands quotient.
 
Given another representation $\pi_{V'}=|\det|^{s'}\sigma_{0}'\rtimes \pi_{0}'$ of $\RG_V(F)$, where $\sigma_{0}'$ is irreducible and tempered representations of $\GL_{n'_0}(E)$, while $\pi_{V_0'}\in \Pi_{\FMG}(\RG_{V_0'})$ is  not necessarily irreducible  or Casselman-Wallach.

\begin{thm}\label{thm: vanishing in appendix}
In the above setting, for 
\[
\pi_V=|\det|^{s_1}\sigma_{1}\times\cdots\times|\det|^{s_l}\sigma_{l}\rtimes\pi_{V_0},
\]
we have
\begin{enumerate}
    \item \[
\Hom_{\RG_V}(\pi_{V'},\mathrm{LQ}(\pi_V))=0
\]    
when $s'>s_1$, and
\item 
\[
\Hom_{\RG_V}(\pi_{V'},\pi_V)=0
\]
when $s'>s_1$.
\end{enumerate}
\end{thm}

First of all, we prove Lemma \ref{lem: app ind}, which implies Theorem \ref{thm: vanishing in appendix}(1) for cases when $\pi_{V_0'}$ is irreducible and Casselman-Wallach using the Langlands classification for $\pi_{V_0'}$. Then we prove Lemma \ref{lem: reduction to irreducible cases} to reduce  Theorem \ref{thm: vanishing in appendix}(1) to these special cases by refining \cite[p10]{moeglin2012conjecture} using results on the second adjointness in \cite{din2021second}. Finally, we verify conditions in Theorem \ref{thm: vanishing in appendix}(1) for all irreducible subquotients of $\pi_V$ using Lemma \ref{lem: dominance}(2), we prove Theorem \ref{thm: vanishing in appendix}(2).

\begin{defin}
Let $s_1\geqslant \cdots\geqslant s_l\geqslant 0$, $\sigma_i$ are tempered representations of $\GL_{n_i}(F)$ and $n=\sum_{i=1}^l{n_i}$.
\begin{enumerate}
    \item We define the exponent $\exp(\pi)$ of the representation
\[
\pi=|\det|^{s_1}\sigma_{1}\times \cdots \times |\det|^{s_l}\sigma_l \]
of $\GL_{n}(F)$ to be the $n$-tuple
\[
(\underbrace{s_1,\cdots,s_1}_{n_1},\cdots,\underbrace{s_l,\cdots,s_l}_{n_l}).
\]
\item We define the exponent $\exp(\pi_V)$ of the representation 
\[
\pi_V=|\det|^{s_1}\sigma_{V,1}\times\cdots\times|\det|^{s_l}\sigma_{V,l}\rtimes\pi_{V_0}\]
by
\[
\exp(\pi_V)=\exp(|\det|^{s_1}\sigma_{V,1}\times\cdots\times|\det|^{s_l}\sigma_{V,l})
\]
\end{enumerate}
\end{defin}
\begin{lem}\label{lem: dominance}
\begin{enumerate}
    \item Let $\pi'$ be an irreducible component of $\pi=|\det|^{s_1}\sigma_{1}\times \cdots \times |\det|^{s_l}\sigma_l$  and $\pi'=\mathrm{LQ(|\det|^{s_1'}\sigma_{1}'\times \cdots \times |\det|^{s_{l'}'}\sigma_{l'})}$, then   
\[
s_1\geqslant s_1'\geqslant s_l' \geqslant s_l
\]
\item $\pi'_V$ be an irreducible component of $\pi_V=|\det|^{s_1}\sigma_{1}\times \cdots \times |\det|^{s_l}\sigma_l\rtimes \pi_{V_0}$  and $\pi_V'=\mathrm{LQ(|\det|^{s_1'}\sigma_{1}'\times \cdots \times |\det|^{s_{l'}'}\sigma_{l'})\rtimes \pi_{V_0'}}$, then   
\[
s_1\geqslant s_1'
\]
\end{enumerate}
\end{lem}
\begin{proof}
  For Point (1), we use the fact that the Langlands quotient is the "most non-tempered" irreducible component (\cite[Proposition 4.4.13]{borel2000continuous}), which we interpret as the dominance of $\exp(\pi)-\exp(\pi')$, that is, 
\[
\exp(\pi)-\exp(\pi')=\sum_{i=1}^{n-1}a_i(e_i-e_{i+1}),\quad a_i\geqslant 0
\]
where $e_i$ is the row vector with $1$ at the $i$-th entry and $0$ at other entries.

So we have
\[
s_1\geqslant s_1'\geqslant s_l' \geqslant s_l
\]
Similarly, the dominance of $\exp(\pi_V)-\exp(\pi_V')$ implies the Point (2).
\end{proof}
\begin{lem} \label{lem: app ind}
In the setting of Theorem \ref{thm: vanishing in appendix}, given irreducible tempered representations $\sigma_{i}'$ of $\GL_{n_i'}(E) \ (1\leqslant i\leqslant l')$, $\pi_{V_{00}'}$ of $\RG_{V_{00}'}(F)$ and $s_1',\cdots,s_{l'}'\in \BR$ satisfying $s_1'\geqslant \cdots \geqslant s_{l'}'>0$, we have
\[
\Hom_{\RG_V}(|\det|^{s_0'}\sigma_{0}'\times|\det|^{s_1'}\sigma_{1}'\times\cdots\times|\det|^{s_{l'}'}\sigma_{l'}'\rtimes\pi_{V_{00}'},\mathrm{LQ}(\pi_V))=0
\]
when $s_0'>s_1$.
\end{lem}

\begin{proof}
We denote $\pi_V'=|\det|^{s_0'}\sigma_{0}'\times|\det|^{s_1'}\sigma_{1}'\times\cdots\times|\det|^{s_{l'}'}\sigma_{l'}'\rtimes\pi_{V_{00}'}$ and we prove the lemma by mathematical induction on $\lceil 2D(\exp(\pi_V'))\rceil$, where
\[
D(\lambda_1,\cdots,\lambda_n)=\sum_{\substack{0\leqslant i<j \leqslant n\\\lambda_{j}>\lambda_i
}}(\lambda_{j}-\lambda_{i})=\sum_{0\leqslant i<j \leqslant n}\max(\lambda_{j}-\lambda_{i},0).
\]

\begin{enumerate}
    \item  When $\lceil 2D(\exp(\pi_V'))\rceil=0$, then $D(\exp(\pi_V'))=0$, then we have
\[
s_0'\geqslant s_1'\geqslant\cdots\geqslant s_{l'}'>0.
\]

In this situation, $|\det|^{s_0'}\sigma_{0}'\times|\det|^{s_1'}\sigma_{1}'\times\cdots\times|\det|^{s_{l'}'}\sigma_{l'}'\rtimes\pi_{V_{00}'}$ has a unique quotient. Suppose 
\[
\Hom_{\RG_V}(|\det|^{s_0'}\sigma_{0}'\times|\det|^{s_1'}\sigma_{1}'\times\cdots\times|\det|^{s_{l'}'}\sigma_{l'}'\rtimes\pi_{V_{00}'},\mathrm{LQ}(\pi_V))\neq 0,
\]
then $\mathrm{LQ}(\pi_V)$ is equal to this quotient, which leads to a contradiction as $s_0'>s_1$.

\item When $\lceil 2D(\exp(\pi_V'))\rceil\neq 0$, there is an  integer $i$ such that $s_i'<s_{i+1}'$, suppose 
\[
\Hom_{\RG_V}(|\det|^{s_0'}\sigma_{0}'\times\cdots\times|\det|^{s_{l'}'}\sigma_{l'}'\rtimes\pi_{V_{00}'},\pi_V)\neq 0,
\]
then there exists an irreducible subquotient $T$ of $\pi_2'=|\det|^{s_{i}'}\sigma_{i}'\times |\det|^{s_{i+1}'}\sigma_{i+1}'$ such that
\[
\Hom_{\RG_V}(|\det|^{s_0'}\sigma_{0}'\times\cdots \times T\times\cdots\times|\det|^{s_{l'}'}\sigma_{l'}'\rtimes\pi_{V_{00}'},\pi_V)\neq 0.
\]

We choose $s_{1}''\geqslant\cdots\geqslant s_{k}''$ and tempered $\sigma_{1}'',\cdots,\sigma_{k}''$ such that $T$ is the Langlands quotient of $\pi_k''=|\det|^{s_{1}''}\sigma_{1}''\times\cdots\times |\det|^{s_{k}''}\sigma_{k}''$. Then 
\[
\Hom_{\RG_V}(|\det|^{s_0'}\sigma_{0}'\times\cdots \times \pi_k''\times\cdots\times|\det|^{s_{l'}'}\sigma_{l'}'\rtimes\pi_{V_{00}'},\pi_V)\neq 0.
\]

We claim that
\begin{itemize}
\item When $|\det|^{s_i'}\sigma_i'\times |\det|^{s_{i+1}'}\sigma_{i+1}'$  is irreducible, then $k=2$ and $(s_{1}'',s_{2}'')=(s_{i+1}',s_i')$, in this situation,
\[
D(\exp(|\det|^{s_0'}\sigma_{0}'\times\cdots \times \pi_k''\times\cdots\times|\det|^{s_{l'}'}\sigma_{l'}'\rtimes\pi_{V_{00}'}))<D(\exp(\pi_V')),
\]
and this operation increases the number of reverse pairs in the tuple.
    \item  When $|\det|^{s_i'}\sigma_i'\times |\det|^{s_{i+1}'}\sigma_{i+1}'$ is reducible,  we have
\[
D(\exp(|\det|^{s_0'}\sigma_{0}'\times\cdots \times \pi_k''\times\cdots\times|\det|^{s_{l'}'}\sigma_{l'}'\rtimes\pi_{V_{00}'}))\leqslant D(\exp(\pi_V'))-\frac{1}{2}.
\]
\end{itemize}   
\end{enumerate}

With this result, the reducible case  decreases $\lceil 2D\rceil$ at least by one and the irreducible cases can only appear consequently for finitely many times. This completes the mathematical induction.

It remains to prove the claim. The irreducible situations are obvious. In reducible cases, from Lemma \ref{lem: dominance}, we have
\[
s_{i+1}'\geqslant s_1''\geqslant \cdots \geqslant s_k''\geqslant s_i'
\]
We set $i_0=\sum_{j=1}^{i-1}{n_j}+1$ and $i_1=\sum_{j=1}^{i+1}n_j$. Then, by definition
\[
\begin{aligned}
&D(\exp(\pi_V'))-
D(\exp(|\det|^{s_0'}\sigma_{0}'\times\cdots \times \pi_k''\times\cdots\times|\det|^{s_{l'}'}\sigma_{l'}'\rtimes\pi_{V_{00}'}))\\=&\sum_{j<i_0}(n_{i}\max(s_{i}'-s_j',0)+n_{i+1}\max(s_{i+1}'-s_j',0)-\sum_{t=1}^{r}n_i''\max(s_t''-s_j',0))\\
&+\sum_{j>i_1}(n_{i}\max(s_{j}'-s_i',0)+n_{i+1}\max(s_{j}'-s_{i+1}',0)-\sum_{t=1}^{r}n_i''\max(s_j'-s_{t}'',0))\\
&+n_in_j(s_{i+1}'-s_i')
\end{aligned}
\]
From \cite{speh1977reducibility} (see also \cite[Theorem 1]{prasad2017reducible}), as $|\det|^{s_i'}\sigma_i'\times |\det|^{s_{i+1}'}\sigma_{i+1}'$ is reducible, we have $s_{i+1}\geqslant s_i+\frac{1}{2}$. Therefore, it suffices to show that 
\[n_{i}\max(s_{i}'-s_j',0)+n_{i+1}\max(s_{i+1}'-s_j',0)-\sum_{t=1}^{r}n_i''\max(s_t''-s_j',0)\geqslant 0,\]
and
\[
n_{i}\max(s_{j}'-s_i',0)+n_{i+1}\max(s_{j}'-s_{i+1}',0)-\sum_{t=1}^{r}n_i''\max(s_j-s_{t}'',0)\geqslant 0.
\]
We set
\[
f_1(x)=n_{i}\max(s_{i}'-x,0)+n_{i+1}\max(s_{i+1}'-x,0),\quad f_2(x)=\sum_{t=1}^{r}n_i''\max(s_t''-x,0),
\]
\[
f_3(x)=n_{i}\max(x-s_i',0)+n_{i+1}\max(x-s_{i+1}',0),\quad f_4(x)=\sum_{t=1}^{r}n_i''\max(x-s_{t}'',0).
\]

Since $|\det|^{s_i'}\sigma_i'\times |\det|^{s_{i+1}'}\sigma_{i+1}'$ and $|\det|^{s_{1}''}\sigma_{1}''\times\cdots\times |\det|^{s_{k}''}\sigma_{k}''$ have the same central character, we have
\[
n_is_i'+n_{i+1}s_{i+1}'=\sum_{t=1}^{k}n_t''s_t''.
\]
Then $f_1(x)=f_2(x)$ when $x\leqslant s_{i}'$ and $f_3(x)=f_4(x)$ when $x\geqslant s_{i+1}'$. Since $f_1(x)=f_2(x)=0$ when $x\geqslant s_{i+1}'$ and $f_3(x)=f_4(x)=0$ when $x\leqslant s_i'$, it suffices to compare the functions when $x\in [s_i',s_{i+1}']$. Notice that the graph of  $f_1(x)$ and $f_3(x)$ are segments in the inteval and $f_2(x),f_4(x)$ are lower convex, we have
\[f_1(x)\geqslant f_2(x),\quad f_3(x)\geqslant f_4(x),\quad x\in [s_i',s_{i+1}'].\]
This completes the proof for the claim.
\end{proof}

When $\pi_{V_0'}$ is an irreducible Casselman-Wallach representation,  $\pi_{V_0'}$ is a the Langlands quotient of 
\[
|\det|^{s_1'}\sigma_{1}'\times\cdots\times|\det|^{s_{l'}'}\sigma_{l'}'\rtimes\pi_{V_{00}'}
\]
for some $s_1'\geqslant\cdots\geqslant s_{l'}'>0$. Suppose $\Hom_{\RG_V}(\pi_{V'},\pi_{V})\neq 0$.
Then 
\[
\Hom_{\RG_V}(|\det|^{s}\sigma_{0}'\times|\det|^{s_1'}\sigma_{1}'\times\cdots\times|\det|^{s_{l'}'}\sigma_{l'}'\rtimes\pi_{V_{00}'},\pi_V)\neq 0.
\]
This contradicts Lemma \ref{lem: app ind}. So we proved Theorem \ref{thm: vanishing in appendix}(1) when $\pi_{V_0'}$ is an irreducible Casselman-Wallach representation.

Then we complete the proof for Theorem \ref{thm: vanishing in appendix}(1) with the following lemma.
\begin{lem}\label{lem: reduction to irreducible cases}
For $V=V_0\oplus^{\perp}(X\oplus Y)$ and $\sigma\in \Irr(\Res_{E/F}\GL(X))$. Suppose there exists  $\pi_{V_0}'\in\Pi_{\FMG}(\RG_{V_0'})$ such that
\[
\Hom_{\RG_V}(\sigma'\rtimes \pi_{V_0'},\pi_V)\neq 0.
\] 
    Then there exists $\pi_{V_0'}'\in \Irr(\RG_{V_0'})$ such that
\[
\Hom_{\RG_V}(\sigma'\rtimes \pi_{V_0'},\pi_V)\neq 0.
\]
\end{lem}
We prove the lemma using the second adjointness as in \cite[p10]{moeglin2012conjecture}. Over Archimedean local fields, the second adjointness for $(\Fg_{V,\BC},K_V)$-modules are proved in \cite{din2021second}.

Let $\Fg_{V,\BC}$ be the complexification of $\RG_V(F)$ and $K_V$ be the maximal compact subgroup of the Lie algebra of $\RG_V(F)$, for a smooth representation $\pi$ of $\RG_V(F)$, we denote by $\pi^{\alg}$ the $(\Fg_{V,\BC},K_V)$-module associated to it, and denote by $\mathcal{M}(\mathfrak{g}_{V,\BC},K_V)$ the category of $(\Fg_{V,\BC},K_V)$-modules (not necessarily finitely generated). From the category equivalence of Harish-Chandra modules (admissible $(\Fg_{V,\BC},K_V)$-modules  of finite length) and Casselman-Wallach representations of $\RG_V(F)$, it suffices to show that there exists irreducible admissible $(\Fg_{V,\BC},K_V)$-module $\pi_V^{\alg}$ such that
\[
\Hom_{\RG_V}(\sigma_{0}^{\p,\alg}\rtimes \pi_{V_0'}^{\prime,\alg},\pi_V^{\alg})\neq 0.
\]
Here we set the parabolic induction of $(\Fg_{V,\BC},K_V)$-modules following \cite[Definition 5.1]{din2021second}. For parabolic subgroup $P_V$ with Levi decomposition $P_V=L_V\rtimes N_V$, we set
\begin{equation}\label{equ: induction on g,K}
\sigma_{0}^{\p,\alg}\rtimes \pi_{V_0'}^{\alg}=(\CO(K_V)\otimes I_{\CU(\Fp)}^{\CU(\Fg)}(\sigma_0^{\prime}\boxtimes \pi_{V_0'})^{\alg})^{K_{L_V}}_{\Fk_{V,\BC}}.
\end{equation}
where 
\begin{itemize}
    \item 
$\CO(K_V)$ is the space of regular functions of $K_V$, or say, the space of $K_V$-finite vectors in $C_c^{\infty}(K_V)$;
\item $I_{\CU(\Fp)}^{\CU(\Fg)}((\sigma_0^{\prime}\boxtimes \pi_{V_0'})^{\alg})$ is the $(\Fg_{V,\BC},K_{L_V}N_V)$-module
\[
\CU(\Fg_{V,\BC})\otimes _{\CU(\Fp_{V,\BC})}(\BC_{-\rho_P}\otimes (\sigma_0^{\prime}\boxtimes \pi_{V_0'})^{\alg}).
\]
\item 
By taking $\CO(K_V)\otimes I_{\CU(\Fp)}^{\CU(\Fg)}(\sigma_0^{\prime}\boxtimes \pi_{V_0'})$ as the space of algebraic functions from $K_V$ to $I_{\CU(\Fp)}^{\CU(\Fg)}(\sigma_0^{\prime}\boxtimes \pi_{V_0'})$, the  $K_V$-action is given by left action,   the $\Fg_{V,\BC}$-action is given by 
\[(X.f)(k)={}^{k^{-1}}X.f(k).\]
\item The action for $\Fk_{V,\BC}$-covariance is for the compatibility of  $\Fg_{V,\BC}$-action restricted to $\Fk_{V,\BC}$ and the $\Fk_{V,\BC}$-action obtained by differentiating the $K_V$-action.
\item The action of $K_{L_V}$ is given by tensor $\CO(K_V)$ and $(\sigma_0'\boxtimes \pi_{V_0'})^{\alg}$ as $K_{L_V}$-module.
\end{itemize}
In particular, the $K_V$-module structure of $\sigma_{0}^{\p,\alg}\rtimes \pi_{V_0'}^{\alg}$ is 
\begin{equation}\label{equ: Lie algebra side}
(\CO(K)\otimes (\sigma_0'\boxtimes \pi_{V_0'})^{\alg})^{K_{L_V}}.
\end{equation}

\begin{thm}(\cite[Thm. 5.10]{din2021second})
\begin{enumerate}
    \item There is a functor
\[
    F: \mathcal{M}(\mathfrak{g}_V,K_V) \lra \mathcal{M}(\mathfrak{l}_V,K_{L_V}),
\]
right adjoint to the functor of normalized parabolic induction from $\RP_{V,X}(F)$ to $\RG_V(F)$. 
\item   $F$ preserves the admissibility and finite-length property.
\end{enumerate}
\end{thm}

\begin{proof}[Proof for Lemma \ref{lem: reduction to irreducible cases}]
(1) We first prove 
$\Hom_{\RG_V}(\sigma_0^{\prime,\alg}\rtimes\pi_{V_0'}^{\alg},\pi_V^{\alg})=0$ by contradiction.
Suppose we have
$\Hom_{\RG_V}(\sigma_0^{\prime,\alg}\rtimes\pi_{V_0'}^{\alg},\pi_V^{\alg})\neq 0$.
The second adjointness implies
\[
\Hom_{\RL_V}(\sigma_{0}^{\p,\alg}\boxtimes \pi_{V_0'}^{\alg},F(\pi_V^{\alg}))\neq 0.
\]

Since $F(\pi_V^{\alg}))$ is admissible of finite-length, there exists an irreducible $\pi_{V_0'}^{\prime,\alg}$ such that
\[
\Hom_{\RL_V}(\sigma_{0}^{\p,\alg}\boxtimes \pi_{V_0'}^{\prime,\alg},F(\pi_V^{\alg}))\neq 0.
\]

(2) Finally, we show that 
$\Hom_{\RG_V}(\sigma_0'\rtimes \pi_{V_0'},\pi_V)\neq 0$ implies $\Hom_{\RG_V}(\sigma_0^{\prime,\alg}\rtimes\pi_{V_0'}^{\alg},\pi_V^{\alg})\neq 0$. It suffices to show that $\sigma_0^{\prime,\alg}\rtimes\pi_{V_0'}^{\alg}$ is dense in $\sigma_0'\rtimes \pi_{V_0'}$. 

From the Iwasawa decomposition $\RG_V(F)=\RP_V(F)$, we have 
\[
\RP_V(F)\bs \RG_V(F)=(\RP_V(F)\cap K_V)\bs K_V,
\]
then
\[
\sigma_0'\rtimes \pi_{V_0}'|_{K_V}=\Ind_{\RP_V(F)}^{\CS,\RG_V(F)}(\sigma_0'\boxtimes \pi_{V_0'})|_{K_V}=\Ind_{K_{L_V}}^{\CS,K_V}(\sigma_0'\boxtimes \pi_{V_0'}|_{K_{L_V}}).
\]

By definition, 
\begin{equation}\label{equ: reduce to compact}
\Ind_{K_{L_V}}^{\CS,K_V}(\sigma_0'\boxtimes \pi_{V_0'}|_{K_{L_V}})=(C_c^{\infty}(K_L)\wh{\otimes} \sigma_0^{\prime})^{K_{L_V}}
\end{equation}
as smooth $K_{L_V}$-representation and the equality preserves topology.  From Dixmier-Malliavin theorem, $\CO(K_L)$ is dense in $C_c^{\infty}(K_L)$ and $\sigma_0^{\prime,\alg}$ is dense in $\sigma_0^{\prime}$, then by comparing (\ref{equ: Lie algebra side})(\ref{equ: reduce to compact}), we have 
\[
\sigma_{0}^{\p,\alg}\rtimes \pi_{V_0'}^{\alg}\text{ is dense in }
\sigma_{0}^{\p}\rtimes \pi_{V_0'}
\]
with respect the topology of $C_c^{\infty}(K_L)\wh{\otimes} \sigma_0^{\prime}$, this completes the proof.    
\end{proof}

Now we complete the proof for Theorem \ref{thm: vanishing in appendix}(1). From Lemma \ref{lem: dominance}(2), every irreducible subquotient $T_V$ of $\pi_V$ is the Langlands quotient of certain $\pi_V'$ satisfying the condition in Theorem \ref{thm: vanishing in appendix}(1) then
\[\Hom_{\RG_V}(\pi_V',T_V)=0\]
for every irreducible subquotient of $\pi_V$. Hence,
\[\Hom_{\RG_V}(\pi_V',\pi_V)=0.\]
So we complete the proof for Theorem \ref{thm: vanishing in appendix}.
\bibliographystyle{alpha} 
\bibliography{cheng}

\begin{thebibliography}{AGRS10}

\bibitem[AB98]{adams1998genuine}
Jeffrey Adams and Dan Barbasch.
\newblock Genuine representations of the metaplectic group.
\newblock {\em Compositio Mathematica}, 113(1):23--66, 1998.

\bibitem[AG08]{aizenbud2008schwartz}
Avraham Aizenbud and Dmitry Gourevitch.
\newblock Schwartz functions on nash manifolds.
\newblock {\em International Mathematics Research Notices}, 2008, 2008.

\bibitem[AGRS10]{aizenbud2010multiplicity}
Avraham Aizenbud, Dmitry Gourevitch, Stephen Rallis, and G{\'e}rard Schiffmann.
\newblock Multiplicity one theorems.
\newblock {\em Annals of Mathematics}, pages 1407--1434, 2010.

\bibitem[Ato18]{atobe2018local}
Hiraku Atobe.
\newblock The local theta correspondence and the local {Gan--Gross--Prasad}
  conjecture for the symplectic-metaplectic case.
\newblock {\em Mathematische Annalen}, 371(1):225--295, 2018.

\bibitem[Ber72]{bernstein1972analytic}
Joseph Bernstein.
\newblock The analytic continuation of generalized functions with respect to a
  parameter.
\newblock {\em Funktsional'nyi Analiz i ego Prilozheniya}, 6(4):26--40, 1972.

\bibitem[BK14]{bernstein2014smooth}
Joseph Bernstein and Bernhard Kr{\"o}tz.
\newblock Smooth {Fr{\'e}chet} globalizations of {Harish-Chandra} modules.
\newblock {\em Israel Journal of Mathematics}, 199(1):45--111, 2014.

\bibitem[BP14]{beuzart2014expression}
Rapha{\"e}l Beuzart-Plessis.
\newblock Expression d'un facteur epsilon de paire par une formule
  int{\'e}grale.
\newblock {\em Canadian Journal of Mathematics}, 66(5):993--1049, 2014.

\bibitem[BP16]{beuzart2016conjecture}
Rapha{\"e}l Beuzart-Plessis.
\newblock La conjecture locale de {Gross-Prasad} pour les repr{\'e}sentations
  temp{\'e}r{\'e}es des groupes unitaires, 2016.

\bibitem[BP19]{beuzart2019local}
Rapha{\"e}l Beuzart-Plessis.
\newblock A local trace formula for the {Gan-Gross-Prasad} conjecture for
  unitary groups: the archimedean case.
\newblock {\em Asterisque}, 2019.

\bibitem[BW00]{borel2000continuous}
Armand Borel and Nolan~R Wallach.
\newblock {\em Continuous cohomology, discrete subgroups, and representations
  of reductive groups}.
\newblock Number~67. American Mathematical Soc., 2000.

\bibitem[BZ77]{bernstein1977induced}
IN~Bernstein and AV~Zelevinsky.
\newblock Induced representations of reductive $\mathfrak{p}$-adic groups. i.
\newblock In {\em Annales scientifiques de l'{\'E}cole normale sup{\'e}rieure},
  volume~10, pages 441--472, 1977.

\bibitem[Cas89]{casselman1989canonical}
William Casselman.
\newblock Canonical extensions of {Harish-Chandra} modules to representations
  of g.
\newblock {\em Canadian Journal of Mathematics}, 41(3):385--438, 1989.

\bibitem[CCZ]{chen2022FJ}
Cheng Chen, Rui Chen, and Jiangliang Zou.
\newblock {Fourier-Jacobi} on reals.
\newblock {\em In preparation}.

\bibitem[Che21]{chen2021local}
Cheng Chen.
\newblock The local {Gross-Prasad} conjecture over archimedean local fields.
\newblock {\em arXiv preprint arXiv:2102.11404}, 2021.

\bibitem[CHM00]{casselman2000bruhat}
William Casselman, Henryk Hecht, and Dragan Milicic.
\newblock Bruhat filtrations and {Whittaker} vectors for real groups.
\newblock In {\em Proceedings of Symposia in Pure Mathematics}, volume~68,
  pages 151--190. Providence, RI; American Mathematical Society; 1998, 2000.

\bibitem[CL22]{chen2022local}
Cheng Chen and Zhilin Luo.
\newblock The local {Gross-Prasad} conjecture over $\mathbb{R}$: Epsilon
  dichotomy.
\newblock {\em arXiv preprint arXiv:2204.01212}, 2022.

\bibitem[CS20]{chen2020schwartz}
Yangyang Chen and Binyong Sun.
\newblock Schwartz homologies of representations of almost linear {Nash}
  groups.
\newblock {\em Journal of Functional Analysis}, page 108817, 2020.

\bibitem[DC91]{du1991representations}
Fokko Du~Cloux.
\newblock Sur les repr{\'e}sentations diff{\'e}rentiables des groupes de {Lie}
  alg{\'e}briques.
\newblock In {\em Annales scientifiques de l'Ecole normale sup{\'e}rieure},
  volume~24, pages 257--318, 1991.

\bibitem[Din21]{din2021second}
Alexander~Yom Din.
\newblock Second adjointness for tempered admissible representations of a real
  group.
\newblock {\em Israel Journal of Mathematics}, 244(1):215--244, 2021.

\bibitem[GGP12]{gan2012symplectic}
Wee~Teck Gan, Benedict~H Gross, and Dipendra Prasad.
\newblock Symplectic local root numbers, central critical $l$-values, and
  restriction problems in the representation theory of classical groups.
\newblock {\em Ast{\'e}risque}, 346(1):1--109, 2012.

\bibitem[GGS17]{gomez2017generalized}
Raul Gomez, Dmitry Gourevitch, and Siddhartha Sahi.
\newblock Generalized and degenerate whittaker models.
\newblock {\em Compositio Mathematica}, 153(2):223--256, 2017.

\bibitem[GI16]{gan2016gross}
Wee~Teck Gan and Atsushi Ichino.
\newblock The {Gross--Prasad} conjecture and local theta correspondence.
\newblock {\em Inventiones mathematicae}, 206(3):705--799, 2016.

\bibitem[GKT]{gan2023theta}
Wee~Teck Gan, Stephen Kudla, and Shuichiro Takeda.
\newblock {\em The Local Theta Correspondence (prelimiary version),
  https://sites.google.com/view/grothendieck-jr/theta-book}.

\bibitem[GP92]{gross1992decomposition}
Benedict~H Gross and Dipendra Prasad.
\newblock On the decomposition of a representation of {$SO_n$} when restricted
  to {$SO_{n-1}$}.
\newblock {\em Canadian Journal of Mathematics}, 44(5):974--1002, 1992.

\bibitem[GP94]{gross1994irreducible}
Benedict~H Gross and Dipendra Prasad.
\newblock On irreducible representations of {$SO(2n+1)\times SO(2m)$}.
\newblock {\em Canadian Journal of Mathematics}, 46(5):930--950, 1994.

\bibitem[GRS11]{ginzburg2011descent}
David Ginzburg, Stephen Rallis, and David Soudry.
\newblock {\em The descent map from automorphic representations of GL (n) to
  classical groups}.
\newblock World Scientific, 2011.

\bibitem[GS12]{gan2012representations}
Wee~Teck Gan and Gordan Savin.
\newblock Representations of metaplectic groups i: epsilon dichotomy and local
  langlands correspondence.
\newblock {\em Compositio Mathematica}, 148(6):1655--1694, 2012.

\bibitem[GSS19]{gourevitch2019analytic}
Dmitry Gourevitch, Siddhartha Sahi, and Eitan Sayag.
\newblock Analytic continuation of equivariant distributions.
\newblock {\em International Mathematics Research Notices},
  2019(23):7160--7192, 2019.

\bibitem[He17]{he2017gan}
Hongyu He.
\newblock On the gan--gross--prasad conjecture for u (p, q).
\newblock {\em Inventiones mathematicae}, 209:837--884, 2017.

\bibitem[IR78]{igusa1978lectures}
Jun-ichi Igusa and S~Raghavan.
\newblock {\em Lectures on forms of higher degree}, volume~59.
\newblock Springer Berlin-Heidelberg-New York, 1978.

\bibitem[Jac09]{jacquet2009archimedean}
Herv{\'e} Jacquet.
\newblock Archimedean {Rankin-Selberg} integrals.
\newblock {\em Contemporary Mathematics}, 14:57, 2009.

\bibitem[JPSS79]{jacquet1979automorphic}
Herv{\'e} Jacquet, Ilja~Iosifovitch Piatetski-Shapiro, and Joseph Shalika.
\newblock Automorphic forms on gl (3) i.
\newblock {\em Annals of Mathematics}, 109(1):169--212, 1979.

\bibitem[JS81]{jacquet1981euler}
Herv{\'e} Jacquet and Joseph~A Shalika.
\newblock On euler products and the classification of automorphic
  representations i.
\newblock {\em American Journal of Mathematics}, 103(3):499--558, 1981.

\bibitem[JS90]{jacquet1990exterior}
Herv{\'e} Jacquet and Joseph Shalika.
\newblock Exterior square l-functions.
\newblock {\em Automorphic forms, Shimura varieties, and L-functions},
  2:143--226, 1990.

\bibitem[JSZ10]{jiang2010uniqueness}
Dihua Jiang, Binyong Sun, and Chen-Bo Zhu.
\newblock Uniqueness of {Bessel} models: the archimedean case.
\newblock {\em Geometric and Functional Analysis}, 20(3):690--709, 2010.

\bibitem[KZ82]{knapp1982classification}
Anthony~W Knapp and Gregg~J Zuckerman.
\newblock Classification of irreducible tempered representations of semisimple
  groups.
\newblock {\em Annals of Mathematics}, pages 389--455, 1982.

\bibitem[LS13]{liu2013uniqueness}
Yifeng Liu and Binyong Sun.
\newblock Uniqueness of {Fourier}--{Jacobi} models: the archimedean case.
\newblock {\em Journal of Functional Analysis}, 265(12):3325--3344, 2013.

\bibitem[Luo20]{luo2020local}
Zhilin Luo.
\newblock A local trace formula for the local {Gan-Gross-Prasad} conjecture for
  special orthogonal groups.
\newblock {\em arXiv preprint arXiv:2009.13947}, 2020.

\bibitem[Luo21]{luothesis}
Zhilin Luo.
\newblock A local trace formula for the local gross-prasad conjecture for
  special orthogonal groups.
\newblock {\em thesis}, 2021.

\bibitem[MW12]{moeglin2012conjecture}
Colette M{\oe}glin and Jean-Loup Waldspurger.
\newblock La conjecture locale de {Gross-Prasad} pour les groupes sp\'eciaux
  orthogonaux: le cas g\'en\'eral.
\newblock {\em Ast{\'e}risque}, 347:167--216, 2012.

\bibitem[Pra17]{prasad2017reducible}
Dipendra Prasad.
\newblock Reducible principal series representations, and langlands parameters
  for real groups.
\newblock {\em arXiv preprint arXiv:1705.01445}, 2017.

\bibitem[Sou93]{soudry1993rankin}
David Soudry.
\newblock {\em {Rankin-Selberg} Convolutions for $\mathrm{SO}_{2l+
  1}\times\mathrm{GL}_n$: Local Theory}, volume 500.
\newblock American Mathematical Soc., 1993.

\bibitem[Sun12a]{sun2012multiplicitynon}
Binyong Sun.
\newblock Multiplicity one theorems for fourier-jacobi models.
\newblock {\em American Journal of Mathematics}, 134(6):1655--1678, 2012.

\bibitem[Sun12b]{sun2012representations}
BinYong Sun.
\newblock On representations of real {Jacobi} groups.
\newblock {\em Science China Mathematics}, 55(3):541--555, 2012.

\bibitem[Sun15]{sun2015almost}
Binyong Sun.
\newblock Almost linear {Nash} groups.
\newblock {\em Chinese Annals of Mathematics, Series B}, 36(3):355--400, 2015.

\bibitem[SV77]{speh1977reducibility}
Birgit Speh and David Vogan.
\newblock A reducibility criterion for generalized principal series.
\newblock {\em Proceedings of the National Academy of Sciences of the United
  States of America}, 74(12):5252, 1977.

\bibitem[SV80]{speh1980reducibility}
Birgit Speh and David~A Vogan.
\newblock Reducibility of generalized principal series representations.
\newblock {\em Acta Mathematica}, 145:227--299, 1980.

\bibitem[SZ12]{sun2012multiplicity}
Binyong Sun and Chen-Bo Zhu.
\newblock Multiplicity one theorems: the archimedean case.
\newblock {\em Annals of Mathematics}, pages 23--44, 2012.

\bibitem[Tre16]{treves2016topological}
Fran{\c{c}}ois Treves.
\newblock {\em Topological Vector Spaces, Distributions and Kernels: Pure and
  Applied Mathematics, Vol. 25}, volume~25.
\newblock Elsevier, 2016.

\bibitem[Vog78]{vogan1978gelfand}
David~A Vogan.
\newblock Gelfand-{Kirillov} dimension for {Harish-Chandra} modules.
\newblock {\em Inventiones mathematicae}, 48(1):75--98, 1978.

\bibitem[Vog93]{vogan1993local}
David~A Vogan.
\newblock The local langlands conjecture.
\newblock {\em Contemporary Mathematics}, 145:305--305, 1993.

\bibitem[Wal88]{wallach1988real}
Nolan~R Wallach.
\newblock {\em Real reductive groups {I}}.
\newblock Academic press, 1988.

\bibitem[Wal10]{waldspurger2010formule}
J-L Waldspurger.
\newblock Une formule int{\'e}grale reli{\'e}e {\`a} la conjecture locale de
  {Gross--Prasad}.
\newblock {\em Compositio Mathematica}, 146(5):1180--1290, 2010.

\bibitem[Wal12a]{waldspurger2012calcul}
Jean-Loup Waldspurger.
\newblock Calcul d’une valeur d’un facteur $\varepsilon$ par une formule
  int{\'e}grale par.
\newblock {\em Ast{\'e}risque}, 347:1--102, 2012.

\bibitem[Wal12b]{waldspurger2012conjecture}
Jean-Loup Waldspurger.
\newblock La conjecture locale de {Gross-Prasad} pour les repr\'esentations
  temp\'er\'ees des groupes sp\'eciaux orthogonaux.
\newblock {\em Ast\' erisque}, No. 347:103--165, 2012.

\bibitem[Wal12c]{waldspurger2012formule}
Jean-Loup Waldspurger.
\newblock Une formule int{\'e}grale reli{\'e}e {\`a} la conjecture locale de
  gross-prasad, 2e partie: extension aux repr{\'e}sentations temp{\'e}r{\'e}es.
\newblock {\em Ast{\'e}risque}, 346:171--312, 2012.

\bibitem[Wal12d]{waldspurger2012variante}
Jean-Loup Waldspurger.
\newblock Une variante d’un r{\'e}sultat de {Aizenbud}, {Gourevitch},
  {Rallis} et {Schiffmann}.
\newblock {\em Ast{\'e}risque}, no. 346:313--318, 2012.

\bibitem[Xue20a]{xue1bessel}
Hang Xue.
\newblock Bessel models for real unitary groups: the tempered case.
\newblock {\em preprint}, 1(5):7, 2020.

\bibitem[Xue20b]{xue2020bessel}
Hang Xue.
\newblock Bessel models for unitary groups and {Schwartz} homology.
\newblock {\em preprint}, 2020.

\bibitem[Xue22]{xuefourier}
Hang Xue.
\newblock {Fourier-Jacobi} models for real unitary groups.
\newblock 2022.

\end{thebibliography}
\end{document}